\documentclass[reqno, 12pt]{amsart}

\usepackage{mathtools}
\usepackage[T1]{fontenc}
\usepackage{amsmath, amsthm, amssymb}
\usepackage[ansinew]{inputenc}
\usepackage{xcolor}
\usepackage{bbm}
\usepackage{verbatim}
\usepackage{mathrsfs}
\usepackage{cite}
\usepackage{xparse}
\usepackage{fullpage}
\usepackage{pdflscape}
\usepackage{stmaryrd}
\usepackage{multirow}
\usepackage{lscape}
\usepackage{dsfont}
\usepackage{amsfonts}
\usepackage{scrextend} 
\usepackage{array} 
\usepackage{subfig}
\usepackage{yfonts}
\usepackage{enumitem}

\usepackage{slashed}

\usepackage[colorlinks=true,allcolors=red]{hyperref}
\usepackage{graphicx} 
\usepackage{accents}

\newcommand{\Hquad}{\hspace{0.5em}} 
 
\DeclareSymbolFont{sfletters}{OML}{cmbrm}{m}{it}
\DeclareMathSymbol{\somega}{\mathord}{sfletters}{"21}
\DeclareSymbolFont{Letters} {U}{zeur}{m}{n} 
\DeclareMathSymbol{\ssomega}{\mathalpha}{Letters}{"21} 
\DeclareSymbolFont{Letters} {U}{zeur}{m}{n} 
\DeclareMathSymbol{\pPhi}{\mathalpha}{Letters}{"1E}

\makeatletter
\newcommand*\bigcdot{\mathpalette\bigcdot@{.5}}
\newcommand*\bigcdot@[2]{\mathbin{\vcenter{\hbox{\scalebox{#2}{$\m@th#1\bullet$}}}}}
\makeatother

\newtheorem{theorem}{Theorem}[section]
\newtheorem{lemma}[theorem]{Lemma}
\newtheorem{prop}[theorem]{Proposition}

\newtheorem{assumpt}[theorem]{Assumptions}

\theoremstyle{definition}

\theoremstyle{remark}
\newtheorem{remark}[theorem]{Remark}

\allowdisplaybreaks[4]

\numberwithin{equation}{section}

\usepackage{scalerel,stackengine}
\stackMath
\newcommand\reallywidehat[1]{
\savestack{\tmpbox}{\stretchto{
  \scaleto{
    \scalerel*[\widthof{\ensuremath{#1}}]{\kern.1pt\mathchar"0362\kern.1pt}
    {\rule{0ex}{\textheight}}
  }{\textheight}
}{2.4ex}}
\stackon[-6.9pt]{#1}{\tmpbox}
}
\makeatletter
\newcommand*{\rom}[1]{\expandafter\@slowromancap\romannumeral #1@}
\makeatother

\begin{document}

\title{Time periodic solutions and Nekhoroshev stability to non-linear massive Klein-Gordon equations in Anti-de Sitter}

\author{Athanasios Chatzikaleas}

\address{Westf\"alische Wilhelms-Universit\"at M\"unster, Mathematical Institute, Einsteinstrasse 62, 48149 M\"unster, Germany}
\email{achatzik@uni-muenster.de}

\author{Jacques Smulevici}

\address{Laboratory Jacques-Louis Lions (LJLL), Sorbonne Universit\'e, 4 place Jussieu, 75252 Paris, France}
\email{jacques.smulevici@ljll.math.upmc.fr}

\dedicatory{}

\begin{abstract}
 We prove the existence of time-periodic solutions to non-linear massive Klein-Gordon equations in Anti-de Sitter as well as their orbital stability over exponentially long times for certain values of the mass corresponding to completely resonant spectrum. We analyse the resonant system in the Fourier space by relying in particular on Zeilberger's algorithm which allows for a systematic way to derive recurrence formulae for the Fourier coefficients. We also show that the derivation and analysis of the Fourier systems easily extends to other semi-linear wave equations such as the co-rotational wave map equation. 
\end{abstract}

\maketitle
\tableofcontents 
\addtocontents{toc}{\protect\setcounter{tocdepth}{1}} 
 
\noindent

\section{Introduction}
This work concerns the construction of time-periodic solutions to non-linear massive Klein-Gordon equations in the Anti-de Sitter spacetime as well as their  orbital stability over exponentially long times, also called \emph{Nekhoroshev stability}. Our main motivation is to develop heuristics, intuition and analytical methods for the rigorous study of small data dynamics to non-linear wave equations, including the Einstein equations, in the vicinity of the Anti-de Sitter spacetime, a general problem whose context we briefly recall here. 

\subsection{(In-)stability of AdS}
The Anti-de Sitter (AdS) spacetime is the unique (up to quotients by isometry subgroups) maximally symmetric solution to the Einstein equations with a negative cosmological constant. In \cite{DafermosHolzegel,DafermosTalk}, Dafermos-Holzegel conjectured the instability of the AdS solution to the Einstein equations with reflective boundary conditions at the conformal infinity. However, the first result concerning the conjectured instability was the numerical study of Bizo\'{n}-Rostworowski \cite{PhysRevLett.107.031102} who studied small perturbations of AdS for the Einstein-massless-scalar field equations with spherical symmetry. On top of their numerical analysis, they also proposed resonant mode mixing as a mechanism responsible for the instability giving rise to transfer of energy from low to high frequencies. In the setting of the Einstein equations, this energy transfer should eventually lead to the formation of a small black hole. This important work then triggered a substantial amount of numerical and heuristic studies, such as the analysis of the resonant interactions outside from spherical symmetry in \cite{MR2978943}, or further numerical analysis of the instability in \cite{PhysRevLett.115.081103, PhysRevLett.114.071102, PhysRevLett.119.191103, PhysRevLett.111.041102, MR4016456,  MR3708611, MR3205859, MR3491620, MR3329398}. \\ \\
A different approach was developed by Moschidis \cite{MR4150259, MR3828498, MoschidisNew, 2018arXiv181204268M, 2017arXiv170408685M}, who rigorously proved the instability of AdS for the Einstein--massless Vlasov system and Einstein--null dust systems based on a physical space mechanism. In addition, Moschidis has further announced similar rigorous proofs of the AdS instability conjecture in the setting of the spherically symmetric Einstein-scalar field system \cite{gm:esf}. \\ \\
Most relevant for the present work, Rostworowski-Maliborski \cite{PhysRevLett.111.051102} enhanced the instability conjecture by numerically constructing time-periodic solutions to the Einstein-massless-scalar field equations with spherical symmetry bifurcating from initial data dominated by 1-modes. This indicates that, although AdS may be unstable for generic perturbations, there should exist special initial data that evolve into time-periodic solutions staying close to AdS for all times. Moreover, their analysis also suggests that these solutions should in fact enjoy better stability properties than the AdS spacetime itself. \\ \\
Following this work, a large class of oscillating solutions that extends the 1-mode periodic solutions of Bizo\'{n}-Rostoworowski-Maliborksi \cite{PhysRevLett.107.031102, PhysRevLett.111.051102} was numerically determined in \cite{PhysRevLett.113.071601} and their stability was analyzed in \cite{PhysRevD.92.084001}.

\subsection{Time-periodic solutions for the Einstein-Klein-Gordon} \label{SectionTimePeriodicRostworowskiMaliborski}
Following \cite{PhysRevLett.107.031102}, in the Fourier analysis of the spherically-symmetric Einstein-Klein-Gordon system, 
 \begin{align*} 
 		\begin{dcases}
 		\text{Ric}  -\frac{1}{2}\text{R} \text{g}    +  \Lambda  \text{g} = 8 \pi \left ( d  \phi \otimes d \phi - \frac{1}{2} |\partial \phi|_{\text{g}}^2    \text{g} \right ), \quad \Lambda<0,  \\
 			 \Box_{\text{g}}   \phi = 0,
 	\end{dcases}
 	\end{align*} 
Rostworowski-Maliborski \cite{PhysRevLett.111.051102} expanded the dynamical variables in terms of the eigenfunctions to the associated linearized operator. The mode couplings are then quantified via oscillatory integrals, the Fourier coefficients of the system, which are essentially integrals of products of the eigenfunctions, see for example \cite{2020arXiv200411049C,PhysRevLett.107.031102,PhysRevLett.111.051102}.   Rostworowski--Maliborski \cite{PhysRevLett.111.051102} numerically solved the spherically symmetric Einstein-Klein-Gordon system with negative cosmological constant and argued the existence of time-periodic solutions bifurcating from initial data dominated by 1-modes. Their method is based on a perturbative series expansion, where the initial data $\left( \phi(0,\cdot), \partial_t \phi(0,\cdot) \right) $ for the wave function are close to a fixed 1-mode. After projecting onto the eigenbasis, one obtains an infinite system of harmonic oscillators for which secular terms\footnote{These are terms of the form $\tau^{\alpha} \sin( \beta \tau)$ and $\tau^{\alpha} \cos( \beta \tau)$ for some $\alpha>0$ and $\beta \in \mathbb{R}$.} may appear and spoil the periodicity of the solutions. If these terms are not removed, they should eventually become responsible for an instability and possible breakdown of solutions, see \cite{PhysRevLett.107.031102}. Thus, in \cite{PhysRevLett.111.051102}, the initial data $\left( \phi(0,\cdot), \partial_t \phi(0,\cdot) \right) $ are prescribed so that the solutions to the infinite system of harmonic oscillators described above are free of secular terms. On the one hand, Rostworowski--Maliborski \cite{PhysRevLett.111.051102} needed to numerically compute the Fourier constants of the resonant system as well verify (a priori infinitely many) conditions for the Fourier coefficients which are necessary for the removal of the secular terms. On the other hand, their method does not allow a priori to understand in which sense the perturbative series expansion converges. Several research efforts have been focused on either explicitly computing the Fourier coefficients or deriving ultraviolet asymptotics \cite{2020arXiv200411049C, MR3435559,MR3435560,MR3512262,MR3270350,PhysRevLett.113.071601,PhysRevD.92.084001,zbMATH07227949}.  \\ \\ 
Having a robust method that indicates how one can treat the Fourier coefficients would not only be beneficial for numerical studies but also for understanding existence, stability and convergence of the solutions. In fact, both the existence and the exponential non-linear stability of the time-periodic solutions we construct here will depend on two, so called, non-degeneracy conditions (Sections \ref{SectionExistenceofTimePeriodicSolutions} and \ref{SectionStabilitySolutionstotheLinearWaveEquation}). These are systems of infinitely many conditions for the Fourier coefficients similar to the ones described above and hence our proof requires extensive amount of information for the Fourier coefficients.  \\ \\
In this paper, we propose a novel analysis for the Fourier coefficients (Section \ref{SectionFourierCoefficients}) that leads for instance to linear recurrence relations. This approach seems well-adapted for deriving effective formulas and asymptotic behaviour, especially in the case where explicit closed formulas for these integrals are a priori too complicated to handle. 
\subsection{Time-periodic solutions for toy models}
We illustrate the method above by rigorously constructing time-periodic solutions for toy models in AdS. This kind of study was initiated in \cite{SakisJacques}, where we considered the conformal cubic wave and the Yang-Mills equations with various ansatz and symmetries. The analysis of \cite{SakisJacques} relied on KAM tools for PDEs, especially those developed in the work of Bambusi-Paleari \cite{MR1819863}. These tools are all Fourier based\footnote{See for instance \cite{MR2021909} for an alternative method relying in particular on the mountain pass theorem.}  which is the main reason we need to compute and analyse the Fourier coefficients that appear in the resonant system. Although in other settings this is a relatively easy task, in the present work this is an important technical aspect of the problem due to the fact that the eigenfunctions are Jacobi polynomials with various extra weights in the integrals coming form the background geometry or the non-linearities. \\ \\
To resolve this issue, we propose a novel systematic way to compute the Fourier coefficients based in particular on Zeilberger's algorithm \cite{MR1379802, MR1644447}. Concerning Zeilberger's algorithm and its applications, we refer the reader to \cite{ANDREWS1993147, MR4490715, MR1448687, MR1395420, MR1766263, MR3390498, MR4355417, MR1973952} 
and the references within. Moreover, motivated by \cite{PhysRevLett.111.051102}, where the time-periodic solutions near AdS for the Einstein equations (or at least the Einstein-scalar-field system in spherical symmetry) are also expected to have good stability properties, we  also study the orbital stability, over exponentially long times, of the time-periodic solutions we construct. In \cite{SakisJacques}, we relied on KAM tools developed in \cite{MR1819863}, while here the orbital stability is based on a work of Bambusi-Nekhoroskev \cite{BAMBUSI199873} where the stability proof essentially boils down to a coercivity estimate for the second differential of a modified Hamiltonian. \\ \\
First, we consider the spherically symmetric Klein-Gordon equation on the fixed $(1+3)$-dimensional AdS space,
\begin{align*}
	-\partial_{t}^2 \psi (t,x)+ \frac{1}{\tan^2(x) } \partial_{x} \left(\tan^2 (x) \partial_{x} \psi(t,x) \right)= \frac{1}{\cos^2(x)} \left[m^2 \psi(t,x)+c \psi^3(t,x)   \right], 
\end{align*} 
for which we will prove both existence of time-periodic solutions and their orbital  stability over exponentially long times. Second, as an illustration of the robustness of the method and especially the analysis of the Fourier system, we also consider the co-rotational wave maps from the fixed $(1+2)$-dimensional AdS space into a $2$-dimensional warped product Riemannian manifold,  
\begin{align*}
	-\partial_{t}^2 \psi (t,x)+ \frac{1}{\tan(x) } \partial_{x} \left(\tan (x) \partial_{x} \psi(t,x) \right)= \frac{1}{\sin^2(x)}  g(\psi(t,x)) g^{\prime}(\psi(t,x)),
\end{align*}
where $g(\psi)= \delta \psi+ \frac{c}{4\delta} \psi^3 + \mathcal{O} (\psi^5)$. 
In both cases, $m,\delta,c$ are constants with $c>0$ 
and we assume a spherically symmetric evolution $\psi: \mathbb{R} \times (0,\pi/2) \rightarrow \mathbb{R}$ subject to Dirichlet boundary conditions at the conformal infinity $\{x=\pi/2\}$. 
We underline the fact that both models we consider can be converted to non-linear wave equations (Section \ref{SectionMainResults}) of the form
\begin{align}\label{WaveEquqtioninFixedAdS}
	\left( \Box_{\text{g}_{\text{AdS}}} -m^2 \right)\phi = \pm  w\phi^3 + \mathcal{O} \left(\phi^5 \right),
\end{align} 
with different values of spatial dimension, mass $m$, and spatial weights $w$. The mass $m \in \mathbb{R} \cup i \mathbb{R}$ is restricted by the usual Breitenlohner-Freedman bound that ensures local well-posedness as well as by the requirement that the spectrum needs to be completely resonant for our purposes\footnote{For instance, in $3+1$ dimensions, this reads $m^2 \ge - (3/2)^2$ and $ 3/2+ \sqrt{ m^2 +(3/2)^2} \in \mathbb{N}$.}. \\ \\
Although the results and methods of this paper do not directly extend to the Einstein-Klein-Gordon system in spherical symmetry, essentially due to the quasi-linearity of the system, the models that we consider still have important common features with the Einstein-Klein-Gordon system in spherical symmetry in correspondence with the numerical evidence of Rostworowski--Maliborski \cite{PhysRevLett.111.051102}. These include the completely resonant linear spectrum and the fact that the eigenfunctions to the linearized operators are all weighted Jacobi polynomials (with parameters depending on the choice of mass and spatial dimension). In particular, we believe that the way we treat the Fourier coefficients and analyse their behaviour in this paper (as well as in our previous work \cite{SakisJacques}) will be useful for future studies focusing on the analogous Fourier coefficients of the Einstein-Klein-Gordon system itself. 
\section{Main results}\label{SectionMainResults}
In this section, we comment on the two models we consider and state our main results. Let $d \geq 2$ be an integer and consider the Anti-de Sitter space $(\mathcal{M}^{1+d}, \text{g}_{\text{AdS}})$ which is the simplest solution to the Einstein equations with negative cosmological constant
\begin{align*}
\text{Ric}(g)=\Lambda \text{g}, \quad \Lambda < 0.	
\end{align*}
For convenience, we fix $\Lambda=-3$ and recall that, in standard compactified coordinates, 
\begin{align*}
	 (t,x,\omega) \in \mathbb{R} \times (0, \pi/2) \times \mathbb{S}^{d-1},
\end{align*}
the $(1+d)$-dimensional AdS metric is given by
\begin{align*}
	\text{g}_{\text{AdS}}(t,x,\omega)= \cos^{-2}(x) \left(-dt^2 +dx^2 +\sin^2(x)d\omega^2  \right) ,  
\end{align*}
where $d\omega^2$ denotes the standard round metric on $\mathbb{S}^{d-1}$. Moreover, the wave operator in these coordinates reads
\begin{align}\label{DefinitionWaveOperator}
	 \Box_{\text{g}_{\text{AdS}}} \psi = \cos^2 (x) \left[- \partial_{t}^2 \psi + \frac{1}{\tan^{d-1}  (x)} \partial_{x} \left( \tan^{d-1} (x) \partial_{x}\psi   \right) +\frac{1}{ \sin ^2  (x)} \slashed{\Delta} \psi \right], 
\end{align}
for any scalar field $\psi:\mathbb{R} \times [0, \pi/2) \times \mathbb{S}^{d-1}  \rightarrow \mathbb{R}$, where $\slashed{\Delta}$ stands for the spherical Laplacian.

\subsection{Non-linear Klein-Gordon fields in AdS}\label{SubsectionModelKG}
Firstly, we consider the massive Klein-Gordon on $(1+3)$-dimensional fixed AdS with a cubic non-linearity,
\begin{align*}
	\left( \Box_{\text{g}_{\text{AdS}}}   -m^2  \right)\psi = c \psi^3 
\end{align*}
in spherical symmetry. Due to \eqref{DefinitionWaveOperator}, this can be written as
\begin{align}\label{ModelKGIntro}
	-\partial_{t}^2 \psi  + \frac{1}{\tan^2(x) } \partial_{x} \left(\tan^2 (x) \partial_{x} \psi \right)= \frac{1}{\cos^2(x)} \left[m^2 \psi +c \psi^3   \right],
\end{align}  
for a scalar field $\psi:\mathbb{R} \times [0,\pi/2)  \rightarrow \mathbb{R}$. The second order ordinary differential equation governing the radial part of the wave equation above has a regular singular point at infinity\footnote{See for instance \cite{MR3089665} for a detailed presentation of well-posedness and boundary conditions for general linear wave equations in AdS.} and hence Frobenius theorem yields the expansion, for a regular solution (at least of the underlying linear equation), 
\begin{align*}
	\psi (t,x,\omega)= \cos^{\delta_{+}}(x)  \left[ \psi^{+}(t,\omega)+ \mathcal{O} \left( \left(\frac{\pi }{2}-x\right)^2 \right) \right] +\cos^{\delta_{-}}(x)  \left[ \psi^{-}(t,\omega)+ \mathcal{O} \left( \left(\frac{\pi }{2}-x\right)^2 \right) \right] ,
\end{align*}
where
\begin{align*}
	\delta_{\pm  } = \frac{3}{2} \pm \sqrt{m^2+ \left( \frac{3}{2} \right)^2}.
\end{align*}
Due to the fact that AdS has a timelike conformal boundary $\mathcal{I}=\{x=\pi/2\}$, in order to determine of evolution of fields on AdS, one has to prescribe, in addition to initial data on the $\{t = 0\}$ hypersurface, suitable boundary conditions at $\mathcal{I}$. Equivalently, one has to place some constraints on the functions $ \psi^{\pm}$ in the expansion above. From now on, we choose $\psi^{-}=0$ that corresponds to Dirichlet boundary condition at $\mathcal{I}$. This is a case of reflecting boundary conditions, where $\mathcal{I}$ acts like a mirror at which perturbations propagating outwards bounce off and return to $\{0 \leq x< \pi/2\}$. Hence, we look for $\epsilon$-size amplitude time-periodic solutions and set
\begin{align*}
	\psi(t,x)= \frac{\epsilon}{\sqrt{c}}\cos^{\delta}(x) \phi(t,x), \quad \delta=\delta_{+},
\end{align*}
so that \eqref{ModelKGIntro} becomes the cubic (sub-critical) non-linear wave equation
\begin{align}\label{MainPDEIntroModelKGA}
  \partial_{t}^2 \phi (t,x)  +\mathsf{L} \phi(t,x)  = -\epsilon^2  \mathsf{W} \phi^3(t,x)  , 
\end{align}
where
\begin{align} \label{MainPDEIntroModelKGB}
     \begin{dcases}
     	~ \mathsf{W} = \cos^{2\delta-2} (x), \\ 
	~\mathsf{L} \phi = -\partial_{x}^2 \phi  -\frac{2-2 \delta  \sin ^2(x)}{\sin (x) \cos (x)} \partial_{x} \phi + \delta^2 \phi .
     \end{dcases}
\end{align}
With respect to an appropriate weighted $L^2 ((0,\pi/2);d\mu)$ space, $\mathsf{L}$ turns out to be self-adjoint (Section \ref{se:sa}) and hence the linear initial value problem obtained by setting $\epsilon=0$ on the right-hand-side of \eqref{MainPDEIntroModelKGA} is well-posed in the weighted $H^1 ((0,\pi/2);d\mu) \times L^2  ((0,\pi/2);d\mu)$ function space. If $\delta \ge 2$, by standard arguments and weighted Sobolev inequality (Lemma \ref{LemmaHardySolobelInequality}), the non-linear wave equation \eqref{MainPDEIntroModelKGA} is also globally well-posed in the same energy space. Finally, note that the Breitenlohner-Freedman bound $m^2 \geq -(3/2)^2$ is given in terms of the conformal mass $\delta=\delta_{+}$ by
\begin{align}\label{BreitenlohnerFreedmanBoundDeltaModelKG}
	\delta \geq \frac{3}{2}
\end{align}
 and that $\delta = 2$ corresponds to the conformal wave equation.

\subsection{Co-rotational wave maps on fixed AdS} \label{se:crwm}

Secondly, we consider a wave map from the $(1+2)$-dimensional fixed AdS into a $2$-dimensional\footnote{We decided to fix the spatial dimensions to $d=2$ because this choice corresponds to the energy critical case with $\mathbb{R}^{1+2}$ as the domain. The proof of the main result of this paper is based on the coercivity estimate \eqref{Assumption2BambusiNekhoroshev} and in the following we prove its validity only for $d=2$. We note that in higher dimensions, say $d=3$, the analogous mode solutions are saddle points of that energy functional. } warped product Riemannian manifold,
\begin{align*}
	N^2=(0,\infty ) \times  \mathbb{S}^{1} , \quad h(u,\Omega)= du^2 +g^2(u) d \Omega^2,
\end{align*}
 where $(u,\Omega)$ are the natural polar coordinates system on $N^2$ and $d\Omega^2$ stands for the standard round metric on $\mathbb{S}^1$. Here, we assume that 
 \begin{align}\label{Assumptionforg1}
 	g \in C^{\infty} (\mathbb{R}), \quad g \text{ odd}, \quad g^{\prime}(0 )>0, \quad g>0 \text{ on } (0,\infty).
 \end{align}
 In this setting, a wave map $U: (\mathcal{M}^{1+2},\text{g}_{\text{AdS}}) \longrightarrow (N^2,h)$ and be written as $U(t,x,\omega)=\left( u(t,x,\omega),\Omega(t,x,\omega) \right)$. We restrict our attention to the special subclass of so-called co-rotational maps where $u(t,x,\omega)=\psi(t,x)$ and $ \Omega(t,x,\omega) =\omega$. Under this ansatz, the wave map system for $U$,
 \begin{align*}
 	\Box_{\text{g}_{\text{AdS}}} U^a + \Gamma^{a }_{bc}(U) \partial_{\mu}U^b
 \partial ^{\mu} U^c = 0,
 \end{align*}
where $\Gamma^{a }_{bc}(U)$ denote the Christoffel symbols of the target manifold, reduces to the single semi-linear wave equation 
 \begin{align}\label{WaveMapsEquation}
 	-\partial_{t}^2 \psi + \frac{1}{\tan(x)} \partial_{x}\left(\tan(x) \partial_{x}\psi \right)  &= \frac{1}{\sin^2(x)}  g(\psi) g^{\prime}(\psi).
\end{align} 
Moreover, besides \eqref{Assumptionforg1}, we assume that
\begin{align}\label{TaylorAssumptionong}
	g(\psi) =\delta \psi+ \frac{c}{4\delta} \psi^3 + \mathcal{O} (\psi^5),\quad  \delta=g^{\prime}(0)>0,\quad c>0
\end{align}
and a straight-froward computation shows that the sectional curvatures of the target manifold behave like 
\begin{align*}
	-\frac{g^{\prime \prime }(\psi) }{g(\psi)} &= -  \frac{3c}{2\delta^2} + \mathcal{O}(\psi^2), \quad 
	 \frac{1-(g^{\prime }(\psi))^2 }{g(\psi)}= \frac{1-\delta^2}{\delta} \frac{1}{\psi} +\mathcal{O}(\psi),
\end{align*}
as $\psi \rightarrow 0$, hence the target manifold can be either positively or negatively curved depending on $c>0$ and $\delta >0$ provided that $\psi$ is sufficiently small. We note that the waves maps from AdS to the sphere $\mathbb{S}^{2}$, where $g(u)=\sin(u)$, $\delta=1$, $c=-2/3$, and to the hyperbolic space $\mathbb{H}^{2}$, where $g(u)=\sinh(u)$, $\delta=1$, $c=2/3$, are included in the setting above. According to \eqref{TaylorAssumptionong}, we also have that 
\begin{align}\label{TaylorAssumptionong2}
	g(\psi) g^{\prime}(\psi) &= \delta^2 \psi+   c    \psi^3 + \mathcal{O} (\psi^5),
\end{align}
hence the co-rotational wave maps equation \eqref{WaveMapsEquation} becomes the following Klein-Gordon equation
 \begin{align}\label{ModelWMIntro}
 	-\partial_{t}^2 \psi + \frac{1}{\tan (x)} \partial_{x}\left(\tan (x) \partial_{x}\psi \right)    &=\frac{1}{\sin^2 (x)} \left[  \delta^2  \psi +c   \psi^3 +  \mathcal{O} (\psi^5)  \right],
\end{align} 
for a scalar field $\psi:\mathbb{R} \times [0,\pi/2)  \rightarrow \mathbb{R}$ satisfying the boundary condition $\psi(t,0)=0$ for all times. Similarly to Section \ref{SubsectionModelKG}, we look for $\epsilon$-side amplitude time-periodic solutions with Dirichlet boundary conditions at the conformal boundary $\mathcal{I}=\{x=\pi/2\}$, set
\begin{align*}
	\psi(t,x)= \frac{\epsilon}{\sqrt{c}}\sin^{\delta}(x) \cos^2(x) \phi(t,x), \quad \delta=g^{\prime}(0)>0,
\end{align*}
and \eqref{ModelWMIntro} becomes the non-linear wave equation
\begin{align}\label{MainPDEIntroModelWMA}
  \partial_{t}^2 \phi (t,x)  + \mathfrak{L} \phi(t,x)  = -\epsilon^2  \mathfrak{W}  \phi^3(t,x) +  \epsilon^4 \mathfrak{E} (x,\phi(t,x) )  , 
\end{align}
where
\begin{align} \label{MainPDEIntroModelWMB}
     \begin{dcases}
     	~ \mathfrak{W}   = \sin^{2\delta-2}(x) \cos^4 (x), \\ 
	~\mathfrak{L} \phi = -\partial_{x}^2 \phi  -\frac{2 (\delta +2) \cos ^2(x)-3}{\sin (x) \cos (x)}\partial_{x}\phi + (\delta+2)^2 \phi ,\\  
	~\mathfrak{E} (x,\phi)  = \frac{\sqrt{c}}{\epsilon^5} \frac{1}{\sin^{\delta+2}(x)\cos^{2}(x)}  \left(  \delta^2 \psi+   c    \psi^3- g(\psi) g^{\prime}(\psi) \right).
     \end{dcases}
\end{align}
For regularity purposes at the center, we will assume that
\begin{align}\label{BreitenlohnerFreedmanBoundDeltaModelWM}
	\delta \geq 1.
\end{align} 
With respect to an appropriate weighted $L^2 ((0,\pi/2);d\mu)$ space, $\mathfrak{L}$ turns out to be self-adjoint (Section \ref{se:sa}) and hence the linear initial value problem obtained by setting $\epsilon=0$ on the right-hand-side of \eqref{MainPDEIntroModelWMA} is well-posed in the weighted $H^{1} ((0,\pi/2);d\mu) \times L^2  ((0,\pi/2);d\mu)$ function space. By standard arguments and weighted Sobolev inequality (Lemma \ref{LemmaHardySolobelInequality}), the non-linear wave equation \eqref{MainPDEIntroModelWMA} is also locally well-posed in a higher regularity space of type $H^2 \times H^1$. 

 \subsection{Weighted Sobolev spaces}\label{SectionWeightedSobolev}
 In both models, the linearized equations take the form $\partial_t^2 \phi + L \phi = 0$, for some operator $L$ with domain $D(L)$ which is self-adjoint for some weighted $L^2( ( 0, \pi/2); d\mu)$ scalar product, with pure point spectrum $\{\omega_n^2 > 0: n \in \mathbb{N} \}$ and complete eigenbasis $\{e_n :n \in \mathbb{N}  \}$. The $L$ operator will be associated to a positive quadratic form which defines in particular an energy space\footnote{Here, by ``energy space'', we mean the requirement of finite ``non-twisted'' energy imposed by the Dirichlet boundary condition at $x=\pi/2$ for the original problem, cf~\cite{MR3089665}. } $H^1 ( (0, \pi/2); d\mu)$ (Section \ref{SectionLineareigenvalueproblems}). For any $s \geq 2$, we define accordingly the higher regularity Sobolev space $H^s ( (0, \pi/2); d\mu)$ as the set consisting of elements $\phi \in \mathcal{D}(L)$ such that
 \begin{align*}
 	\left( \phi | L^s \phi \right)  := \int_0^{\pi/2} \phi L^s \phi d \mu = \sum_{n \in \mathbb{N}} \omega_n^{2s} |(\phi| e_n)_{}|^2   < + \infty,
 \end{align*} 
 where $L$ is the linear operator appearing in the equation and $(\cdot|\cdot)$ stands for the $L^2(  0, \pi/2; d\mu)$ inner product given by \eqref{DefinitionInnerProduct}. 
 
  \subsection{Notation}
We summarize the notation we will use depending on the model in  Table \ref{NotationTable} below.
\vspace{0.4cm} 
\begin{table}[h!]
\centering
 \begin{tabular}{|c |c|c|c|} 
 \hline
  Model & Klein-Gordon (KG) & Wave Maps (WM) & Both (KG and WM)  \\ [0.5ex] 
 \hline 
Equation & 
\eqref{MainPDEIntroModelKGA}--\eqref{MainPDEIntroModelKGB} & 
\eqref{MainPDEIntroModelWMA}--\eqref{MainPDEIntroModelWMB} & \eqref{WaveEquationBothModels}
     \\ [0.5ex]  \hline 
Font & 
Serif  & 
Fraktur & 
Standard   \\ [0.5ex] \hline 
Linearized operator & 
$\mathsf{L} $ & 
$\mathfrak{L}  $ &
$L $   \\ [0.5ex]  \hline
 Eigenvalues & 
 $\somega_n$ & 
$\ssomega_n $ &
$ \omega_n$    \\ [0.5ex]  \hline 
Eigenfunctions & 
$\mathsf{e}_n $ & 
$\mathfrak{e}_n  $ &
$e_n $  \\ [0.5ex]  \hline 
Linear flow & 
$ \mathsf{\Phi}^t$ & 
$  \pPhi^t  $ &
$ \Phi^t $  \\  [0.5ex] \hline
Fourier coefficients & 
$ \mathsf{C}_{ijkm}$ & 
$ \mathfrak{C}_{ijkm} $ &
$ C_{ijkm}$   \\ [0.5ex]  \hline
 \end{tabular} 
 \vspace{0.4cm}
 \caption{Notation depending on the model.}
 \label{NotationTable}
\end{table} 
\subsection{Statement of the main results}
This work is a continuation of \cite{SakisJacques} where we establish a rigorous proof for the existence of time-periodic solutions to two conformal wave equations on the Einstein cylinder $\mathbb{R} \times \mathbb{S}^3$, the conformal cubic wave equation and the spherically symmetric Yang-Mills. Here, we investigate both the existence and the non-linear stability of time-periodic solutions to spherically symmetric massive Klein-Gordon equations of the form
\begin{align}\label{WaveEquationBothModels}
	\partial_{t}^2 \phi +L \phi = -\epsilon^2 W \phi^3 + \epsilon^4 E(\phi),
\end{align}
for a scalar field $\phi: \mathbb{R} \times [0,\pi/2) \rightarrow \mathbb{R}$, where, using the notation from Table \ref{NotationTable}, $L$, $W$ and $E$ are given by \eqref{MainPDEIntroModelKGA}-\eqref{MainPDEIntroModelKGB} and \eqref{MainPDEIntroModelWMA}-\eqref{MainPDEIntroModelWMB} with $E=0$ in the case of the KG model. To state our main result, we denote by 
\begin{itemize}[leftmargin=*]
	\item $\{\omega_n ^2: n\geq 0\}$ the set of eigenvalues to the linearized operator $L$,
	\item $\{ e_{n} (x):n \geq 0 \}$ the set of eigenfunctions to the linearized operator $L$.
\end{itemize}
Notice that these depend on the conformal mass $\delta$. At this stage, we assume (and later verify) that the set of all eigenfunctions $\{e_m: m\geq 0\}$ to the linearized operator $L$ forms an orthonormal and complete basis for the weighted $L^2((0,\pi/2);d\mu)$, where $d\mu$ is some integration measure depending on the model. Furthermore, for any $s\geq 1$, we denote by $H^s ((0,\pi/2);d\mu) $ the natural $L^2((0,\pi/2);d\mu)$-based Sobolev space associated to the operator $L$ (Section \ref{SectionWeightedSobolev}). As in \cite{MR1819863}, the time-periodic solutions we construct have frequencies $\omega$ that satisfy the following Diophantine condition  
	 \begin{align}\label{DefinitionSetMathcalWalpha}
	 \omega \in	\mathcal{W}_{\alpha}= \left\{ \omega \in \mathbb{R}: \Hquad   |\omega \cdot l -\omega_{j} | \geq \frac{\alpha}{l}, \Hquad \forall (l,j) \in \mathbb{N}^2,\Hquad l\geq 1, \Hquad \omega_{j} \neq l \right\},
	 \end{align}
for some real number $\alpha \in (0,1/3)$. We make the following assumptions concerning the regularity and the choice of the conformal masses $\delta$.
\begin{assumpt}[Assumptions on the conformal mass for the existence of time-periodic solutions]\label{AssumptionsondeltaExistence}
Besides the corresponding lower bounds \eqref{BreitenlohnerFreedmanBoundDeltaModelKG} and \eqref{BreitenlohnerFreedmanBoundDeltaModelWM}, we assume that the conformal masses $\delta$ are, in both models, fixed integers, hence $\delta \in  [2,\infty)\cap \mathbb{N} $ and $\delta \in [1,\infty)\cap \mathbb{N} $ for the KG and WM models respectively.  
\end{assumpt}

\begin{assumpt}[Assumptions on the conformal mass for the non-linear stability of time-periodic solutions over exponentially long times]\label{AssumptionsondeltanonlinearStability}
We assume that  $\delta \in  [2,9]\cap \mathbb{N} $  for the KG model.  
\end{assumpt}

 \begin{remark}[Remarks on the Assumptions \ref{AssumptionsondeltaExistence} and \ref{AssumptionsondeltanonlinearStability}]
The bounds \eqref{BreitenlohnerFreedmanBoundDeltaModelKG} and \eqref{BreitenlohnerFreedmanBoundDeltaModelWM} are necessary for the local well-posedness of the linear equations. On the one hand, for the existence of small amplitude time-periodic solutions, we rely on a method of Bambusi-Paleari  \cite{MR1819863} for which the assumption that $\delta \in \mathbb{N}$ is essential as it implies that the square roots of the eigenvalues to the linearized operators are all integers, namely $\{\omega_{n}:n \geq 0 \} \subseteq \mathbb{N}$. In this case, the linear spectrum is completely resonant and the set of Diophantine frequencies in \eqref{DefinitionSetMathcalWalpha} is non-empty for some $ \alpha \in (0,1/3)$, see  \cite{MR1819863}. We note that, for $\delta \notin \mathbb{N}$, one can modify the proof in \cite{MR1819863} and establish sufficient conditions for the existence of quasi-periodic solutions. On the other hand, for the non-linear stability of the time-periodic solutions over exponentially long times, we rely a method of Bambusi-Nekhoroshev \cite{BAMBUSI199873}, for which the main assumption is the coercivity estimate \eqref{Assumption2BambusiNekhoroshev}. Our analysis reveals that the time-periodic solutions bifurcating from the first 1-mode verifies this estimate only for $\delta$ satisfying Assumption \ref{AssumptionsondeltanonlinearStability}, leaving open the stability or instability for higher values of the mass. 
\end{remark}
For any $\zeta \in H^s ((0,\pi/2);d\mu)$, let 
\begin{align*}
	\Phi ^{t} (\zeta ) = \Phi ^{t} \left(\sum_{m=0}^{\infty} \zeta^{(m)} e_{m} \right) = \sum_{m=0}^{\infty} \zeta^{(m)}\cos(\omega_{m}t)e_{m} 
\end{align*}
be the solution to the initial value problem consisting of the linearized equation in \eqref{WaveEquationBothModels}  coupled to $\zeta$ as initial data and zero initial velocity, i.e.~the solution to the linear problem
\begin{align} \label{LinearizedEquation}
\begin{cases}
	\partial_{t}^2 \phi_{\text{linear}} (t,\cdot) +L \phi_{\text{linear}}(t,\cdot) =0, \quad t \in \mathbb{R}, \\
 \phi_{\text{linear}}(0,\cdot)=\zeta,  \quad \partial_{t} \phi_{\text{linear}} (0,\cdot)= 0.
\end{cases}
\end{align}
The solutions we obtain bifurcate from a suitable rescale of the first eigenfunction $e_{0}$.  Specifically, we consider the initial data $\kappa_0 e_0$, where $\kappa_0$ is defined as
\begin{align}\label{Kappa0}
	\kappa_{0}=
	\begin{dcases}
	2^{\delta } \Gamma \left(\delta -\frac{1}{2}\right)\sqrt{\frac{\Gamma \left(\delta +\frac{1}{2}\right)}{3 \Gamma (\delta ) \Gamma \left(2 \delta -\frac{3}{2}\right)}}, &\text{ for KG}	 \\
	\frac{2 \sqrt{2 \delta  (\delta +1) (2 \delta +1) (2 \delta +3)}}{3 (\delta +1)},& \text{ for WM}	
	\end{dcases}
\end{align}
and the first eigenmode\footnote{Note that in both cases, $e_0(x)$ is just a constant independent of $x$.} is given by
\begin{align}\label{e0}
	e_{0}=
	\begin{dcases}
	\frac{2}{\sqrt[4]{\pi }} \sqrt{\frac{\delta  \Gamma (\delta )}{\Gamma \left(\delta -\frac{1}{2}\right)}}  , &\text{ for KG}	 \\
	  \sqrt{2(\delta +1) (\delta +2)} ,  & \text{ for WM}	
	\end{dcases}
\end{align}
for all $\delta$ satisfying \eqref{BreitenlohnerFreedmanBoundDeltaModelKG} and \eqref{BreitenlohnerFreedmanBoundDeltaModelWM} respectively. Under these assumptions and notations, we prove the following two results. 

\begin{theorem}[Main result 1: Existence of time-periodic solutions for KG and WM]\label{MainTheoremBothTheoremsinOne1} 
	Fix a real number $ \alpha \in (0,1/3)$ as well as $s \ge 2$ and $\delta  $ satisfying \ref{AssumptionsondeltaExistence}. Also, let $\kappa_{0}   $ and $e_0$ be given by \eqref{Kappa0} and \eqref{e0} respectively. Then,  there exists a family  
	 \begin{align*}
	 	\{ \phi_{\epsilon}(t,\cdot):\epsilon \in \mathcal{E}_{\alpha}\} 
	 \end{align*}
	 of time-periodic solutions to the non-linear wave equation \eqref{WaveEquationBothModels} where $\mathcal{E}_{\alpha}$ is an uncountable set that has zero as an accumulation point. In addition, each element $ \phi_{\epsilon}(t,\cdot)$,
\begin{enumerate} [leftmargin=*]
	\item has period $T_{\epsilon}=2\pi / \omega_{\epsilon}  $ with  $ \omega_{\epsilon} \in \mathcal{W}_{\alpha} \cap [1,\infty)$ and $|T_{\epsilon}-2\pi |\lesssim \epsilon^2$, 
	\item belongs to $    H^1 ( \left[0,T_{\epsilon}  \right] ;H^{s}((0,\pi/2);d\mu) )$,
	\item stays, for all times, close to the solution to the linearized equation,
	\begin{align} \label{EstimateinTheorem24}
	\sup_{t\in \mathbb{R}}   \left \|  
	\left(\phi_{\epsilon}(t,\cdot), \partial_{t}\phi_{\epsilon}(t,\cdot)   \right)
	-
	\left( \Phi^{t \omega_{\epsilon} }\left( \epsilon \kappa_{0}   e_{0}   \right ),
	\partial_{t}\Phi^{t \omega_{\epsilon} }\left( \epsilon\kappa_{0}   e_{0}   \right )  \right) \right \|_{(H^1 \times L^2)((0,\pi/2);d\mu)  }   \lesssim \epsilon^2.
	\end{align} 
	\end{enumerate} 
\end{theorem}

\begin{theorem}[Main result 2: Non-linearly stability of time-periodic solutions for the KG over exponentially long times]\label{MainTheoremBothTheoremsinOne2} 
	Fix a real number $ \alpha \in (0,1/3)$, $\delta  $ satisfying the Assumption \ref{AssumptionsondeltanonlinearStability} and let $\kappa_{0}   $ and $e_0$ be given by \eqref{Kappa0} and \eqref{e0} respectively. Then, the family $\{ \phi_{\epsilon}(t,\cdot):\epsilon \in \mathcal{E}_{\alpha}\}$ of time-periodic solutions constructed in Theorem \ref{MainTheoremBothTheoremsinOne1} is non-linearly stable over exponentially long times. Specifically, consider the phase space $\mathcal{P} =H^1((0,\pi/2);d\mu)\times L^2((0,\pi/2);d\mu)$ and all initial data $\boldsymbol{\phi}(0)=(\phi(0,\cdot),\partial_{t}\phi(0,\cdot)) \in \mathcal{P}$	that are close to the initial data of the time-periodic solution $\boldsymbol{\phi}_{\epsilon}(0)=(\phi_{\epsilon}(0,\cdot),\partial_{t}\phi_{\epsilon}(0,\cdot))$, meaning that 
\begin{align*}
	\|\boldsymbol{\phi}(0)-\boldsymbol{\phi}_{\epsilon}(0) \|_{\mathcal{P}} \lesssim \epsilon^2.
\end{align*}     
Then, the solution $ \boldsymbol{\phi}(t)=(\phi(t),\partial_{t} \phi(t))$ to the non-linear KG model given by \eqref{MainPDEIntroModelKGA} bifurcating from $\boldsymbol{\phi}(0)  $ remains close to the time-periodic solutions over exponentially long times, that is  
\begin{align*}
	\sup_{|t| \lesssim  \exp \left( a / \epsilon^2 \right)} d \left(  \boldsymbol{\phi}(t),  \Gamma_\epsilon \right) \lesssim \epsilon^2 ,  
\end{align*}  
for some constant $a>0$, where $d$ denotes the distance function and $\Gamma_\epsilon$ the (closed) orbit of $\boldsymbol{\phi}_{\epsilon}$ both in the phase space $\mathcal{P}$.
\end{theorem} 

\subsection{Remarks}
Provided that $\delta$ satisfies the Assumption \ref{AssumptionsondeltaExistence}, one can follow the work of Bambusi-Paleari \cite{MR1819863} and prove that, given any fixed integer $\gamma \geq 0$, there exists a family of time-periodic solutions to the non-linear wave equation \eqref{WaveEquationBothModels} bifurcating from the rescaled 1-mode $\kappa_{\gamma}e_{\gamma}$ with a suitable constant $\kappa_{\gamma} \neq 0$ and not only for the first eigenmode $e_0$, as stated in Theorem \ref{MainTheoremBothTheoremsinOne1}. However, our analysis in Section \ref{SectionStabilitySolutionstotheLinearWaveEquation} suggests that, among all these time-periodic solutions, the ones that bifurcate from the first 1-mode are non-linearly stable over exponentially long times and all the higher modes $e_\gamma$ with $\gamma \geq 1$ can be shown to be saddle points of the functional appearing on the right hand side of \eqref{Assumption2BambusiNekhoroshev}.

\subsection{Strategy of the proofs}\label{SubsectionStrategyofProof}

In this section, we present the main ingredients of the proofs of Theorems \ref{MainTheoremBothTheoremsinOne1} and \ref{MainTheoremBothTheoremsinOne2}.

\subsubsection{Proof of Theorem \ref{MainTheoremBothTheoremsinOne1}}

Theorem \ref{MainTheoremBothTheoremsinOne1} will follow from the framework of Bambusi-Paleari \cite{MR1819863}. For simplicity in the exposition, we present their theorem here in the functional setting of our applications, thus $H^s$ and $L$ below are as introduced in \ref{SectionWeightedSobolev}.
 
 \begin{theorem}[Bambusi-Paleari \cite{MR1819863}]\label{TheoremBambusiPaleari}
 Consider a non-linear wave equation of the form
\begin{align}\label{WaveEquationTheoremBambusiPaleari}
	\partial_{t}^2 \phi + L \phi =\epsilon^2 f^{(3)}(\phi) +\epsilon^4 f^{(\geq 4)}(\phi),  
\end{align}
where $f^{(3)}(\phi)$ is an homogeneous polynomial of degree three and $f^{(\geq 4)}(\phi)$ stands for an error term that has a zero of order at least four at zero. Fix a real number $ \alpha \in (0,1/3)$ and let $s \geq 1$ be such that 
\begin{itemize}[leftmargin=*]
	\item $f^{(3)}:H^s \rightarrow H^s$ is a bounded homogeneous polynomial of degree three,
	\item $f^{(3)}$ leaves invariant the domain $  H^{s+2}$ of the linearized operator $L$, 
	  \item $f^{(\geq 4)}(\phi)$ is differentiable in $H^s$, its differentiable is a Lipschitz map and satisfies the estimate
\begin{align*}
	\| d f^{(\geq 4)}(\phi_1) - d f^{(\geq 4)}(\phi_2) \|_{H^s} \le C \epsilon^3 \| \phi_1 -\phi_2 \|_{H^s}, 
\end{align*} 
for all $\| \phi_1 \|_{H^s} \leq \epsilon$, $\| \phi_2 \|_{H^s} \leq \epsilon$. 
\end{itemize}  
 Also, define the operator  
	\begin{align*} 
	 \mathcal{M}  \left(\zeta  \right) : =   L \zeta   + \langle f^{(3)}  \rangle(\zeta ), \Hquad
	\langle f^{(3)}  \rangle(\zeta )=  \frac{1}{2\pi } \int_{0}^{2\pi } \Phi ^{-t} \left[f^{(3)} \left(\Phi ^{t} (\zeta  ) \right) 
	\right]  dt , 
\end{align*} 	
for $\zeta \in H^{s+2 }$, and assume that $\xi=\phi(0) \in H^{s+2} $ are initial data such that
\begin{enumerate}[leftmargin=*]
	\item $\xi$ is a zero of the operator $  \mathcal{M}   $,
	 \begin{align}\label{Assumption1BambusiPaleari}
	 \mathcal{M}   \left(\xi \right)=0 ,
     \end{align} 
\item the non-degeneracy condition holds,
\begin{align}\label{Assumption2BambusiPaleari}
	\ker \left(d  \mathcal{M}   \left(\xi  \right)\right)=\{0\}.
\end{align}
\end{enumerate}
Then, there exists a family 
	 \begin{align*}
	 	\{ \phi_{\epsilon}(t,\cdot):\epsilon \in \mathcal{E}_{\alpha}\}  
	 \end{align*}
	 of time-periodic solutions to the non-linear wave equation \eqref{WaveEquationTheoremBambusiPaleari}, where $\mathcal{E}_{\alpha}$ is an uncountable set that has zero as an accumulation point. In addition, each element $ \phi_{\epsilon}(t)$,
\begin{enumerate}[leftmargin=*]
	\item has period $T_{\epsilon}=2\pi / \omega_{\epsilon}  $ with  $ \omega_{\epsilon} \in \mathcal{W}_{\alpha} \cap [1,\infty)$ and $|T_{\epsilon}-2\pi |\lesssim \epsilon^2$,  
	\item belongs in $ H^1 ( \left[0,T_{\epsilon}  \right] ;H^s )$,
	\item stays, for all times, close to the solution to the linearized equation,
	\begin{align*}
			\sup_{t\in \mathbb{R}}\Big[  \left \|  \phi_{\epsilon}(t,\cdot)-\Phi^{t \omega_{\epsilon} }\left( \epsilon\xi  \right ) \right \|_{H^s } + \left \| \partial_{t} \phi_{\epsilon}(t,\cdot)-\partial_{t}\Phi^{t \omega_{\epsilon} }\left( \epsilon\xi  \right ) \right \|_{L^2 }\Big]  		 \lesssim \epsilon^2.
	\end{align*} 
	\end{enumerate}
\end{theorem}  
For both models KG and WM, we will apply Theorem \ref{TheoremBambusiPaleari} with $\xi=\kappa_0 e_0$ an appropriate rescale of the first 1-mode. To do so, we prove the following: 
	\begin{itemize}[leftmargin=*]
		\item Proposition \ref{PropositionAssumption1inBambusiPaleari}, where we establish the validity of assumption \eqref{Assumption1BambusiPaleari} by showing that the rescaled 1-mode initial data $\kappa_0 e_0$ is a zero of the operator $\mathcal{M}$.
		\item Proposition \ref{PropositionAssumption2inBambusiPaleari}, where we establish the validity of the non-degeneracy condition, namely assumption \eqref{Assumption2BambusiPaleari}. We note that this is the hardest part to prove and the part where the novel analysis for the Fourier coefficients will be needed.
		\item Lemmata \ref{LemmaRegularityConditionsKG} and \ref{LemmaRegularityConditionsWM}, where we establish the required regularity assumptions for the non-linearities $f^{(3)}(\phi)$ and $f^{(\geq 4)}(\phi)$ for the KG and WM models respectively.
	\end{itemize}

 \subsubsection{Proof of Theorem \ref{MainTheoremBothTheoremsinOne2}}

 Theorem \ref{MainTheoremBothTheoremsinOne2} will follow from a result of Bambusi-Nekhoroshev \cite{BAMBUSI199873}. In the following lines, we follow \cite{BAMBUSI199873} and state their main result only in the setting of the KG model. Firstly, we rewrite the non-linear wave equation \eqref{MainPDEIntroModelKGA} in the phase space
\begin{align*}
	 \mathbf{p}= (p_1,p_2) \in  \mathcal{P} =H^{1} \left( (0,\pi/2);d\mu\right) \times L^2 \left( (0,\pi/2) ;d\mu\right) ,
\end{align*}
using symbols in bold to denote its elements. In particular, for any $\boldsymbol{\phi} =(\phi,\partial_{t}\phi)\in  \mathcal{P}$, the non-linear wave equation \eqref{MainPDEIntroModelKGA} can be written as 
\begin{align}\label{MainEquationForBambusiNekhoroshevIntro}
	\partial_{t}\boldsymbol{\phi}=
	A\boldsymbol{\phi} +   \epsilon^2  f^{(3)}(\boldsymbol{\phi}),  
\end{align}
where 
\begin{align*}
	A:\mathcal{D} (A) \rightarrow \mathcal{P},\quad  A=
	\begin{pmatrix}
		0 & 1 \\
		-L & 0
	\end{pmatrix}, \quad f^{(3)} (\boldsymbol{\phi})=\begin{pmatrix}
		0 \\
		-W \phi^3 
	\end{pmatrix}.
\end{align*} 
We note that  \eqref{MainEquationForBambusiNekhoroshevIntro} is Hamiltonian and the Hamiltonian function 
\begin{align*}
	\mathcal{H}:\mathcal{P}  \rightarrow \mathbb{R},  \quad \mathcal{H} =\mathcal{H}(\boldsymbol{\phi}), \quad \boldsymbol{\phi}=(\phi_1,\phi_2)  
\end{align*}
is given by 
 \begin{align}\label{DefinitionHamiltonianForProofDetails}
 	\mathcal{H}(\boldsymbol{\phi})  
 	&= h_{\Omega}(\boldsymbol{\phi})  + \epsilon^2  f(\boldsymbol{\phi})   , 
 \end{align}
 where
 \begin{align*}
	h_{\Omega}(\boldsymbol{\zeta})=h_{\Omega}(\zeta_{1},\zeta_{2})=\int_{0}^{\pi/2} \left(  \left(L^{1/2}\zeta_1 \right)^2 + (\zeta_2)^2 \right) dt.
\end{align*}
is the harmonic energy and
 \begin{align}\label{DefinitionLeadingOrderintheHamiltonianIntro}
  	f(\boldsymbol{\phi}) =  \frac{1}{4} \int_{0}^{\pi/2} W \phi_1^4 d\mu   .
 \end{align} 
Furthermore, let  
\begin{align*} 
	\boldsymbol{\Phi}^{t}(\boldsymbol{\zeta}) = e^{A t}\boldsymbol{\zeta}
\end{align*}
be the solution to the initial value problem consisting of the linearized equation in \eqref{MainEquationForBambusiNekhoroshevIntro} coupled to $\boldsymbol{\zeta}$ as initial data, that is
\begin{align} \label{LinearizedEquationForBambusiNekhoroshev}
\begin{cases}
	\partial_{t}\boldsymbol{\phi}_{\text{linear}} (t)=
	A\boldsymbol{\phi}_{\text{linear} } (t), \quad t \in \mathbb{R}, \\
 \boldsymbol{\phi}_{\text{linear}}(0)=\boldsymbol{\zeta}
\end{cases}
\end{align}
and denote by   
\begin{align*}
	\Gamma_{\text{linear}}(\boldsymbol{\zeta}) = \left \{\boldsymbol{\Phi}^{t}(\boldsymbol{\zeta}):t \in \mathbb{R} \right\} 
\end{align*} 
its trajectory in the phase space $\mathcal{P}$. Moreover, we denote by $d(\cdot,\cdot)$ the distance function in $\mathcal{P}$ induced by the energy norm $\|\cdot \|_{\mathcal{P}}$. The  hypersurface of constant harmonic energy $E$ will be denoted by 
\begin{align*}
	\Sigma_E = \left \{\boldsymbol{\zeta}  \in \mathcal{P}: h_{\Omega}(\boldsymbol{\zeta}) = E \right\}.
\end{align*} 
The result of Bambusi-Nekhoroshev \cite{BAMBUSI199873} that we will use is stated below.
\begin{theorem}[Bambusi-Nekhoroshev \cite{BAMBUSI199873}]\label{TheoremBambusiNekhoroshev}
	 Consider a non-linear wave equation of the form
\begin{align}\label{WaveEquationTheoremBambusiNekhoroshev}
	\partial_{t}\boldsymbol{\phi}=
	A\boldsymbol{\phi} +   \epsilon^4  f^{(3)}(\boldsymbol{\phi}),  
\end{align}
written in the phase space
\begin{align*}
	\mathcal{P} =H^{1} \left( (0,\pi/2);d\mu\right) \times L^2 \left( (0,\pi/2) ;d\mu\right) ,
\end{align*}
where $f^{(3)}(\phi)$ is an homogeneous polynomial of degree three. We assume that \eqref{WaveEquationTheoremBambusiNekhoroshev} is Hamiltonian with the Hamiltonian given by \eqref{DefinitionHamiltonianForProofDetails} and that both $f$ and its symplectic differential $J d f $ can be extended to bounded analytic functions in a neighbourhood of $\Sigma_E$ for some $E> 0$. Also, we define the operator
\begin{align*}
	\langle f  \rangle(\boldsymbol{\zeta})=\frac{1}{2\pi }\int_{0}^{2\pi } f \left(\boldsymbol{\Phi}^{t}(\boldsymbol{\zeta}) \right) dt,
\end{align*}
where $f$ is given by \eqref{DefinitionLeadingOrderintheHamiltonianIntro}, and consider its restriction $\langle f \rangle  |_{\Sigma_E}$ to an hypersurface of constant harmonic energy. Moreover, we assume that $\boldsymbol{\xi}=\boldsymbol{\phi}(0) \in\Sigma  \subset  \mathcal{P}$ are initial data such that
\begin{enumerate}[leftmargin=*]
	\item $\boldsymbol{\xi}$ is a critical point for the  function  $\langle f  \rangle |_{\Sigma_E}$,  
	 \begin{align}\label{Assumption1BambusiNekhoroshev}
	\exists \lambda \in \mathbb{R}:\Hquad \nabla_{\boldsymbol{\zeta}} \left[ \lambda  h_{\Omega}(\boldsymbol{\zeta}) +   \langle f \rangle (\boldsymbol{\zeta}) \right]\big|_{\boldsymbol{\zeta}=\boldsymbol{\xi}} =0 ,
     \end{align}  
\item $\xi$ is an extremum point (maximum or minimum) for the function  $\langle f  \rangle |_{\Sigma_E}$, namely the non-degeneracy condition holds,
\begin{align}\label{Assumption2BambusiNekhoroshev}
		 \pm  d^2  ( \langle f  \rangle |_{\Sigma_E } )(\boldsymbol{\xi})(\boldsymbol{X},\boldsymbol{X}) \geq c_{\perp}  \| \boldsymbol{X}\|^2, \quad \forall \boldsymbol{X} \in \left(T_{\boldsymbol{\xi}}\Gamma_{\text{linear}}(\boldsymbol{\xi})  \right)^{\perp}\subset T_{\boldsymbol{\xi} } \Sigma_E ,
	\end{align}
	for some strictly positive constant $c_{\perp}> 0$, with the plus and minus sign when $\boldsymbol{\xi}$ is a minimum or maximum point respectively.
\end{enumerate}
Then, there exists $\epsilon_0>0$ sufficiently small and a numerical constant $a>0$, such that, for all $0 \leq \epsilon<\epsilon_0$ and for all the initial data $ \boldsymbol{\phi}(0) $ with
\begin{align*}
	d \left( \boldsymbol{\phi}(0),  \Gamma_{\text{linear}} (   \epsilon \boldsymbol{\xi} ) \right) \lesssim \epsilon^2,
\end{align*}
the solution $ \boldsymbol{\phi}(t)=(\phi(t),\partial_{t} \phi(t))$ to the non-linear wave equation \eqref{WaveEquationTheoremBambusiNekhoroshev} bifurcating from the  initial data  $\boldsymbol{\phi}(0) $ remains, over exponentially long times,  close to the solution to the linearized equation, that is
\begin{align*}
	 \sup_{|t| \lesssim  \exp \left( a / \epsilon^2 \right)} d \left( \boldsymbol{\phi}(t), \Gamma_{\text{linear}} ( \epsilon \boldsymbol{\xi} ) \right)  \lesssim \epsilon^2.
\end{align*}     
\end{theorem}

 We will apply Theorem \ref{TheoremBambusiNekhoroshev} only\footnote{For the WM model, one can check that, although the coercivity estimate \eqref{Assumption2BambusiNekhoroshev} holds true (only for $d=2$), the analyticity conditions for the corresponding non-linearity fail, so Theorem \ref{TheoremBambusiNekhoroshev} does not apply a priori.} for the KG model \eqref{MainPDEIntroModelKGA} using the same rescale of the first 1-mode $\kappa_0 e_0$, thus obtaining the non-linear stability of the time-periodic solutions as stated in Theorem \ref{MainTheoremBothTheoremsinOne2}.  To do so, we define
  \begin{align*}
	 \boldsymbol{\xi} =(\xi_1,\xi_2), \quad \xi_1=   \kappa_{0}   e_{0}  , \quad \xi_2=0.
	 \end{align*} 
and prove the following: 
	\begin{itemize}[leftmargin=*]
		\item Proposition \ref{LemmaAssumption1BambusiNekhoroshev}, where we establish the validity of assumption \eqref{Assumption1BambusiNekhoroshev} by showing that the rescaled 1-mode initial data $\boldsymbol{\xi}$ are critical points for the  function  $\langle f  \rangle |_{\Sigma_E}$ for some $E>0$.
		\item Proposition \ref{LemmaAssumption2BambusiNekhoroshev}, where we establish the validity of the non-degeneracy condition, namely assumption \eqref{Assumption2BambusiNekhoroshev}. We note that this is the hardest part to prove and the part where the novel analysis for the Fourier coefficients will be needed.
		\item Lemma \ref{LemmaHardySolobelInequality}, where we establish the required regularity assumptions for the non-linearity  $f (\boldsymbol{\phi})$ given in the Hamiltonian  \eqref{DefinitionHamiltonianForProofDetails} and its symplectic differential $J d f (\boldsymbol{\phi})$.
	\end{itemize} 
From these steps, the non-linear stability of the time-periodic solutions follows as stated in Theorem \ref{MainTheoremBothTheoremsinOne2}. Indeed, fix $\delta $ satisfying Assumption \ref{AssumptionsondeltanonlinearStability}, let $\kappa_{0}   $ and $e_0$ be given by \eqref{Kappa0} and \eqref{e0} respectively, and define the rescaled first 1-mode data $\boldsymbol{\xi} =( \kappa_{0}   e_{0},0) \in \mathcal{P} $. In addition, fix a real number $ \alpha \in (0,1/3)$ and let $\{ \phi_{\epsilon}(t,\cdot):\epsilon \in \mathcal{E}_{\alpha}\}$ be the family of time-periodic solutions constructed in Theorem \ref{MainTheoremBothTheoremsinOne1} with some frequency $\omega_{\epsilon} \in \mathcal{W}_{\alpha}$. Consider the phase space $\mathcal{P} =H^1((0,\pi/2);d\mu)\times L^2((0,\pi/2);d\mu)$ and all initial data $\boldsymbol{\phi}(0)=(\phi(0,\cdot),\partial_{t}\phi(0,\cdot)) \in \mathcal{P}$	that are close to the initial data of the time-periodic solution $\boldsymbol{\phi}_{\epsilon}(0)=(\phi_{\epsilon}(0,\cdot),\partial_{t}\phi_{\epsilon}(0,\cdot))$, meaning that 
\begin{align}\label{AssumptionSmallIntialDataforProof}
	\|\boldsymbol{\phi}(0)-\boldsymbol{\phi}_{\epsilon}(0) \|_{\mathcal{P}} \lesssim \epsilon^2.
\end{align}     
Let $\boldsymbol{\phi} (t)=(\phi (t,\cdot),\partial_{t}\phi (t,\cdot))$ be the solution to the non-linear KG equation \eqref{MainPDEIntroModelKGA} emanating from the data $\boldsymbol{\phi}(0)  $ and denote by $\Gamma_{\epsilon}$ the (closed) orbit of $\boldsymbol{\phi}_{\epsilon}$ in the phase space $\mathcal{P}$. Moreover, let  $d$ denote the distance function in $\mathcal{P}$. Now, the main assumption on the initial data of Theorem \ref{TheoremBambusiNekhoroshev} holds true since
\begin{align*}
	d \left( \boldsymbol{\phi}(0),  \Gamma_{\text{linear}} ( \epsilon \boldsymbol{\xi} ) \right)&=
	\inf_{t \in\mathbb{R}  } \left\| \boldsymbol{\phi}(0)- \boldsymbol{\Phi}^{t}  ( \epsilon \boldsymbol{\xi} ) \right\|_{\mathcal{P}} 
	 \leq  \left\| \boldsymbol{\phi}(0)- \boldsymbol{\Phi}^{0}  ( \epsilon \boldsymbol{\xi} ) \right\|_{\mathcal{P}}  \\ 
		&\leq 
		\|\boldsymbol{\phi}(0)-\boldsymbol{\phi}_{\epsilon}(0) \|_{\mathcal{P}}
		+
	  \left\| \boldsymbol{\phi}_{\epsilon}(0)- \boldsymbol{\Phi}^{0}  ( \epsilon \boldsymbol{\xi} ) \right\|_{\mathcal{P}} \\
	  &\leq 
		\|\boldsymbol{\phi}(0)-\boldsymbol{\phi}_{\epsilon}(0) \|_{\mathcal{P}}
		+\sup_{t\in \mathbb{R}}
	  \left\| \boldsymbol{\phi}_{\epsilon}(t)- \boldsymbol{\Phi}^{t \omega_{\epsilon}}  ( \epsilon \boldsymbol{\xi} ) \right\|_{\mathcal{P}} \lesssim 2 \epsilon^2,
	\end{align*}
due to the assumption \eqref{AssumptionSmallIntialDataforProof} and the  estimate \eqref{EstimateinTheorem24} that we get from Theorem \ref{MainTheoremBothTheoremsinOne1}. Hence, Theorem \ref{TheoremBambusiNekhoroshev} applies and yields
\begin{align}\label{AssumptionSupLinearProof}
	 \sup_{|t| \lesssim   \exp \left( a / \epsilon^2 \right)} d \left( \boldsymbol{\phi}(t), \Gamma_{\text{linear}} ( \epsilon \boldsymbol{\xi} ) \right)  \lesssim \epsilon^2,
\end{align}   
for some constant $a>0$. Finally, we claim that \eqref{AssumptionSupLinearProof} implies the main estimate of Theorem \ref{MainTheoremBothTheoremsinOne2}, namely 
\begin{align*}
	\sup_{|t| \lesssim  \exp \left( a / \epsilon^2 \right)} d \left(  \boldsymbol{\phi}(t),  \Gamma_\epsilon \right) \lesssim \epsilon^2 .  
\end{align*}  
Indeed, for all $t$ such that $|t| \lesssim \exp \left( a / \epsilon^2 \right)$,  fix $  \tau_{\star}(t) $ so that
\begin{align*}
	d \left( \boldsymbol{\phi}(t), \Gamma_{\text{linear}} ( \epsilon \boldsymbol{\xi} ) \right) = \inf_{\tau \in \mathbb{R}}  \left\| \boldsymbol{\phi}(t) - \boldsymbol{\Phi}^{\tau}  ( \epsilon \boldsymbol{\xi} ) \right\|_{\mathcal{P}} =
	  \left\| \boldsymbol{\phi}(t ) - \boldsymbol{\Phi}^{\tau_{\star}(t)\omega_{\epsilon} }  ( \epsilon \boldsymbol{\xi} ) \right\|_{\mathcal{P}}. 
\end{align*}
Notice that $\tau_{\star}(t)$ can be chosen so that $\tau_{\star}(t) \in [0,T_{\epsilon}]$, since $\boldsymbol{\Phi}^{t \omega_{\epsilon}} ( \epsilon \boldsymbol{\xi} )$ is $T_{\epsilon}$-periodic with  $T_{\epsilon}=2\pi/\omega_{\epsilon}$, and hence $|\tau_{\star}(t)| \lesssim  \exp \left( a / \epsilon^2 \right)$. Then, we can estimate
\begin{align*}
	d \left(  \boldsymbol{\phi}(t),  \Gamma_\epsilon \right) &= \inf_{\boldsymbol{\psi} \in \Gamma_\epsilon} \|   \boldsymbol{\phi}(t) - \boldsymbol{\psi}    \|_{\mathcal{P}}
	 \leq \left\| \boldsymbol{\phi}(t)-\boldsymbol{\phi}_{\epsilon}\left( \tau_{\star}(t)  \right) \right\|_{\mathcal{P}} \\
 	& \leq  
 	\left\| \boldsymbol{\phi}(t)- \boldsymbol{\Phi}^{\tau_{\star}(t)\omega_{\epsilon}  }  ( \epsilon \boldsymbol{\xi} ) \right\|_{\mathcal{P}} +
 	\left\|\boldsymbol{\Phi}^{\tau_{\star}(t)\omega_{\epsilon}  }  ( \epsilon \boldsymbol{\xi} ) - \boldsymbol{\phi}_{\epsilon}\left( \tau_{\star}(t)  \right)  \right\|_{\mathcal{P}}  \\
 	&=  d \left( \boldsymbol{\phi}(t), \Gamma_{\text{linear}} ( \epsilon \boldsymbol{\xi} ) \right)  +
 	\left\|\boldsymbol{\Phi}^{\tau_{\star}(t)\omega_{\epsilon}  }  ( \epsilon \boldsymbol{\xi} ) - \boldsymbol{\phi}_{\epsilon}\left( \tau_{\star}(t)  \right)  \right\|_{\mathcal{P}} \\ 
 	&\leq  d \left( \boldsymbol{\phi}(t), \Gamma_{\text{linear}} ( \epsilon \boldsymbol{\xi} ) \right)  +
 	\sup_{\tau \in \mathbb{R} }\left\| \boldsymbol{\Phi}^{\tau \omega_{\epsilon}  }  ( \epsilon \boldsymbol{\xi} ) - \boldsymbol{\phi}_{\epsilon}\left( \tau  \right)  \right\|_{\mathcal{P}}  \\ 
 &  \lesssim 2\epsilon^2,
 \end{align*}
 due to \eqref{AssumptionSupLinearProof} and the estimate \eqref{EstimateinTheorem24} that we get from Theorem \ref{MainTheoremBothTheoremsinOne1}.

 \subsection{Acknowledgments}

 The authors would like to thank Anxo Farina Biasi for helpful discussions as well as Professor Peter Paule at the Research Institute for Symbolic Computation in Linz, Austria for giving us access information to the Zeilberger package in Mathematica. All computations presented here can be easily verified using Mathematica \cite{Mathematica}, Fast Zeilberger Package \cite{MathematicaFastZeilbergerPackage} and MultiSum Package \cite{MathematicaMultiSumPackage}.

 \subsection{Organization of the paper}
We split the paper into the following sections. We begin with Section \ref{SectionLineareigenvalueproblems} where we study the linear eigenvalue problems, and, Section \ref{SectionFourierCoefficients}, where we define and study the Fourier coefficients associated to the non-linearities.  In particular,  we derive computationally efficient formulas for the Fourier coefficients on resonant indices. This allows us to prove novel recurrence relations that we use to study the monotonicity for the Fourier coefficients and establish uniform estimates. Section  \ref{SectionExistenceofTimePeriodicSolutions} is devoted to  the existence of time-periodic solutions to the non-linear wave equations KG and WM where we prove Theorem \ref{MainTheoremBothTheoremsinOne1}. To do so, we consider an appropriate rescale of the first 1-mode and apply Theorem \ref{TheoremBambusiPaleari}. After reducing both assumptions \eqref{Assumption1BambusiPaleari} and \eqref{Assumption2BambusiPaleari} of Theorem \ref{TheoremBambusiPaleari}  to a  set of conditions for the Fourier coefficients, we use the analysis from Section \ref{SectionFourierCoefficients} to rigorously prove their validity. Finally, we conclude with Section  \ref{SectionStabilitySolutionstotheLinearWaveEquation} where we establish the non-linear stability of the time-periodic solutions to the KG model by proving Theorem \ref{MainTheoremBothTheoremsinOne2}. To do so, we consider the same rescale of the first 1-mode and apply Theorem \ref{TheoremBambusiNekhoroshev}. After reducing both assumptions \eqref{Assumption1BambusiNekhoroshev} and \eqref{Assumption2BambusiNekhoroshev} of Theorem \ref{TheoremBambusiNekhoroshev} to a set of conditions for the Fourier coefficients, we use the analysis from Section \ref{SectionFourierCoefficients} to rigorously prove their validity.

\section{The linear eigenvalue problems}\label{SectionLineareigenvalueproblems}
We consider the non-linear wave equation
\begin{align*} 
	\partial_{t}^2 \phi +L \phi = -\epsilon^2 W \phi^3 + \epsilon^4 E(\phi),
\end{align*}
for a scalar field $\phi: \mathbb{R} \times (0,\pi/2) \rightarrow \mathbb{R}$, where, using the notation from Table \ref{NotationTable}, $L$, $W$ and $E$ are given by \eqref{MainPDEIntroModelKGA}-\eqref{MainPDEIntroModelKGB} and \eqref{MainPDEIntroModelWMA}-\eqref{MainPDEIntroModelWMB} with $E=0$ in the case of the KG model. The  linearized operators are given by
\begin{align*}
L \phi &=
\begin{dcases}
\mathsf{L} \phi =  -\partial_{x}^2 \phi  -\frac{2-2 \delta  \sin ^2(x)}{\sin (x) \cos (x)} \partial_{x} \phi + \delta^2 \phi, \quad \text{for the KG}\\  
	\mathfrak{L} \phi = -\partial_{x}^2 \phi  -\frac{2 (\delta +2) \cos ^2(x)-3}{\sin (x) \cos (x)}\partial_{x} \phi + (\delta+2)^2 \phi ,\quad \text{for the WM}.
\end{dcases}
\end{align*}
 In both cases, the $L$ operator can be written in the Sturm-Liouville form
 \begin{align*}
 	L \phi &= -\frac{1}{w} \partial_{x} \left(w \partial_{x}\phi \right)+
 	\begin{dcases}
 		\delta^2 \phi,\quad \text{for the KG}, \\
 		 (\delta+2)^2 \phi,\quad \text{for the WM}, \\
 	\end{dcases}
 \end{align*}
 where $\delta  $ stands for the conformal mass and $w\in C^{\infty}[0,\pi/2]$ is a spatial weight, 
\begin{align}\label{DefinitionWeightw}
w(x) &=
\begin{dcases}
\mathsf{w}(x)= \sin ^2(x) \cos ^{2 \delta -2}(x) , \quad \text{for the KG}\\  
	\mathfrak{w} (x)= \sin ^{2\delta+1}(x) \cos ^{3}(x),\quad \text{for the WM}.
\end{dcases}
\end{align}   
 Finally, we note that, for all $\delta$ satisfying Assumption \ref{AssumptionsondeltaExistence} and for all $s\geq  1$, standard weighted Sobolev space arguments (see for example \cite{MR802206} and Chapter 2 in \cite{MR1774162}) show that $H^s ((0,\pi/2);d\mu)$ endowed with the norm $\|\cdot\|_{H^s ((0,\pi/2);d\mu)}$  is a Banach space and that the inclusion  $H^s ((0,\pi/2);d\mu) \hookrightarrow L^2 ((0,\pi/2);d\mu)$ is compact.

 \subsection{Self-adjointness} \label{se:sa}
We consider the Hilbert space
\begin{align}\label{DefinitionHilbertSpace}
	L^2 \left( \left(0,\frac{\pi }{2} \right);d\mu \right), \quad d\mu(x)=w(x)dx,
\end{align}
associated with the inner product
 \begin{align}\label{DefinitionInnerProduct}
 	(f|g)=  \int_{0}^{\pi /2} f(x) g(x) d\mu(x).
 \end{align}
The linearized operator $L$ is generated by a closed sesquilinear form 
\begin{align*}
	\alpha :H^1 \left( \left(0,\frac{\pi }{2} \right);d\mu \right) \times H^1 \left( \left(0,\frac{\pi }{2} \right);d\mu \right) \rightarrow \mathbb{R}
\end{align*}
defined by
\begin{align*}
	\alpha (\phi,\psi) &= \int_{0}^{ \pi /2}\partial_{x} \phi (x)\partial_{x}\psi(x) d\mu(x) +
	\begin{dcases}
		\delta^2  \int_{0}^{ \pi /2}\phi (x) \psi(x)   d\mu(x), \quad \text{for the KG}, \\
		 (\delta+2)^2 \int_{0}^{ \pi /2}\phi (x) \psi(x)   d\mu(x), \quad \text{for the WM},
	\end{dcases} 
\end{align*}
Clearly there exists positive constants $c_1$ and $c_2$, depending only on $\delta$, such that
	\begin{align*}
		c_1 \|\phi \|_{H^1 ((0,\pi/2);d\mu)}^2 \leq \alpha (\phi,\phi) \leq c_2 \|\phi \|_{H^1 ((0,\pi/2);d\mu)}^2,  
	\end{align*}
	for all $\phi  \in H^1((0,\pi/2);d\mu)$, hence the sesquilinear form is elliptic and strictly positive. In particular, $L$ is self-adjoint on its domain $\mathcal{D}(L)$. Therefore, all the eigenvalues of $L$ are real and any two eigenfunctions corresponding to different eigenvalues are orthogonal to each other. Moreover, by compactness, the set of eigenfunctions forms a complete basis for $L^2 ((0,\pi/2);d\mu)$. 
	
%
%

\subsection{The linear eigenvalue problem} \label{SectionEigenvalueProblem}
Next, we study the linear eigenvalue problems 
\begin{align*}
	L \phi =\lambda \phi, \quad \phi \in \mathcal{D}(L).
\end{align*}
Setting $\phi(x)=  u(y)$ with $ y=\cos(2x)$, the linear eigenvalue problems transform into
\begin{align*}
	\left(1-y^2\right) u^{\prime \prime} (y)+ 
	(   (\delta-2) -(\delta+1)  y   ) 
	u^{\prime}(y)+\frac{1}{4}\left(\lambda -\delta ^2 \right) u(y)=0, \\
	\left(1-y^2\right) u^{\prime \prime} (y)+ 
	( (1-\delta)   -(\delta+3)  y  ) 
	u^{\prime}(y)+\frac{1}{4}\left(\lambda -(2+\delta) ^2 \right) u(y)=0,
\end{align*}
for the KG and WM respectively. Now, the eigenfunctions are given by the Jacobi polynomials $P_{n}^{(a,b)}(y)$ of degree $n$,
\begin{align*}
	e_n (x) = N_n P^{(a,b)}(\cos(2x)),
\end{align*}
with parameters $(a,b)=(1/2,\delta-3/2)$ for the KG and $(a,b)=(\delta,1)$ for the WM. Note that the  Breitenlohner-Freedman bounds, $\delta \geq  3/2$ for the KG and $\delta>0$ for the WM, ensure the integrability of the weights $ (1-y)^{a} (1+y)^{b}$ with respect to which the Jacobi polynomials $P_n^{(a,b)}(y)$ form an orthonormal and complete basis. We refer the reader to \cite{MR0372517, MR2723248} for a review on Jacobi polynomials. Consequently, the eigenfunctions and eigenvalues to the linearized operators $L$ are given respectively by
\begin{align*}
	e_{n}(x) &=
	\begin{dcases}
		\mathsf{e}_n(x)= \mathsf{N}_{n}    P_n^{\left(\frac{1}{2} ,\delta -\frac{3}{2}\right)}(\cos (2 x)), \Hquad \text{for the KG}, \\
		\mathfrak{e}_n(x)= \mathfrak{N}_{n}    P_n^{\left(\delta ,1\right)}(\cos (2 x))   , \Hquad \text{for the WM},
	\end{dcases} , \\
	\omega_{n}^2 &=
	\begin{dcases}
		 \somega_{n} ^2=(2n+\delta )^2  , \Hquad \text{for the KG}, \\
		 \ssomega_{n} ^2=(2n+\delta+2 )^2    , \Hquad \text{for the WM},
	\end{dcases}
\end{align*} 
for all integers $n \geq 0$. Moreover, the set of eigenfunctions $\{e_{n}  : n \geq 0 \}$ forms an orthonormal and complete basis for $L^2 ( (0,\pi/2) ;d\mu  ) $ given in \eqref{DefinitionHilbertSpace} endowed with the inner product \eqref{DefinitionInnerProduct}. Here, $N_{n}  $ stand for normalization constants,
\begin{align*}
N_{n}=
\begin{dcases}
	\mathsf{N}_{n}  = \sqrt{
	 \frac{2\Gamma (n+1)}{\Gamma \left(n+\frac{3}{2}\right)} \frac{( 2 n+\delta) \Gamma (n+\delta )}{ \Gamma \left(n+\delta -\frac{1}{2}\right)} }, \Hquad \text{for the KG},\\
	 \mathfrak{N}_{n}  =\sqrt{\frac{2 (\delta +n+1) (\delta +2 n+2)}{n+1}}, \Hquad \text{for the WM},
\end{dcases}
\end{align*}
for all integers $n \geq 0$, so that
\begin{align}\label{OrthogonalityoftheEigenfunctions1}
 	(e_{n}  |e_{m}  ) = \mathds{1}(n=m), 
 \end{align}
for all integers $n,m \geq 0$, where $(\cdot|\cdot)$ denotes the inner product defined in \eqref{DefinitionInnerProduct}. This orthogonality condition follow immediately from the orthogonality of the Jacobi polynomials.

\section{The Fourier coefficients}  \label{SectionFourierCoefficients}
 
In this section, we define and study the Fourier coefficients associated with the cubic non-linearity in the non-linear wave equation
\begin{align} \label{NonLinearWaveequationSectionFourier}
	\partial_{t}^2 \phi +L \phi = -\epsilon^2 W \phi^3 + \epsilon^4 E(\phi),
\end{align}
for a scalar field $\phi: \mathbb{R} \times (0,\pi/2) \rightarrow \mathbb{R}$, where, using the notation from Table \ref{NotationTable}, $L$, $W$ and $E$ are given by \eqref{MainPDEIntroModelKGA}-\eqref{MainPDEIntroModelKGB} and \eqref{MainPDEIntroModelWMA}-\eqref{MainPDEIntroModelWMB} with $E=0$ in the case of the KG model.    

  \subsection{Definition of the Fourier coefficients} 
  
 Let $\phi(t,\cdot)$ be a solution to the non-linear wave equation \eqref{NonLinearWaveequationSectionFourier} expanded in terms of the basis $\{e_{n}   : n\geq 0\}$ of $L^2 ((0,\pi/2);d\mu) $,
\begin{align*}
	\phi(t,\cdot) = \sum_{m=0}^{\infty} \phi^{(m)}(t) e_{m}  .
\end{align*}
We substitute the latter into the non-linear wave equation \eqref{NonLinearWaveequationSectionFourier} and use the expansion
\begin{align}\label{DefinitionFourierByExpantion}
	We_{i}   e_{j}    e_{k} =\sum_{m=0}^{\infty}C_{ijkm} e_m
\end{align}
to obtain 
\begin{align*}
   W\phi^3(t,\cdot)   =\sum_{i,j,k=0}^{\infty}  \phi^{(i)}(t) \phi^{(j)}(t) \phi^{(k)}(t) We_{i}   e_{j}    e_{k}     =  \sum_{m=0}^{\infty}  \left( \sum_{i,j,k=0}^{\infty}C_{ijkm}   \phi^{(i)}(t) \phi^{(j)}(t) \phi^{(k)}(t)\right)e_{m} .
\end{align*}
Then, \eqref{NonLinearWaveequationSectionFourier} can be rewritten in the Fourier space as an infinite system of non-linear harmonic oscillators, 
\begin{align*} 
	& \frac{d^2}{ dt^2} \phi^{(m)}(t) + (A \phi(t))^{(m)} = \left( f (  \phi(t) ) \right)^{(m)}, 
\end{align*}
for all integers $m \geq 0$. Here, $A$ represents the linearized operator $L$ in the Fourier space,
\begin{align*}
	& (A \phi(t) )^{(m)} =  \omega_{m}  ^2 \phi^{(m)}(t) ,
\end{align*}  
and the non-linearity splits into
\begin{align*}
	\left( f  (  \phi (t)) \right)^{(m)}=
	\epsilon^2 \left( f^{(3)} (  \phi(t) ) \right)^{(m)}+
	\epsilon^4 \left( f^{(\geq 4)}(  \phi (t)) \right)^{(m)}
\end{align*}
   where the leading order is an homogeneous polynomial of degree three,
  \begin{align}\label{DefinitionfToThe3}
  	\left( f^{(3)}(  \phi (t)) \right)^{(m)}= - \sum_{i,j,k=0}^{\infty}C_{ijkm}   \phi^{(i)}(t) \phi^{(j)}(t) \phi^{(k)}(t).
  \end{align}
  Equivalently, one can take the inner product defined in \eqref{DefinitionInnerProduct} to both sides of \eqref{DefinitionFourierByExpantion} to obtain
	\begin{align*} 
	 C_{ijkm}    &= \left(W e_{i}   e_{j}    e_{k}    | e_{m}    \right) 
	  = \int_{0}^{\pi /2} W(x)	  e_{i}   (x)e_{j}   (x) e_{k}   (x)e_{m}   (x) d\mu(x) .  
	\end{align*} 	 
Using the trigonometric identities
\begin{align*}
	\sin^2 (x) = \frac{1-y}{2}, \quad \cos^2 (x) = \frac{1+y}{2}, \quad y=\cos(2x),
\end{align*} 
we can express these with respect to $y \in [-1,1]$. According to the notation from Table \ref{NotationTable}, we find
\begin{align}
	\mathsf{ C}_{ijkm}   & =\frac{1}{4^{\delta }}\prod_{\lambda_1 \in \{i,j,k,m\}} \mathsf{ N}_{\lambda_1}    \int_{-1}^{1} 
	(1-y)^{1/2} (1+y)^{2\delta -5/2} 
	 \prod_{\lambda_2 \in \{i,j,k,m\}}	 P_{\lambda_2}^{\left(\frac{1}{2},\delta -\frac{3}{2}\right)}(y)  dy,\label{DefinitionOfFourierCoefficientsModelKG} \\
	 \mathfrak{ C}_{ijkm}   & =\frac{1}{4^{\delta +2}} \prod_{\lambda_1 \in \{i,j,k,m\}} \mathfrak{N}_{\lambda_1}  \int_{-1}^{1} 
	  (1-y)^{2 \delta -1} (y+1)^3 
	    \prod_{\lambda_2 \in \{i,j,k,m\}}   P_{\lambda_2}^{\left(\delta ,1\right)}(y) dy , \label{DefinitionOfFourierCoefficientsModelWM}
\end{align}
for the KG and WM respectively. 
 \subsection{Vanishing Fourier coefficients}

We call a quadruple $(i,j,k,m)$ \textit{resonant} if and only if at least one of the conditions
\begin{align}\label{DefinitionResonantIndices}
	\omega_{i}    \pm  \omega_{j}    \pm  \omega_{k}    \pm  \omega_{m}    = 0
\end{align}
 is satisfied (the plus and minus signs are independent). Firstly, we show that the Fourier coefficients vanish on resonant indices with only one minus sign.

\begin{lemma}[Vanishing Fourier coefficients on resonant indices with only one minus sign]\label{LemmaVanishingFourier}
Fix $\delta$ satisfying Assumption \ref{AssumptionsondeltaExistence}. Then, for any integers $i,j,k,m \geq 0$ such that \eqref{DefinitionResonantIndices} holds with only one minus sign, we have
\begin{align*}
	C_{ijkm}   =0.
\end{align*} 
\end{lemma}
 \begin{proof}
 	We prove the result only for the KG model. Fix $\delta$ satisfying Assumption \ref{AssumptionsondeltaExistence} and pick integers $i,j,k,m \geq 0$ such that $\somega_{i}    + \somega_{j}    + \somega_{k}    - \somega_{m}    = 0$. Then\footnote{Notice that this is possible only in the case where $\delta$ is an integer.}, $m=\delta +i+j+k$ and according to \eqref{DefinitionOfFourierCoefficientsModelKG}, we have
\begin{align*}
		\mathsf{C}_{ijkm}    = \int_{-1}^{1} \mathsf{Q}_{N}(y) P_{m}^{\left(\frac{1}{2},\delta -\frac{3}{2}\right)}(y)  (1-y)^{\frac{1}{2}} (1+y)^{\delta  -\frac{3}{2}} d y	,
\end{align*} 
where   
\begin{align*} 
	\mathsf{Q}_{N}(y)&=  \frac{1}{4^{\delta }}\prod_{\lambda_1 \in \{i,j,k,m\}} \mathsf{N}_{\lambda_1}     
	 \prod_{\lambda_2 \in \{i,j,k\}}	 P_{\lambda_2}^{\left(\frac{1}{2},\delta -\frac{3}{2}\right)}(y)  (1+y)^{\delta -1}  ,
\end{align*}
is a polynomial of degree $N=\delta -1+i+j+k<\delta +i+j+k=m$. Consequently, the Fourier constant vanishes since the set $\{P_{m}^{\left(1/2,\delta -3/2\right)}(y): m\geq 0 \}$ forms an orthonormal and complete basis with respect to the weight $ (1-y)^{1/2} (1+y)^{\delta  -3/2} $. The other results now follow immediately using the symmetries of the Fourier coefficients with respect to $i,j,k,m$. For the WM model, the proof is similar.  
 \end{proof}

 \subsection{Non-vanishing Fourier coefficients} 
Next, we turn our attention to the the non--vanishing Fourier coefficients on resonant indices and aim towards proving the following uniform estimates as stated below.

\begin{prop}[Uniform estimates for $C_{00mm}$]\label{PropositionUniformEstimates}
	For all $\delta   $ satisfying the Assumption \ref{AssumptionsondeltaExistence}, we have that
	\begin{align}\label{FourierConditionExistence}
		\frac{ C_{00 00}}{\omega_{0}^2}-2\frac{ C_{00 mm}}{\omega_{m}^2} \neq 0,  
	\end{align}
	for all integers $m \geq 1$. Moreover, 
	\begin{align}\label{FourierConditionStability1}
	 \left \{\frac{ C_{00 00}}{\omega_{0}^2}-2\frac{ C_{00 mm}}{\omega_{m}^2} :m \geq 1 \right\} \subset (0,\infty )  
	\end{align}
	if and only if $\delta$ satisfies Assumption \ref{AssumptionsondeltanonlinearStability}.
	In this case, there exists a uniform positive constant $c_{\perp}=c_{\perp}(\delta)$, depending only on $\delta$, such that
	\begin{align}\label{FourierConditionStability2}
		\min_{\substack{ m \geq 1   }} \left(	\frac{ C_{00 00}}{\omega_{0}^2}-2\frac{ C_{00 mm}}{\omega_{m}^2} \right)  \geq c_{\perp}(\delta)
	\end{align}
\end{prop}

 \begin{remark}[Interpretation of the results in Proposition \ref{PropositionUniformEstimates}]\label{RemarkIndefiniteness}
 	The Fourier coefficients
\begin{align}\label{FourierForFigureReference}
	\left \{\frac{ C_{00 00}}{\omega_{0}^2}-2\frac{ C_{00 mm}}{\omega_{m}^2} :m \geq 1 \right\}
\end{align}
will later play a crucial role both in the existence and the non-linear stability of the time-periodic solutions to the non-linear wave equation \eqref{NonLinearWaveequationSectionFourier}. In fact, the existence of time-periodic solutions will follow from \eqref{FourierConditionExistence} and their non-linear stability from \eqref{FourierConditionStability1}-\eqref{FourierConditionStability2}. For any $\delta$ that does not satisfy Assumption \ref{AssumptionsondeltanonlinearStability}, namely $\delta \in [10,\infty) \cap \mathbb{N}$ for the KG, the sequences given by \eqref{FourierForFigureReference}, are neither uniformly positive nor uniformly negative but they change sign as $m$ increases. 
\end{remark}

 \subsubsection{General strategy}  \label{SectionGeneralStrategy}
 
 In the following lines, we develop a general method that yields linear recurrence relations for the Fourier coefficients from which both their asymptotic behaviour and closed formulas follow easily (Section \ref{SectionTimePeriodicRostworowskiMaliborski}). In order to establish Proposition \ref{PropositionUniformEstimates}, we study the monotonicity of the sequences
\begin{align*}
 \left \{	f(m)=\frac{C_{00mm}}{\omega_{m}^2} :m \geq 1 \right\}
\end{align*} 
and proceed as follows. To begin with, we use several identities of orthogonal polynomials from \cite{MR2723248,MR1820610,MR0372517} to remove the integrals in \eqref{DefinitionOfFourierCoefficientsModelKG} and \eqref{DefinitionOfFourierCoefficientsModelWM} obtaining the so called ``computationally efficient formulas'',
\begin{align}\label{DefinitionSumfGegenralStrategy}
	f(m)=\frac{C_{00mm}}{\omega_{m}^2} = \sum_{k=0}^{K } F(m,k),
\end{align} 
for some integer $K$ (possibly depending on $m$) and a function $F(m,k)$, both given in \textit{closed} formulas depending on the model we consider. Most importantly, $F(m,k)$ turns out to be a hypergeometric function with respect to both arguments $m$ and $k$, meaning that both $F(m+1,k)/F(m,k)$ and $F(m,k+1)/F(m,k)$ are rational functions of $m$ and $k$. Next, we are interested in finding linear recurrence relations for the sums $f(m)$ in \eqref{DefinitionSumfGegenralStrategy}. We claim that it suffices to find a recurrence relation for the summand $F(m,k)$ with respect to $m$. In other words, given a hypergeometric function $F(m,k)$ and an integer $K$, we would like to find a function $G(m,k)$, as well as coefficients $\{\alpha_{j}(m):j =0,1,\dots,J\}$, for some (possibly large) integer $J$, all in closed forms, so that  
\begin{align}\label{RecurrenceRelationFGegenralStrategy}
	\sum_{j=0}^{J} \alpha_{j}(m)F(m+j,k)=G(m,k+1)-G(m,k).
\end{align}
Then, if \eqref{RecurrenceRelationFGegenralStrategy} is possible, a linear recurrence relation for $f(m)$, given by \eqref{DefinitionSumfGegenralStrategy}, follows immediately by summing both sides of \eqref{RecurrenceRelationFGegenralStrategy} with respect to $k$. Indeed, since the coefficients $\{\alpha_{j}(m):j =0,1,\dots,J\}$ are all independent of $k$, summing both sides of \eqref{RecurrenceRelationFGegenralStrategy} with respect to $k$ yields
\begin{align*} 
 \sum_{k=0}^{K} \left(	\sum_{j=0}^{J} \alpha_{j}(m)F(m+j,k) \right) &=\sum_{k=0}^{K} \left(G(m,k+1)-G(m,k) \right) \Longrightarrow \nonumber \\
  \sum_{j=0}^{J}\alpha_{j}(m)\left(	\sum_{k=0}^{K} F(m+j,k) \right) &= G(m,K+1)-G(m,0)  \Longrightarrow \nonumber \\
   \sum_{j=0}^{J}\alpha_{j}(m)f(m+j) &=G(m,K+1)-G(m,0),
\end{align*}
which is the desired recurrence relation for $f(m)$. The coefficients $\{\alpha_{j}(m):j =0,1,\dots,J\}$ and the function $G(m,k)$ can be found explicitly using Zeilberger's algorithm, see \cite{MR1644447} and Chapter 6 in \cite{MR1379802}.  The significance of the recurrence relation \eqref{RecurrenceRelationFGegenralStrategy} for the summand is twofold. On the one hand, for all proper hypergeometric functions $F(m,k)$, it always exists for a (possibly large) integer $J$ (Theorem 6.2.1, page 105 in \cite{MR1379802}). On the other hand, once the coefficients $\{\alpha_{j}(m):j =0,1,\dots,J\}$ and $G(m,k)$ are found in closed formulas, then one can rigorously prove the validity of \eqref{RecurrenceRelationFGegenralStrategy}  by a (possibly long but) straightforward computation. In Appendix \ref{AppendixZeilbergerAlgorithm}, we describe how this algorithm works in general (Lemma \ref{SectionAppendixalgorithm}) and we also run it through by hand (Section \ref{SectionAppendixalgorithmfortheLemma}) producing step-by-step the results that will follow later in this section.

\subsubsection{KG model} 

 Firstly, we focus on the KG model and derive a computationally efficient formula the Fourier coefficient
	\begin{align*} 
	 \mathsf{C}_{00 mm}     
	 &=\frac{1}{4^{\delta}} \left( \mathsf{ N}_{0}     \mathsf{N}_{m}      \right)^2 \int_{-1}^{1} 
	 	 \left( P_{m}^{\left(\frac{1}{2},\delta -\frac{3}{2}\right)}(y) \right)^2 (1-y)^{\frac{1}{2}} (1+y)^{2\delta -\frac{5}{2}}  dy,
	\end{align*}
for all integers $m \geq 1$. To do so, we will use the following identities:
\begin{itemize}  [leftmargin=*]
	\item The symmetries of the Jacobi polynomials \cite{MR0372517},
	\begin{align}\label{SymmetriesJacobi}
		P_{m}^{(\alpha,\beta)}(y)=(-1)^m P_{m}^{(\alpha,\beta)}(-y),
	\end{align}
	valid for all real $\alpha,\beta \geq -1$, real $y\in [-1,1]$ and integers $m \geq 0$.
	\item The relation between Jacobi and Gegenbauer polynomials (relation 18.7.16 in \cite{MR2723248}),
	\begin{align}\label{JacobiInTermsofGegenabuer}
	 xP^{(\delta-\frac{3}{2},\frac{1}{2})}_{n}\left(2x^{2}-1\right)=\frac{\left(\frac{1}{2}\right)_{n+1}}{\left(\delta-1\right)_{n+1}}	C^{(\delta-1)}_{2n+1}\left(x\right),
	\end{align}
	valid for all real $\delta \geq 3/2$, real $x\in [-1,1]$ and integers $n \geq 0$.
\item The combination of a connection and a linearization formula for Gegenbauer polynomials (relation (20) in  \cite{MR1820610}),  \begin{align}\label{CombinationConnectionLinearization}
	\left(	 C^{(\delta-1)}_{2m+1}\left(x\right) \right)^2= \frac{\Gamma \left(2 \delta -\frac{5}{2}\right)}{\sqrt{\pi } 4^{\delta -1} \Gamma (\delta -1)^2} \frac{\Gamma (2 m+2 \delta -1)}{\Gamma (2 m+2)} \sum_{k=0}^{2m+1} 	 l_{m}  (k) C^{\left(2 \delta -\frac{5}{2} \right)}_{2k}\left(x\right),  
\end{align}
valid for all real $\delta \geq 3/2$, real $x\in [-1,1]$ and integers $n \geq 0$,  where the coefficients are given by
\begin{align}\label{Coefficientlmdeltak}
	 l_{m}  (k)=  \frac{(4 \delta +4 k-5) \Gamma \left(k+\frac{1}{2}\right) \Gamma (k+\delta -1) \Gamma \left(-k+2 m+\frac{3}{2}\right) \Gamma (k+2 m+2 \delta -1)}{\Gamma \left(k+\delta -\frac{1}{2}\right) \Gamma (k+2 \delta -2) \Gamma (-k+2 m+2) \Gamma \left(k+2 m+2 \delta -\frac{1}{2}\right)}.
\end{align}  
\item The connection formula for Gegenbauer polynomials (relation 18.18.16 in \cite{MR2723248})
\begin{align}\label{ConnectionFormula}
	 C_{2 k}^{\left(2 \delta -\frac{5}{2}\right)}(x) = \sum _{p=0}^k \zeta_{k} (p) C_{2 p}^{(2 \delta -2)}(x),
\end{align}
valid for all real $\delta \geq 3/2$, real $x\in [-1,1]$ and integers $k \geq 0$, where the coefficient reads
\begin{align}\label{Coefficientszkdeltap}
\zeta_{k} (p) =	\frac{  \delta +p-1}{\delta -1}  \frac{  \left(-\frac{1}{2}\right)_{k-p} \left(2 \delta -\frac{5}{2}\right)_{k+p} }{ \Gamma (k-p+1) (2 \delta -1)_{k+p}}.
\end{align} 
\end{itemize}
The following result establishes a computationally efficient formula for $\mathsf{C}_{0 0 mm}  $.

\begin{lemma}[Computationally efficient formula for $\mathsf{C}_{00 mm}  $ ]\label{LemmaComputationallyEfficientFormulaforC00mmKG}
	Fix $\delta  $ satisfying Assumption \ref{AssumptionsondeltaExistence} and let $m \geq 0$ be any integer. Then,  
	 \begin{align*}
 	\mathsf{C}_{00 mm}     
	 &=  \sum_{k=0}^{2m+1} \mathsf{S}_{m}(k),
\end{align*}
where
\begin{align*} 
	\mathsf{S}_m (k)&=  \frac{\Gamma (\delta +1) \Gamma \left(2 \delta -\frac{3}{2}\right)}{\pi ^2 \Gamma \left(\delta -\frac{1}{2}\right)} \frac{( 5-4 \delta -4 k) \Gamma (k+\delta -1) \Gamma \left(k+2 \delta -\frac{5}{2}\right)}{\Gamma \left(k+\delta -\frac{1}{2}\right) \Gamma (k+2 \delta -2) \Gamma (k+2 \delta -1)} \frac{\Gamma \left(k-\frac{1}{2}\right) \Gamma \left(k+\frac{1}{2}\right)}{\Gamma (k+1)}\cdot \\
	 & \cdot \frac{(\delta +2 m) \Gamma \left(-k+2 m+\frac{3}{2}\right) \Gamma (k+2 m+2 \delta -1)}{\Gamma (-k+2 m+2) \Gamma \left(k+2 m+2 \delta -\frac{1}{2}\right)}.
\end{align*}
\end{lemma}
\begin{proof}
	Fix $\delta  $ satisfying Assumption \ref{AssumptionsondeltaExistence} and let $m \geq 0$ be any integer. Then, we use the symmetries of the Jacobi polynomials \eqref{SymmetriesJacobi},   the change of variables $z=-y$ and $z=2x^2-1$, as well as  \eqref{JacobiInTermsofGegenabuer} to express the Jacobi polynomials in terms of the Gegenbauer polynomials to obtain 
\begin{align*} 
	 \mathsf{C}_{00 mm}    
	 	 &=\frac{1}{4^{\delta}} \left( \mathsf{N}_{0}   \mathsf{ N}_{m}     \right)^2 \int_{-1}^{1} 
	 	  \left( P_{m}^{\left(\frac{1}{2},\delta -\frac{3}{2}\right)}(y) \right)^2 (1-y)^{\frac{1}{2}} (1+y)^{2\delta -\frac{5}{2}}  dy \\
	 	  &=\frac{1}{4^{\delta}}  \left( \mathsf{N}_{0}   \mathsf{ N}_{m}     \right)^2 \int_{-1}^{1} 
	 	 \left((-1)^m P_{m}^{\left(\delta -\frac{3}{2},\frac{1}{2}\right)}(-y) \right)^2 (1-y)^{\frac{1}{2}} (1+y)^{2\delta -\frac{5}{2}}    dy \\
	 	 &=\frac{1}{4^{\delta}} \left( \mathsf{N}_{0}   \mathsf{ N}_{m}     \right)^2 \int_{-1}^{1} 
	 	  \left( P_{m}^{\left(\delta -\frac{3}{2},\frac{1}{2}\right)}(z) \right)^2 (1-z)^{2\delta -\frac{5}{2}}  (1+z)^{\frac{1}{2}}    dz \\
	 	 &=\frac{1}{4^{\delta}}  \left( \mathsf{N}_{0}   \mathsf{ N}_{m}     \right)^2 \int_{0}^{1} 
	 	  \left( P_{m}^{\left(\delta -\frac{3}{2},\frac{1}{2}\right)}(2x^2-1) \right)^2 (2(1-x^2))^{2\delta -\frac{5}{2} } (2x^2)^{ \frac{1}{2}}   4x dx \\
	 	 &=   \left( \mathsf{N}_{0}   \mathsf{ N}_{m}     \right)^2 \int_{0}^{1} 
	 	 \left(x P_{m}^{\left(\delta -\frac{3}{2},\frac{1}{2}\right)}(2x^2-1) \right)^2 (1-x^2)^{2\delta-\frac{5}{2} } dx \\
	 	 & = \left(\mathsf{N}_{0}   \mathsf{ N}_{m}    \frac{\left(\frac{1}{2}\right)_{m+1}}{\left(\delta-1\right)_{m+1}}  \right)^2 \int_{0}^{1} 
	 	\left( C^{(\delta-1)}_{2m+1}\left(x\right) \right)^2 (1-x^2)^{2\delta-\frac{5}{2} } dx .
	\end{align*}
Moreover,  \eqref{CombinationConnectionLinearization} yields that 
\begin{align*}
		\int_{0}^{1} 
	 	\left( C^{(\delta-1)}_{2m+1}\left(x\right) \right)^2 (1-x^2)^{2\delta-\frac{5}{2} } dx=
	 	 \frac{\Gamma \left(2 \delta -\frac{5}{2}\right)}{\sqrt{\pi } 4^{\delta -1} \Gamma (\delta -1)^2} \frac{\Gamma (2 m+2 \delta -1)}{\Gamma (2 m+2)} \sum_{k=0}^{2m+1} 	 l_{m}  (k)  \mathsf{J}_m(k)  ,
	\end{align*} 
	where we set
	\begin{align*}
		\mathsf{J}_m(k) =\int_{0}^{1} 
	 	 C^{\left(2 \delta -\frac{5}{2} \right)}_{2k}\left(x\right) (1-x^2)^{2\delta-\frac{5}{2} } dx.
	\end{align*}
To compute $\mathsf{J}_m(k) $, we use the connection formula \eqref{ConnectionFormula} and obtain
\begin{align*}
	\mathsf{J}_m(k) =
\sum _{p=0}^k \zeta_{k} (p) \int_{0}^{1} 
	 C_{2 p}^{(2 \delta -2)}\left(x\right)(1-x^2)^{2\delta-\frac{5}{2} } dx.
\end{align*}
Notice that, for $\lambda=2 \delta -2$, the integral above is of the form
\begin{align*}
	 \int_{0}^{1} 
	  C_{2 p}^{(\lambda)}(x) (1-x^2)^{\lambda-\frac{1}{2} } dx
\end{align*}
that vanishes unless $p =0$ due to the fact that the set of Gegenbauer polynomials $\{C_{2 p}^{(\lambda)}(x) : p \geq 0\}$ forms an orthonormal and complete basis with respect to the weight $(1-x^2)^{\lambda-1/2 }$. Therefore, using the fact that $C_{0}^{(2 \delta -2)}\left(x\right)=1$, we deduce
\begin{align*}
	\mathsf{J}_m(k)=     \zeta_{k} (0) \int_{0}^{1} 
	  (1-x^2)^{2\delta-\frac{5}{2} } dx
	  =   \zeta_{k} (0) \frac{\sqrt{\pi } \Gamma \left(2 \delta -\frac{3}{2}\right)}{2 \Gamma (2 \delta -1)}
\end{align*}
 and putting all together yields
\begin{align*}
		\mathsf{C}_{00mm}     
	 &= \left(\mathsf{N}_{0}   \mathsf{ N}_{m}    \frac{\left(\frac{1}{2}\right)_{m+1}}{\left(\delta-1\right)_{m+1}}  \right)^2 \frac{\Gamma \left(2 \delta -\frac{5}{2}\right)}{\sqrt{\pi } 4^{\delta -1} \Gamma (\delta -1)^2} \frac{\sqrt{\pi } \Gamma \left(2 \delta -\frac{3}{2}\right)}{2 \Gamma (2 \delta -1)} \frac{\Gamma (2 m+2 \delta -1)}{\Gamma (2 m+2)} \sum_{k=0}^{2m+1} 	 l_{m}  (k)  \zeta_{k} (0) .
	\end{align*} 
Finally,  a direct computation using the closed formulas for $ l_{m}  (k)  $ and $   \zeta_{k} (0)$ simplifies the result and yields the closed formula as stated above. 
	 \end{proof}

Next, Lemma \ref{LemmaComputationallyEfficientFormulaforC00mmKG} allows us to prove  a linear recurrence relation for the Fourier coefficients.

\begin{lemma}[Recurrence relation for $\mathsf{C}_{00mm}$]\label{LemmaRecurrenceRelationKG}
	Fix $\delta  $ satisfying Assumption \ref{AssumptionsondeltaExistence}. Then, the sequences  
	 \begin{align*}
 	\left \{  \mathsf{f}_m= \frac{ \mathsf{C}_{00 mm}}{\somega_{m}^2} : m \geq 1 \right\}      
\end{align*}
satisfy the 2-step recurrence relation
\begin{align*}
	\mathsf{P}_{m} \mathsf{f} _{m}- \mathsf{Q}_{m}  \mathsf{f} _{m+1}+\mathsf{R}_{m}  \mathsf{f} _{m+2}=0,
\end{align*}
where
\begin{align*}
	\mathsf{P}_m &=2 (m+1)^2 (2 m+1) (2 m+3) (\delta +m) (\delta +2 m) (\delta +2 m+3) (2 \delta +2 m-1) \big(9 \delta +4 m^2 \\
	&+4 \delta  m+12 m+7\big) , \\
	\mathsf{Q}_m & =(\delta +2 m+2)^2 \big(315 \delta ^5+728 \delta ^4+914 \delta ^3+418 \delta ^2-41 \delta +832 \delta ^2 m^6+704 \delta ^3 m^5 \\
	&+5760 \delta ^2 m^5+320 \delta ^4 m^4+4432 \delta ^3 m^4+15432 \delta ^2 m^4+64 \delta ^5 m^3+1824 \delta ^4 m^3+10352 \delta ^3 m^3 \\
	&+20224 \delta ^2 m^3 +336 \delta ^5 m^2+3620 \delta ^4 m^2+11064 \delta ^3 m^2+13402 \delta ^2 m^2+512 \delta  m^7+3840 \delta  m^6 \\
	&+11424 \delta  m^5+17168 \delta  m^4+13616 \delta  m^3+5280 \delta  m^2+128 m^8+1024 m^7+3296 m^6+5440 m^5 \\
	&+4792 m^4+2016 m^3+154 m^2+572 \delta ^5 m+2880 \delta ^4 m+5330 \delta ^3 m+4100 \delta ^2 m+698 \delta  m \\
	&-140 m-30\big) , \\
	\mathsf{R}_m &= 2 (m+2) (2 m+5) (\delta +m+1)^2 (\delta +2 m+1) (\delta +2 m+4) (2 \delta +2 m+1) (2 \delta +2 m+3)\cdot \\
	&~\cdot\big(5 \delta +4 m^2+4 \delta  m+4 m-1\big)  .
\end{align*}
\end{lemma}
\begin{proof}
	Fix $\delta  $ satisfying Assumption \ref{AssumptionsondeltaExistence} and pick an integer $m \geq 1$. According to Lemma \ref{LemmaComputationallyEfficientFormulaforC00mmKG}, we have that 
	\begin{align*}
	\mathsf{f}_{m}=	 \frac{ \mathsf{C}_{00 mm}}{\somega_{m}^2} =\sum_{k=0}^{2m+1} \mathsf{F}(m,k),
	\end{align*}
	for some function $ \mathsf{F}(m,k)$ given by a closed formula. Now, Zeilberger's algorithm (Chapter 6 in \cite{MR1379802}) yields functions $ \mathsf{P}_m $, $\mathsf{Q}_m $ and $ \mathsf{R}_m $ as defined above, as well as a function $\mathsf{G}_m(k)$, so that  
	\begin{align}\label{ZeilbergerResultModelKG}
		 \mathsf{P}_m  \mathsf{F}(m,k) - \mathsf{Q}_m \mathsf{F}(m+1,k) +  \mathsf{R}_m \mathsf{F}(m+2,k)=\mathsf{G}(m,k+1)-\mathsf{G}(m,k).
	\end{align}
	We note that, since all the functions are given explicitly in closed formulas, one can rigorously prove the validity of \eqref{ZeilbergerResultModelKG} by a long but straight-forward computation. Finally, using the closed formula for $\mathsf{G}(m,k)$, we compute
	\begin{align*}
		\mathsf{G}(m,2m+2)-\mathsf{G}(m,0)=0,
	\end{align*}
	and, as explained in Section \ref{SectionGeneralStrategy}, summing \eqref{ZeilbergerResultModelKG} with with respect to $k$ proves the claim. In Appendix \ref{AppendixZeilbergerAlgorithm}, we describe how this algorithm works in general (Lemma \ref{SectionAppendixalgorithm}) and we also run it through by hand (Section \ref{SectionAppendixalgorithmfortheLemma}) producing step-by-step the results of this Lemma.  
\end{proof}

Now, with the recurrence relation of Lemma \ref{LemmaRecurrenceRelationKG} at hand,  the monotonicity of the Fourier coefficients follows easily.

\begin{lemma}[Monotonicity for $\mathsf{C}_{00mm}$]\label{LemmaMonotonicityC00mmKG}
	Fix $\delta  $ satisfying Assumption \ref{AssumptionsondeltaExistence}. Then, the sequences  
	 \begin{align*}
 	\left \{ \mathsf{f}_{m}= \frac{ \mathsf{C}_{00 mm}}{\somega_{m}^2} : m \geq 1 \right\}      
\end{align*}
are all strictly decreasing with respect to $m$.
\end{lemma} 
\begin{proof}
	Fix $\delta  $ satisfying Assumption \ref{AssumptionsondeltaExistence} and pick an integer $m \geq 1$. Then, we set
	\begin{align*}
		  \mathsf{x}_m= \frac{\mathsf{f}_{m+1}}{\mathsf{f}_m}, 
	\end{align*}
	and Lemma \ref{LemmaRecurrenceRelationKG} yields the recurrence relation
	\begin{align*}
		\mathsf{x}_{m+1}=\mathsf{A}_m - \frac{\mathsf{B}_m}{\mathsf{x}_m}, 
	\end{align*}
where
\begin{align*}
  \mathsf{A}_m=\frac{\mathsf{Q}_m}{\mathsf{R}_m}, \quad \mathsf{B}_m=\frac{\mathsf{P}_m}{\mathsf{R}_m}.
\end{align*}
In the Appendix (Lemma \ref{Lemma1AuxiliarycomputationsKGAppendix}), we show that
\begin{align*}
	\mathsf{R}_m > 0, \quad \mathsf{B}_m >0, \quad \mathsf{A}_m -\mathsf{B}_m <1,
\end{align*}  
for all integers $m\geq 1$. We claim that $\mathsf{x}_m <1$, for all integers $m \geq 1$. Indeed, one can use Lemma \ref{LemmaComputationallyEfficientFormulaforC00mmKG} to compute $\mathsf{C}_{0011}$ and $\mathsf{C}_{0022}$ so that
\begin{align*}
	\mathsf{x}_1=\frac{\mathsf{C}_{2}}{\mathsf{C}_1}=
	\frac{\somega_{1}^2\mathsf{C}_{0022}}{\somega_{2}^2 \mathsf{C}_{0011}} = \frac{63 \delta ^5+196 \delta ^4+111 \delta ^3-78 \delta ^2-52 \delta -24}{80 \delta ^5+448 \delta ^4+492 \delta ^3-136 \delta ^2-236 \delta -48}.
\end{align*}
In the Appendix (Lemma \ref{Lemma2AuxiliarycomputationsKGAppendix}), we show that $\mathsf{x}_1 <1$, for all integers $\delta \geq 2$. Next, we assume that there exists an integer $m \geq 1$ such that $\mathsf{x}_m <1$. Then, since $\mathsf{B}_m>0$ and $\mathsf{A}_m -\mathsf{B}_m <1$, we infer 
\begin{align*}
	\mathsf{x}_{m+1}=\mathsf{A}_m - \frac{\mathsf{B}_m}{\mathsf{x}_m} <\mathsf{A}_m - \mathsf{B}_m<1,
\end{align*}
that completes the proof.
\end{proof}

Finally, we prove Proposition \ref{PropositionUniformEstimates} for the KG. 

\begin{proof}[Proof of Proposition \ref{PropositionUniformEstimates} for the KG]
	Fix $\delta   $ satisfying the Assumption \ref{AssumptionsondeltaExistence} and pick an integer $m \geq 1$. Then, Lemma \ref{LemmaMonotonicityC00mmKG} yields that	the sequences
	\begin{align*}
		\left \{\frac{ \mathsf{C}_{00 00}}{\somega_{0}^2}-2\frac{ \mathsf{C}_{00 mm}}{\somega_{m}^2} :m \geq 1 \right\}
	\end{align*}
	are all strictly increasing. Firstly, we claim that
	\begin{align*}
		\frac{ \mathsf{C}_{00 00}}{\somega_{0}^2}-2\frac{ \mathsf{C}_{00 mm}}{\somega_{m}^2}  \neq 0,
	\end{align*}
for all $\delta \ge 2.$
	Indeed, we split the set $\{m\geq 1\}$ into $\{m=1\}$, $\{m=2\}$ and $\{m \geq 3\}$ and using Lemma \ref{LemmaComputationallyEfficientFormulaforC00mmKG} to compute $\mathsf{C}_{0000}$, $\mathsf{C}_{0011}$, $\mathsf{C}_{0022}$ and $\mathsf{C}_{0033}$, we infer
\begin{align*}  
	\frac{ \mathsf{C}_{00 00}}{\somega_{0}^2}-2\frac{ \mathsf{C}_{00 11}}{\somega_{1}^2}  &= \frac{ (2 \delta -1) U_1(\delta)}{(\delta +2) (2 \delta +1)}
		 \frac{ \pi  \Gamma (4 \delta -3)}{ 2^{8 \delta -7}\Gamma \left(\delta -\frac{1}{2}\right) \Gamma \left(\delta +\frac{1}{2}\right)^3}. \\
	\frac{ \mathsf{C}_{00 00}}{\somega_{0}^2}-2\frac{ \mathsf{C}_{00 22}}{\somega_{2}^2}
	 &= \frac{(2 \delta -1) (2 \delta +1) U_2(\delta)}{4^{\delta +3} (\delta +1)^2 (2 \delta +3)} \frac{\Gamma (\delta ) \Gamma \left(2 \delta -\frac{3}{2}\right)}{\Gamma \left(\delta +\frac{3}{2}\right)^3}, \\
	 \frac{ \mathsf{C}_{00 00}}{\somega_{0}^2}-2\frac{ \mathsf{C}_{00 mm}}{\somega_{m}^2}
	 & \geq \frac{ \mathsf{C}_{00 00}}{\somega_{0}^2}-2\frac{ \mathsf{C}_{00 33}}{\somega_{3}^2} 
	 = \frac{(2 \delta +5) U_3\left(\delta\right)}{2^{2 \delta +7} (2 \delta +3)^2}\frac{\Gamma (\delta ) \Gamma \left(2 \delta -\frac{3}{2}\right)}{\Gamma \left(\delta -\frac{1}{2}\right) \Gamma \left(\delta +\frac{7}{2}\right)^2},
	\end{align*} 
	for all integers $m \geq 3$, where
	\begin{align*}
	U_{1} (\delta) & = -\delta ^3+10 \delta ^2-2, \\
		U_2(\delta) &=\delta ^6-66 \delta ^5+101 \delta ^4+296 \delta ^3+48 \delta ^2-8 \delta -12, \\
		U_{3} (\delta) &=83 \delta ^8+54 \delta ^7+2072 \delta ^6+13644 \delta ^5+25529 \delta ^4+15930 \delta ^3-1008 \delta ^2-1944 \delta -810.
	\end{align*}
A direct computation shows that both $U_1$ and $U_2$ do not have integer roots whereas $U_3 (\delta)>0$, for all integers $\delta \geq 2$, that proves the claim. Secondly, we claim that 
\begin{align*}
	\left \{	\frac{ \mathsf{C}_{00 00}}{\somega_{0}^2}-2\frac{ \mathsf{C}_{00 mm}}{\somega_{m}^2}  > 0 : m \geq 1 \right \} \subset (0,\infty) 
	\end{align*}
if and only if $\delta$ satisfies Assumption \ref{AssumptionsondeltanonlinearStability}. Indeed, the claim is equivalent to
	\begin{align*}
		& 
		 \min_{m \geq 1} \left(	\frac{ \mathsf{C}_{00 00}}{\somega_{0}^2}-2\frac{ \mathsf{C}_{00 mm}}{\somega_{m}^2} \right) >0 
		  \Longleftrightarrow \frac{ \mathsf{C}_{00 00}}{\somega_{0}^2}-2\frac{ \mathsf{C}_{00 11}}{\somega_{1}^2}  >0 
 \Longleftrightarrow -\delta ^3+10 \delta ^2-2>0 
	\end{align*}
	which in turn holds true if and only if $\delta$ satisfies Assumption \ref{AssumptionsondeltanonlinearStability}, namely $2\leq  \delta \leq 9 $. Finally, in this case, we set
	\begin{align*}
		\mathsf{c}_{\perp}(\delta)= \min_{m \geq 1} \left(	\frac{ \mathsf{C}_{00 00}}{\somega_{0}^2}-2\frac{ \mathsf{C}_{00 mm}}{\somega_{m}^2} \right)= 	\frac{ \mathsf{C}_{00 00}}{\somega_{0}^2}-2\frac{ \mathsf{C}_{00 11}}{\somega_{1}^2}   >0,
	\end{align*}
	that completes the proof. 
\end{proof}

Finally, we note the following remark.

\begin{remark}[Closed formulas for $\mathsf{C}_{00mm}$]\label{RemarkClosedFormulasMathsfC00mmModelKG}
	One can solve the recurrence relation of Lemma \ref{LemmaRecurrenceRelationKG}  to find closed formulas for $\mathsf{C}_{00mm}$ provided that $\delta$ is fixed (and not too large). For example, for $\delta \in \{2,3,4,5\}$, we find
	\begin{align*}
		\mathsf{C}_{00mm} &= \frac{8}{\pi }, \\
		\mathsf{C}_{00mm} &=\frac{4 (3 m (m+3)+7)}{\pi  (m+1) (m+2)} , \\
		\mathsf{C}_{00mm} &=\frac{16 (m (m+4) (20 m (m+4)+153)+297)}{5 \pi  (m+1) (m+3) (2 m+3) (2 m+5)} , \\
		\mathsf{C}_{00mm} &= \frac{20 (m (m+5) (m (m+5) (28 m (m+5)+475)+2694)+5148)}{7 \pi  (m+1) (m+2) (m+3) (m+4) (2 m+3) (2 m+7)}, 
	\end{align*}
	respectively, for all integers $m\geq 1$.   Figure \ref{FigureMathsfC00mmFordelta345} illustrates the Fourier coefficients $\mathsf{C}_{00mm}$ and their limiting values.
	\begin{figure}[h]
\centering
\includegraphics[width=.55\textwidth]{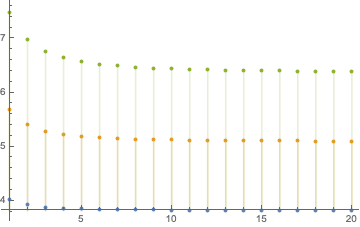}
\caption{The Fourier coefficients $\mathsf{C}_{00mm}$ for $\delta=3$ (blue), for $\delta=4$ (orange) and for $\delta=5$ (green). They are all decreasing with respect to $m$ and approach their limiting values $12/\pi$, $16/\pi $ and $20/\pi $ respectively. }\label{FigureMathsfC00mmFordelta345} 
\end{figure}
\end{remark}

 \subsubsection{WM model}

Secondly, we focus on the WM model and derive a computationally efficient formula the Fourier coefficient 
\begin{align*}  
	\mathfrak{ C}_{00 mm}      
	 &=\frac{1}{4^{ 2+\delta }}   \mathfrak{ N}_{0} ^2   \mathfrak{  N}_{m} ^2 \int_{-1}^{1}  
	\left( P_{m}^{\left(\delta,1 \right)}(y)\right)^2 (1-y)^{  2 \delta -1} (1+y)^{3}  dy  ,
\end{align*}
	for all integers $m \geq  0$. To do so, we will use the following identities:
	\begin{itemize}[leftmargin=*]
		\item The connection formula for Jacobi polynomials (relation 18.18.14 in \cite{MR2723248}),
		\begin{align}\label{ConnectionFormulaJacobi}
			P_m^{(\delta ,1)}(y)=\sum _{l=0}^m  L_m( l) P_l^{(2 \delta -1,1)}(y),
		\end{align}
		valid for all integers $m \geq 0$ and real $\delta\geq 1$, where the coefficient is given by
		\begin{align*}
			L_m( l)=\frac{(2 \delta +2 l+1) \Gamma (m+2) \Gamma (l+2 \delta +1) \Gamma (l+m+\delta +2) }{\Gamma (l+2) \Gamma (-l+m+1) \Gamma (m+\delta +2) \Gamma (l+m+2 \delta +2)}(-(\delta-1) )_{m-l}.
		\end{align*}
		Using the fact that $\delta$  satisfies Assumption \ref{AssumptionsondeltaExistence} together with relation 5.2.7 in \cite{MR2723248}, we infer
		\begin{align*}
			(-(\delta-1) )_{m-l}= \frac{(-1)^{m-l} \Gamma(\delta)}{\Gamma(\delta-m+l)}   \mathds{1}(0 \leq m-l \leq \delta-1),
		\end{align*}
		hence one can rewrite \eqref{ConnectionFormulaJacobi} as follows
		\begin{align*}
			P_m^{(\delta ,1)}(y)    
			=\sum _{k=0}^m  L_m( m-k) P_{m-k}^{(2 \delta -1,1)}(y) 
			= \sum _{k=0}^{\min (m,\delta -1)} \tilde{L}_m(k) P_{m-k}^{(2 \delta -1,1)}(y),
		\end{align*}
		valid for all integers $m \geq 0$ and integers $\delta\geq 1$, where the coefficient is given by
		\begin{align*}
			\tilde{L}_m(k) =\frac{(-1)^k \Gamma (\delta ) \Gamma (m+2) (2 \delta -2 k+2 m+1) \Gamma (-k+2 m+\delta +2) \Gamma (-k+m+2 \delta +1)}{\Gamma (k+1) \Gamma (\delta -k) \Gamma (-k+m+2) \Gamma (m+\delta +2) \Gamma (2 (m+\delta +1)-k)}.
		\end{align*}
		\item The recurrence relation for Jacobi polynomials
		\begin{align}
			y^2 P_n^{(2 \delta -1,1)}(y) &= U_n P_{n+2}^{(2 \delta -1,1)}(y)+V_n P_{n+1}^{(2 \delta -1,1)}(y)+W_n P_n^{(2 \delta -1,1)}(y)\nonumber  \\
			&+X_n P_{n-1}^{(2 \delta -1,1)}(y)+Y_n P_{n-2}^{(2 \delta -1,1)}(y),\label{RecurrenceRelation1}
		\end{align}
		valid for all integers $n \geq 0$ and real $y\in[-1,1]$, where the coefficients are given by
		\begin{align*}
			U_n &= P_n P_{n+1}, \quad V_n = P_n Q_{n+1}+Q_{n}P_{n}, \quad 
			W_n = P_n R_{n+1}+Q_{n}^2 +R_n P_{n-1},  \\
			X_{n} &= R_n(Q_n + Q_{n-1}), \quad Y_n=R_n R_{n-1}.
		\end{align*}
		This can be proved by applying twice the well-known recurrence relation \cite{MR0372517}
		\begin{align}\label{RecurrenceRelation2}
			y P_n^{(2 \delta -1,1)}(y) = P_n P_{n+1}^{(2 \delta -1,1)}(y)+Q_n P_n^{(2 \delta -1,1)}(y)+R_n P_{n-1}^{(2 \delta -1,1)}(y),
		\end{align}
		valid for all integers $n \geq 0$ and real $y\in[-1,1]$, where the coefficients are given by
		\begin{align*}
			P_n= \frac{(n+1) (2 \delta +n+1)}{(\delta +n+1) (2 \delta +2 n+1)}, \quad Q_n=\frac{(1-\delta ) \delta }{(\delta +n) (\delta +n+1)},\quad R_n=\frac{(n+1) (2 \delta +n-1)}{(\delta +n) (2 \delta +2 n+1)}.
		\end{align*}
		Here, the convention $P_{-n}^{(2\delta-1,1)}(y)=0$, for all real $y\in [-1,1]$ and integers $n \geq 0$, is used. 
		\item The orthogonality of the Jacobi polynomials \cite{MR0372517},
		\begin{align}\label{OrthogonolityOfJacobiPolynomials}
			\int_{-1}^{1}  P_n^{(2 \delta -1,1)}(y)P_m^{(2 \delta -1,1)}(y)  (1-y)^{2 \delta -1}(1+y)dy=	\xi_{n} \mathds{1}(n=m),
		\end{align}
		valid for all integers $n \geq 0$ and $m \geq 0$, where the coefficients are given by
		\begin{align*}
			\xi_{n}= \frac{2^{2 \delta +1} (n+1)}{(2 \delta +n) (2 \delta +2 n+1)}.
		\end{align*}
	\end{itemize}
The following result establishes a computationally efficient formula for $\mathfrak{ C}_{00mm}$.
 \begin{lemma}[Computationally efficient formula for $\mathfrak{ C}_{00mm}$]\label{LemmaComputationallyEfficientFormulaforC00mmWM}
	Fix $\delta  $ satisfying Assumption \ref{AssumptionsondeltaExistence} and let $m \geq 0$ be any integer. Then,
	\begin{align*}
		\mathfrak{ C}_{00mm} &=\sum_{k=0}^{\min \{ {\delta-1,m} \} } \mathfrak{ S}_{m}(k ), 
	\end{align*}
	where
	\begin{align*}
		\mathfrak{ S}_{m}(k ) &=\frac{2 (\delta +1) (\delta +2) \Gamma (\delta )^2}{\Gamma (k+1) \Gamma (k+3)} \frac{\Gamma (m+1) \Gamma (m+2)}{\Gamma (\delta -k) \Gamma (-k+\delta +2)}  \frac{(2 \delta -2 k+2 m+1) \mathfrak{ V}_{m}(k)}{\Gamma (-k+m+1) \Gamma (-k+m+2)}
		\cdot \\
		& \cdot \frac{\delta +2 m+2}{\Gamma (m+\delta +1) \Gamma (m+\delta +2)} \frac{\Gamma (-k+m+2 \delta ) \Gamma (-k+m+2 \delta +1)}{\Gamma (2 (m+\delta +1)-k)} \cdot \\
		& \cdot \frac{\Gamma (-k+2 m+\delta ) \Gamma (-k+2 m+\delta +2)}{\Gamma (2 (m+\delta +2)-k)}.
	\end{align*}
	Here, $\mathfrak{ V}_{m}(k)$ is a polynomial with respect to $k$ of degree $8$,
	\begin{align*}
		\mathfrak{ V}_{m}(k) &=k^8+\sum_{j=0}^{7} \mathfrak{ v}_{j}(m) k^j,
	\end{align*}
	where the coefficients are given by
	\begin{align*}
		\mathfrak{ v}_0 (m) &= 12 \delta ^4+24 \delta ^3+12 \delta ^2+\delta ^4 m^4-10 \delta ^3 m^4+11 \delta ^2 m^4-8 \delta ^4 m^3+80 \delta ^2 m^3-\delta ^4 m^2+70 \delta ^3 m^2 \\
		&+145 \delta ^2 m^2+22 \delta  m^4+72 \delta  m^3+74 \delta  m^2+20 \delta ^4 m+84 \delta ^3 m+88 \delta ^2 m+24 \delta  m, \\
		\mathfrak{ v}_1(m) &= -16 \delta ^4-72 \delta ^3-80 \delta ^2-24 \delta -8 \delta ^3 m^4+48 \delta ^2 m^4-8 \delta ^4 m^3+80 \delta ^3 m^3+92 \delta ^2 m^3+36 \delta ^4 m^2 \\
		&+122 \delta ^3 m^2-150 \delta ^2 m^2-8 \delta  m^4-156 \delta  m^3-316 \delta  m^2-32 m^4-96 m^3-88 m^2+28 \delta ^4 m \\
		&-70 \delta ^3 m -274 \delta ^2 m-184 \delta  m-24 m, \\
		\mathfrak{ v}_2(m) &= -40 \delta ^4-56 \delta ^3+46 \delta ^2+66 \delta +24 \delta ^2 m^4+52 \delta ^3 m^3-192 \delta ^2 m^3+24 \delta ^4 m^2 \\
		&-162 \delta ^3 m^2-528 \delta ^2 m^2-72 \delta  m^4-260 \delta  m^3-126 \delta  m^2-16 m^4+16 m^3+100 m^2-48 \delta ^4 m \\
		&-302 \delta ^3 m-210 \delta ^2 m+136 \delta  m+76 m+12  ,\\
		\mathfrak{ v}_3(m) &= 16 \delta ^4+128 \delta ^3+152 \delta ^2+14 \delta -120 \delta ^2 m^3-120 \delta ^3 m^2+156 \delta ^2 m^2-32 \delta  m^4+136 \delta  m^3 \\
		&+ 500 \delta  m^2+32 m^4+128 m^3+112 m^2-32 \delta ^4 m+80 \delta ^3 m+488 \delta ^2 m+312 \delta  m-4 m-16  ,\\
		\mathfrak{ v}_4(m) &= 16 \delta ^4+8 \delta ^3-106 \delta ^2-100 \delta +198 \delta ^2 m^2+112 \delta  m^3+6 \delta  m^2+16 m^4-16 m^3-124 m^2 \\
		&+112 \delta ^3 m+36 \delta ^2 m-244 \delta  m-96 m-11  ,\\
		\mathfrak{ v}_5(m) &= -32 \delta ^3-36 \delta ^2+30 \delta -120 \delta  m^2-32 m^3-24 m^2-120 \delta ^2 m-72 \delta  m+36 m+20  ,\\
		\mathfrak{ v}_6(m) &= 24 \delta ^2+22 \delta +24 m^2+52 \delta  m+20 m-2  ,\\
		\mathfrak{ v}_7(m) &= -8 \delta -8 m-4  .
	\end{align*}
\end{lemma}
\begin{proof}
Fix $\delta  $ satisfying Assumption \ref{AssumptionsondeltaExistence} and let $m \geq 0$ be any integer. Then, we use the connection formula	\eqref{ConnectionFormulaJacobi}, to obtain
	\begin{align*}  
	& \mathfrak{ C}_{00 mm}      
	 =\frac{1}{4^{ 2+\delta }}  \mathfrak{  N}_{0} ^2    \mathfrak{ N}_{m} ^2 \int_{-1}^{1}  
	\left( P_{m}^{\left(\delta,1 \right)}(y)\right)^2 (1-y)^{  2 \delta -1} (1+y)^{3}  dy \\
	&=\frac{1}{4^{ 2+\delta }}   \mathfrak{  N}_{0} ^2    \mathfrak{ N}_{m} ^2
	\sum_{k=0}^m L_m(k) \sum_{l=0}^m L_m(l)
	\int_{-1}^{1}  
	 P_{k}^{\left(2\delta-1,1 \right)}(y)
	 P_{l}^{\left(2\delta-1,1 \right)}(y)  (1-y)^{  2 \delta -1} (1+y)^{3}  dy \\
	 &=\frac{1}{4^{ 2+\delta }}  \mathfrak{  N}_{0} ^2    \mathfrak{ N}_{m} ^2 
	\sum_{k=0}^m L_m(k) \sum_{l=0}^m L_m(l)
	\int_{-1}^{1}  (y^2+2y+1)
	 P_{k}^{\left(2\delta-1,1 \right)}(y)
	 P_{l}^{\left(2\delta-1,1 \right)}(y)  (1-y)^{  2 \delta -1} (1+y)  dy \\
	 &=\frac{1}{4^{ 2+\delta }}  \mathfrak{  N}_{0} ^2    \mathfrak{ N}_{m} ^2 
	\sum_{k=0}^m L_m(k) \sum_{l=0}^m L_m(l)
	(\mathfrak{I}(k,l)+\mathfrak{J}(k,l)+\mathfrak{K}(k,l)) ,
\end{align*}
where we set
\begin{align*}
	\mathfrak{I}(k,l) &= \int_{-1}^{1}  y^2
	 P_{k}^{\left(2\delta-1,1 \right)}(y)
	 P_{l}^{\left(2\delta-1,1 \right)}(y)  (1-y)^{  2 \delta -1} (1+y)  dy   ,\\
	\mathfrak{J}(k,l) &= 2\int_{-1}^{1}  y
	 P_{k}^{\left(2\delta-1,1 \right)}(y)
	 P_{l}^{\left(2\delta-1,1 \right)}(y)  (1-y)^{  2 \delta -1} (1+y)  dy   ,\\
	\mathfrak{K}(k,l) &=  \int_{-1}^{1} 
	 P_{k}^{\left(2\delta-1,1 \right)}(y)
	 P_{l}^{\left(2\delta-1,1 \right)}(y)  (1-y)^{  2 \delta -1} (1+y)  dy  .
\end{align*}
The orthogonality of the Jacobi polynomials \eqref{OrthogonolityOfJacobiPolynomials} immediately yields
\begin{align*}
	\mathfrak{K}(k,l) &=\xi_{k} \mathds{1}(k=l).
\end{align*}
Now, we use the recurrence relation \eqref{RecurrenceRelation1} to obtain
\begin{align*}
	\mathfrak{J}(k,l) &= 2\int_{-1}^{1}  y
	 P_{k}^{\left(2\delta-1,1 \right)}(y)
	 P_{l}^{\left(2\delta-1,1 \right)}(y)  (1-y)^{  2 \delta -1} (1+y)  dy \\
	 &= 2P_k \int_{-1}^{1}  
	  P_{k+1}^{(2 \delta -1,1)}(y)
	 P_{l}^{\left(2\delta-1,1 \right)}(y)  (1-y)^{  2 \delta -1} (1+y)  dy \\
	 &+ 2Q_k \int_{-1}^{1}  
	  P_k^{(2 \delta -1,1)}(y)
	 P_{l}^{\left(2\delta-1,1 \right)}(y)  (1-y)^{  2 \delta -1} (1+y)  dy  \\
	 &+ 2R_k\int_{-1}^{1}  
	  P_{k-1}^{(2 \delta -1,1)}(y)
	 P_{l}^{\left(2\delta-1,1 \right)}(y)  (1-y)^{  2 \delta -1} (1+y)  dy \\
	 &= 2P_k \xi_{ k+1 }\mathds{1}( k+1=l )  
	 + 2Q_k \xi_{ k }\mathds{1}( k=l )  
	 + 2R_k \xi_{ k-1 }\mathds{1}( k-1=l ) .
\end{align*}
Similarly, we use the recurrence relation \eqref{RecurrenceRelation2} to obtain
\begin{align*}
	\mathfrak{I}(k,l) &= \int_{-1}^{1}  y^2
	 P_{k}^{\left(2\delta-1,1 \right)}(y)
	 P_{l}^{\left(2\delta-1,1 \right)}(y)  (1-y)^{  2 \delta -1} (1+y)  dy \\
	 &=U_k  \int_{-1}^{1}  P_{k+2}^{(2 \delta -1,1)}(y)
	 P_{l}^{\left(2\delta-1,1 \right)}(y)  (1-y)^{  2 \delta -1} (1+y)  dy \\
	 &+V_k \int_{-1}^{1}  P_{k+1}^{(2 \delta -1,1)}(y)
	 P_{l}^{\left(2\delta-1,1 \right)}(y)  (1-y)^{  2 \delta -1} (1+y)  dy \\
	 &+W_k \int_{-1}^{1}  P_k^{(2 \delta -1,1)}(y)
	 P_{l}^{\left(2\delta-1,1 \right)}(y)  (1-y)^{  2 \delta -1} (1+y)  dy \\
	 &+X_k \int_{-1}^{1}  P_{k-1}^{(2 \delta -1,1)}(y)
	 P_{l}^{\left(2\delta-1,1 \right)}(y)  (1-y)^{  2 \delta -1} (1+y)  dy\\
	 &+Y_k \int_{-1}^{1}   P_{k-2}^{(2 \delta -1,1)}(y)
	 P_{l}^{\left(2\delta-1,1 \right)}(y)  (1-y)^{  2 \delta -1} (1+y)  dy  \\
	 &=U_k \xi_{k+2} \mathds{1}( k+2=l ) 
	 +V_k \xi_{k+1} \mathds{1}( k+1=l )  
	 +W_k \xi_{k} \mathds{1}( k=l )   \\
	 &+X_k \xi_{k-1} \mathds{1}( k-1=l )  
	 +Y_k \xi_{k-2} \mathds{1}( k-2=l ) .  
\end{align*}
Finally, putting all together yields
\begin{align*}
	\mathfrak{C}_{00 mm}      
	 &=\frac{1}{4^{ 2+\delta }}  \mathfrak{ N}_{0} ^2    \mathfrak{N}_{m} ^2 
	\sum_{k=0}^m L_m(k) \sum_{l=0}^m L_m(l)
	(\mathfrak{I}(k,l)+\mathfrak{J}(k,l)+\mathfrak{K}(k,l)) \\
	&=\frac{1}{4^{ 2+\delta }}  \mathfrak{ N}_{0} ^2    \mathfrak{N}_{m} ^2 
	\sum_{k=0}^m L_m(k) \Bigg[  
	L_m(k+2) U_k \xi_{k+2}  
	 +L_m(k+1)V_k \xi_{k+1}    
	 +L_m(k)W_k \xi_{k}  \\  
	 &+L_m(k-1)X_k \xi_{k-1}    
	 +L_m(k-2)Y_k \xi_{k-2}
	+  2L_m(k+1)P_k \xi_{ k+1 }  \\
	 & + 2L_m(k)Q_k \xi_{ k }  
	 + 2L_m(k-1)R_k \xi_{ k-1 } 
	+L_m(k)\xi_{k}  
	\Bigg],
\end{align*}
that completes the proof.
\end{proof} 
Next, Lemma \ref{LemmaComputationallyEfficientFormulaforC00mmWM} allows us to prove  a linear recurrence relation for the Fourier coefficients.
\begin{lemma}[Recurrence relation for $\mathfrak{C}_{00mm}$]\label{LemmaRecurrenceRelationWM}
	Fix $\delta  $ satisfying Assumption \ref{AssumptionsondeltaExistence}. Then, the sequences
	\begin{align*}
		\left\{ \mathfrak{f}_m = \frac{\mathfrak{C}_{00mm}}{\ssomega_{m}^2} : m \geq 1 \right \}
	\end{align*}
	satisfy the 2-step recurrence relation
	\begin{align*}
		\mathfrak{P}_m  \mathfrak{f}_m - \mathfrak{Q}_m  \mathfrak{f}_{m+1} +\mathfrak{R}_m  \mathfrak{f}_{m+2} =0,
	\end{align*}
	where
	\begin{align*}
		\mathfrak{P}_m &= (m+1) (m+2) (2 m+1) (\delta +m+2) (\delta +2 m+2) (\delta +2 m+5) \big(3 \delta +2 m^2 \\
		&+2 \delta  m+10 m+11\big)  , \\
		\mathfrak{Q}_m &= (\delta +2 m+4)^2 \big(3 \delta ^4+29 \delta ^3+117 \delta ^2+255 \delta +28 \delta ^2 m^4+16 \delta ^3 m^3+188 \delta ^2 m^3+4 \delta ^4 m^2 \\
		&+70 \delta ^3 m^2 +439 \delta ^2 m^2+24 \delta  m^5+222 \delta  m^4+780 \delta  m^3+1286 \delta  m^2+8 m^6+96 m^5+462 m^4 \\
		&+1136 m^3+1493 m^2 +8 \delta ^4 m+89 \delta ^3 m+410 \delta ^2 m+973 \delta  m+980 m+244\big)  , \\
		\mathfrak{R}_m &=  (m+2) (\delta +m+2) (\delta +m+3) (\delta +2 m+3) (\delta +2 m+6) (2 \delta +2 m+7)\cdot \\
		& \cdot  \left(\delta +2 m^2+2 \delta  m+6 m+3\right) .
	\end{align*}
\end{lemma}
\begin{proof}
	The proof is similar to the one of Lemma \ref{LemmaRecurrenceRelationKG}. 
\end{proof}

Now, with the recurrence relation of Lemma \ref{LemmaRecurrenceRelationWM} at hand,  the monotonicity of the Fourier coefficients follows easily.

\begin{lemma}[Monotonicity for $\mathfrak{C}_{00mm}$]\label{LemmaMonotonicityC00mmWM}
	Fix $\delta  $ satisfying Assumption \ref{AssumptionsondeltaExistence}. Then, the sequences
	\begin{align*}
		\left\{ \mathfrak{f}_m = \frac{\mathfrak{C}_{00mm}}{\ssomega_{m}^2} : m \geq 1 \right \}
	\end{align*}
	are strictly decreasing with respect to $m$.
\end{lemma}
\begin{proof}
Fix $\delta  $ satisfying Assumption \ref{AssumptionsondeltaExistence} and pick an integer $m \geq 1$. Then, we set
	\begin{align*}
		\mathfrak{x}_m = \frac{\mathfrak{f}_{m+1}}{\mathfrak{f}_m},
	\end{align*}
	and Lemma \ref{LemmaRecurrenceRelationWM} yields the recurrence relation
	\begin{align*}
		\mathfrak{x}_{m+1} =\mathfrak{A}_m - \frac{\mathfrak{B}_m}{\mathfrak{x}_m},
	\end{align*}
	where
	\begin{align*}
		\mathfrak{A}_m = \frac{\mathfrak{Q}_m}{\mathfrak{R}_m}, \quad \mathfrak{B}_m = \frac{\mathfrak{P}_m}{\mathfrak{R}_m}.
	\end{align*}
In the Appendix (Lemmata \ref{Lemma1AuxiliarycomputationsWMAppendix} and \ref{Lemma2AuxiliarycomputationsWMAppendix}), we show that
\begin{align*}
	\mathfrak{R}_m >0, \quad \mathfrak{B}_m >0, \quad \mathfrak{A}_m- \mathfrak{B}_m <1, \quad \mathfrak{x}_1<1.
\end{align*}
for all integers $m \geq 1$ and $\delta \geq 1$. The rest of the proof is similar to the one of Lemma \ref{LemmaMonotonicityC00mmKG}.
\end{proof}

Finally, we prove Proposition \ref{PropositionUniformEstimates} for the WM. 

\begin{proof}[Proof of Proposition \ref{PropositionUniformEstimates} for the WM]
	Fix $\delta   $ satisfying the Assumption \ref{AssumptionsondeltaExistence} and pick an integer $m \geq 1$. Then, Lemma \ref{LemmaMonotonicityC00mmKG} yields that	the sequences
	\begin{align*}
		\left \{\frac{ \mathfrak{C}_{00 00}}{\ssomega_{0}^2}-2\frac{ \mathfrak{C}_{00 mm}}{\ssomega_{m}^2} :m \geq 1 \right\}
	\end{align*}
	are all strictly increasing and hence
	\begin{align*} 
		 \min_{m \geq 1} \left(	\frac{ \mathfrak{C}_{00 00}}{\ssomega_{0}^2}-2\frac{ \mathfrak{C}_{00 mm}}{\ssomega_{m}^2} \right)= 	\frac{ \mathfrak{C}_{00 00}}{\ssomega_{0}^2}-2\frac{ \mathfrak{C}_{00 11}}{\ssomega_{1}^2} .
	\end{align*}
	Using Lemma \ref{LemmaComputationallyEfficientFormulaforC00mmKG} to compute $\mathfrak{C}_{0000}$ and $\mathfrak{C}_{0011}$, we compute
	\begin{align*}  
	\frac{ \mathfrak{C}_{00 00}}{\ssomega_{0}^2}-2\frac{ \mathfrak{C}_{00 11}}{\ssomega_{1}^2}  = 	\frac{9 \delta +21}{8 \delta ^4+68 \delta ^3+190 \delta ^2+199 \delta +60}.
		\end{align*} 
	and observe that
	\begin{align*}
		& \left \{\frac{ \mathfrak{C}_{00 00}}{\ssomega_{0}^2}-2\frac{ \mathfrak{C}_{00 mm}}{\ssomega_{m}^2} :m \geq 1 \right\} \subset (0,\infty	)  \Longleftrightarrow
		 \min_{m \geq 1} \left(	\frac{ \mathfrak{C}_{00 00}}{\ssomega_{0}^2}-2\frac{ \mathfrak{C}_{00 mm}}{\ssomega_{m}^2} \right) >0  \\
& \Longleftrightarrow \frac{ \mathfrak{C}_{00 00}}{\ssomega_{0}^2}-2\frac{ \mathfrak{C}_{00 11}}{\ssomega_{1}^2}  >0 
 \Longleftrightarrow   1\leq  \delta < \infty.
	\end{align*}
	 Finally, in this case, we set
	\begin{align*}
		\mathfrak{c}_{\perp}(\delta)= 	\frac{ \mathfrak{C}_{00 00}}{\ssomega_{0}^2}-2\frac{ \mathfrak{C}_{00 11}}{\ssomega_{1}^2}   >0,
	\end{align*}
	that completes the proof.
\end{proof}

Finally, we note the following remarks.

\begin{remark}[Closed formulas for $\mathfrak{C}_{00mm}$]\label{RemarkClosedFormulasMathfrakC00mmModelWM}
	One can solve the recurrence relation of Lemma \ref{LemmaRecurrenceRelationWM}  to find closed formulas for $\mathfrak{C}_{00mm}$ provided that $\delta$ is fixed. For example, for $\delta \in \{1,2,3,4\}$, we find
	\begin{align*}
		\mathfrak{C}_{00mm} &= \frac{18 \left(m^2+3 m+1\right)}{4 m^2+12 m+5} , \\
		\mathfrak{C}_{00mm} &= \frac{12 (m (m+4) (2 m (m+4)+17)+18)}{(2 m+1) (2 m+3) (2 m+5) (2 m+7)} , \\
		\mathfrak{C}_{00mm} &= \frac{5 (m (m+5) (m (m+5) (3 m (m+5)+50)+312)+360)}{(m+2) (m+3) (2 m+1) (2 m+3) (2 m+7) (2 m+9)}, \\
		\mathfrak{C}_{00mm} &= \frac{45 (m (m+6) (m (m+6) (m (m+6) (2 m (m+6)+61)+703)+4004)+5040)}{2 (m+2) (m+4) (2 m+1) (2 m+3) (2 m+5) (2 m+7) (2 m+9) (2 m+11)}, 
	\end{align*}
	respectively, for all integers $m\geq \delta -1$.  
\end{remark}

\begin{remark}[Uniform asymptotic expansion for $\mathfrak{ C}_{00mm}$ for large $m$]
	We note that, due to the computationally efficient formula (Lemma \ref{LemmaComputationallyEfficientFormulaforC00mmWM}), one can easily derive the limit of $\mathfrak{ C}_{00mm} $ as $m \rightarrow \infty$ uniformly with respect to any fixed $\delta$. Indeed, for $m \geq  \delta-1$, we have that
	\begin{align*}
		\mathfrak{ C}_{00mm} &=\sum_{k=0}^{ \delta-1  } \mathfrak{ S}_{m}(k ).
	\end{align*}
	Using Stirling's formula, we expand in series the Gamma functions in $ \mathfrak{ S}_{m}(k )$ and obtain
	\begin{align*}
		 \mathfrak{ S}_{m}(k )=\frac{(\delta +1) (\delta +2) \Gamma (\delta )^2}{2^{2 \delta +1}} \frac{T(k)}{\Gamma (k+1) \Gamma (k+3) \Gamma (\delta -k) \Gamma (-k+\delta +2)} + \mathcal{O} \left(\frac{1}{m} \right),
	\end{align*}
	as $m \rightarrow \infty$. Here, $T(k)$ is a polynomial with respect to $k$ of degree 4, 
	\begin{align*}
		T(k)= 16k^4 + \sum_{j=0}^3 t_j k^j,
	\end{align*}
	where its coefficients are given by
	\begin{align*}
		t_0 &= \delta ^4-10 \delta ^3+11 \delta ^2+22 \delta,  \quad 
		t_1 = -8 \delta ^3+48 \delta ^2-8 \delta -32 , \\
		t_2 &= 24 \delta ^2-72 \delta -16 , \quad 
		t_3 = 32-32 \delta .
	\end{align*}
	Since $\delta \geq 1$ is fixed, we infer
	\begin{align*}
		\mathfrak{ C}_{00mm}  
		& = \mathfrak{ C}_{\infty}(\delta)+ \mathcal{O} \left(\frac{1}{m} \right),
	\end{align*}
	as $m \rightarrow \infty$, where the leading order term is given by
	\begin{align*}
		\mathfrak{ C}_{\infty}(\delta)= \frac{(\delta +1) (\delta +2) \Gamma (\delta )^2}{2^{2 \delta +1}}\sum_{k=0}^{ \delta-1  } \frac{T(k)}{\Gamma (k+1) \Gamma (k+3) \Gamma (\delta -k) \Gamma (-k+\delta +2)}.
	\end{align*}
	Now, we claim that this can be found in closed formula uniformly with respect to a fixed $\delta$. Indeed, Zeilberger's algorithm yields that it satisfies the 1-step recurrence relation 
	\begin{align*}
	\begin{dcases}
	\mathfrak{ C}_{\infty}(1)=\frac{9}{2}	, \\
	\mathfrak{ C}_{\infty}(\delta+1)= \frac{(\delta +3) (2 \delta -1)}{2 (\delta +1) (\delta +2)} \mathfrak{ C}_{\infty}(\delta), \quad \delta =1,2,\dots.
	\end{dcases}
	\end{align*}
	This follows by a recurrence relation for the summand in $\mathfrak{ C}_{\infty}(\delta)$ that can be rigorously proved as explained in Section \ref{SectionGeneralStrategy}. Finally, one can easily solve the recurrence relation above to deduce that
	\begin{align}\label{LimitValueMathfrakC00mm}
		\mathfrak{ C}_{\infty}(\delta)= \frac{9}{2} \prod_{j=1}^{\delta-1} \frac{(j+3) (2 j -1)}{2 (j+1) (j +2)} =\frac{3 (\delta +2) \Gamma \left(\delta -\frac{1}{2}\right)}{2 \sqrt{\pi } \Gamma (\delta +1)}.
	\end{align}
Figures \ref{Figure1MathfrakC00mm} and \ref{Figure2MathfrakC00mm} illustrate the Fourier coefficients $\mathfrak{ C}_{00mm} $ and their limiting values $\mathfrak{ C}_{\infty}(\delta)$ given by \eqref{LimitValueMathfrakC00mm}.
\begin{figure}[h]
\centering
\includegraphics[width=.55\textwidth]{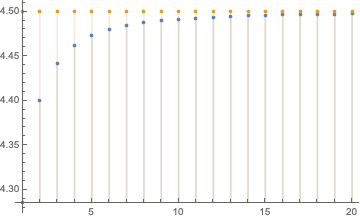}
\caption{The Fourier coefficient $\mathfrak{C}_{00mm}$ for $\delta=1$. It is strictly increasing with respect to $m$ (Lemma \ref{LemmaComputationallyEfficientFormulaforC00mmWM}) and approaches the limit value $\mathfrak{C}_{\infty}(1)=9/2$ given by \eqref{LimitValueMathfrakC00mm}. It is increasing only for $\delta=1$ (compare with Figure \ref{Figure2MathfrakC00mm}). This is not in contrast with Lemma \ref{LemmaMonotonicityC00mmWM} since the rescaled Fourier coefficient $\mathfrak{C}_{00mm}/\ssomega_{m}^2$ for $\delta=1$ is strictly decreasing with respect to $m$.  }\label{Figure1MathfrakC00mm}
\bigbreak
\includegraphics[width=.55\textwidth]{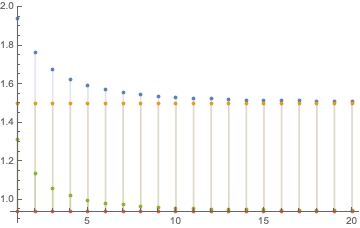}
\caption{The Fourier coefficients $\mathfrak{C}_{00mm}$ for $\delta=2$ (blue) and $\delta=3$ (green). They are both strictly decreasing with respect to $m$ (Lemma \ref{LemmaComputationallyEfficientFormulaforC00mmWM}) and approach the limit values $\mathfrak{C}_{\infty}(2)=3/2$ (orange) and $\mathfrak{C}_{\infty}(3)=15/16$ (red) respectively given by \eqref{LimitValueMathfrakC00mm}. Both $\mathfrak{C}_{00mm}$ and $\mathfrak{C}_{00mm}/\ssomega_{m}^2$ are strictly decreasing for all $\delta \geq 2$. (Lemma \ref{LemmaComputationallyEfficientFormulaforC00mmWM}).}\label{Figure2MathfrakC00mm}
\end{figure}
\end{remark}

 \section{Existence of time-periodic solutions to the non-linear wave equation} \label{SectionExistenceofTimePeriodicSolutions}

In this section, we establish the existence of time-periodic solutions to the non-linear wave equation  
\begin{align} \label{NonLinearWaveequationSectionExistenceofTimePeriodicSolutions}
	\partial_{t}^2 \phi +L \phi = -\epsilon^2 W \phi^3 + \epsilon^4 E(\phi),
\end{align}
for a scalar field $\phi: \mathbb{R} \times (0,\pi/2) \rightarrow \mathbb{R}$, where, using the notation from Table \ref{NotationTable}, $L$, $W$ and $E$ are given by \eqref{MainPDEIntroModelKGA}-\eqref{MainPDEIntroModelKGB} and \eqref{MainPDEIntroModelWMA}-\eqref{MainPDEIntroModelWMB} with $E=0$ in the case of the KG model. To do so, we will apply Theorem \ref{TheoremBambusiPaleari} to an appropriate rescale of the first 1-mode and verify the conditions \eqref{Assumption1BambusiPaleari} and \eqref{Assumption2BambusiPaleari} of Theorem \ref{TheoremBambusiPaleari}. For $\zeta=\phi(0) \in H^s ((0,\pi/2); d\mu)$,  we denote by 
\begin{align*}
	\Phi ^{t} (\zeta) = \Phi ^{t} \left(\sum_{m=0}^{\infty} \zeta^{(m)} e_{m} \right) = \sum_{m=0}^{\infty} \zeta^{(m)}\cos(\omega_{m}t)e_{m} 
\end{align*}
the solution to the initial value problem consisting of the linearized equation in \eqref{NonLinearWaveequationSectionExistenceofTimePeriodicSolutions} coupled to $\zeta$ as initial data and zero initial velocity,  
\begin{align*} 
\begin{cases}
	\partial_{t}^2 \phi_{\text{linear}} (t,\cdot) +L \phi_{\text{linear}}(t,\cdot) =0, \quad t \in \mathbb{R}, \\
 \phi_{\text{linear}}(0,\cdot)=\zeta,  \quad \partial_{t} \phi_{\text{linear}} (0,\cdot)= 0.
\end{cases}
\end{align*}
Furthermore, we define the operator $\mathcal{M}$ from Theorem \ref{TheoremBambusiPaleari}, 
\begin{align*} 
	 \mathcal{M}  \left(\zeta  \right) : =   L \zeta   + \langle f^{(3)}  \rangle(\zeta ), \quad \langle f^{(3)}  \rangle(\zeta )=  \frac{1}{2\pi } \int_{0}^{2\pi } \Phi ^{-t} \left[f^{(3)} \left(\Phi ^{t} (\zeta  ) \right) 
	\right]  dt ,
\end{align*} 	
where $f^{(3)}$ stands for the leading order non-linearity in \eqref{NonLinearWaveequationSectionExistenceofTimePeriodicSolutions}. Recall that this is given in the Fourier space by \eqref{DefinitionfToThe3},
\begin{align*}
	f^{(3)}(  \phi ) =\sum_{m=0}^{\infty} \left( f^{(3)}(  \phi ) \right)^{(m)} e_m, \quad \left( f^{(3)}(  \phi ) \right)^{(m)}= - \sum_{i,j,k=0}^{\infty}C_{ijkm}   \phi^{(i)}(t) \phi^{(j)}(t) \phi^{(k)}(t).
\end{align*} 
In addition, we define the rescaled first 1-mode by
\begin{align}\label{Definition1modeinitialdata}
	\xi=\sum_{m=0}^{\infty}\xi^{(m)} e_m, \quad  \xi^{(m)} =  \kappa_{0}      \mathds{1}(m=0), \quad 
	  \kappa_{0}      = \pm 2\omega_{0}   \sqrt{\frac{2}{3C_{00 0 0}  }} .
\end{align}
To begin with, we establish the condition \eqref{Assumption1BambusiPaleari} in Theorem \ref{TheoremBambusiPaleari}.

\begin{prop}[Condition \eqref{Assumption1BambusiPaleari} in Theorem \ref{TheoremBambusiPaleari}]\label{PropositionAssumption1inBambusiPaleari}
Fix $\delta$ satisfying Assumption \ref{AssumptionsondeltaExistence} and let $\xi $ be the rescaled initial data defined by \eqref{Definition1modeinitialdata}. Then, 
\begin{align*} 
	  \mathcal{M}(\xi)=0.
\end{align*} 
\end{prop} 
	\begin{proof}
Fix $\delta$ satisfying Assumption \ref{AssumptionsondeltaExistence}, let $\xi $ be the rescaled initial data defined by \eqref{Definition1modeinitialdata} and pick any integer $m \geq 0$. Then, we compute
\begin{align*}
(  L \xi  )^{(m)} &=  \omega_{m}    ^2 \xi ^{(m)}  =  \kappa_{0}      \omega_{0}   ^2 \mathds{1}(m=0), \\
	( \Phi ^{t} (\xi  ) )^{(m)} &=\xi ^{(m)}\cos \left( \omega_{m}  t \right)=
	  \kappa_{0}     \cos \left( \omega_{0}   t \right) \mathds{1}(m=0), \\
\left ( f^{(3)}\left(\Phi ^{t} (\xi ) \right) \right )^{(m)} & = 
	 - \sum_{i,j,k}  C_{ijkm}  ( \Phi ^{t} (\xi  ) )^i 
	 ( \Phi ^{t} (\xi  ) )^j ( \Phi ^{t} (\xi  ) )^k  
	 = -  \kappa_{0}  ^3 C_{0 0 0 m}     \cos^3  (  \omega_{0}  t )  , \\
 \left( \Phi ^{t} \left[f^{(3)} \left(\Phi ^{t} (\xi ) \right) 
	\right] \right)^{(m)} &=  ( f\left(\Phi ^{t} (\xi ) \right) )^{(m)}  \cos (\omega_{m}  t ) 
	 = -  \kappa_{0}  ^3 C_{0 0 0 m}     \cos^3  (  \omega_{0}  t )\cos (\omega_{m}  t )   , \\
( \langle f^{(3)} \rangle )^{(m)}  &= -  \kappa_{0}  ^3 C_{0 0 0 m}      \frac{1}{2\pi } \int_{0}^{2\pi }  \cos^3  (  \omega_{0}  t )\cos (\omega_{m}  t ) dt \\
&=-  \kappa_{0}  ^3 C_{0 0 0 m}      \frac{1}{2\pi } \int_{0}^{2\pi }  \Big[ \frac{1}{8} \cos((\omega_{m}  -3\omega_{0}  )t)+\frac{3}{8} \cos((\omega_{m}  -\omega_{0}  )t) \\
&+ \frac{3}{8} \cos((\omega_{m}  +\omega_{0}  )t)
+\frac{1}{8} \cos((\omega_{m}  +3\omega_{0}  )t) \Big] dt \\
&=-  \kappa_{0}  ^3 C_{0 0 0 m}    \Big[    \frac{1}{8} \mathds{1}(m=\delta )+\frac{3}{8}  \mathds{1}(m=0) \Big] \\
&=  - \frac{1}{8}  \kappa_{0}  ^3 C_{000\delta }      \mathds{1}(m=\delta )- \frac{3}{8}   \kappa_{0}  ^3 C_{0 0 0 0}    \mathds{1}(m=0)  ,
\end{align*}
where we used the fact that 
\begin{align*}
	\frac{1}{2\pi }\int_{0}^{2\pi} \cos(\Omega t) dt = \mathds{1}(\Omega=0)
\end{align*}
as well as
\begin{align*}
	& \omega_{m}  -3\omega_{0}  =0 \Longleftrightarrow  m=
	\begin{dcases}
		\delta , \quad \text{for the KG},\\
		\delta+2 , \quad \text{for the WM}
	\end{dcases},  \\
	& \omega_{m}  -\omega_{0}  =0 \Longleftrightarrow   m=0, \quad  \omega_{m}  +\omega_{0}   \neq 0, \Hquad  \omega_{m}  +3\omega_{0}   \neq 0.
\end{align*}
Notice that the resonant condition $\omega_{m}  -3\omega_{0}  =0 $ has only 1 minus sign, hence Lemma \ref{LemmaVanishingFourier} yields that $C_{000m } =0$ with $m=\delta$ for the KG and $m=\delta+2$ for the WM. Therefore,
\begin{align*}
	(\mathcal{M}\left(\xi \right) )^ m  & =  ( A \xi)^{(m)}  + (\langle f^{(3)} \rangle )^{(m)} =
	  \kappa_{0}     \Big[ \omega_{0}  ^2  -  \frac{3}{8}  \kappa_{0}  ^2 C_{0 0 0 0}       \Big] \mathds{1}(m=0)=0,
\end{align*}
that completes the proof. 
\end{proof}
 
Next, we derive the differential of $\mathcal{M}$ at the rescaled first 1-mode.

\begin{lemma}[Differential of $\mathcal{M}$]\label{LemmaDifferentialofMat1modes}
Fix $\delta$ satisfying Assumption \ref{AssumptionsondeltaExistence} and let $\xi $ be the rescaled initial data defined by \eqref{Definition1modeinitialdata}. Then, for all $h \in H^2 ((0,\pi/2); d\mu)$,  we have that
\begin{align*} 
  (	d \mathcal{M}( \xi )[h] )^{(m)} =  
	 - 2\omega_{0}^2  h^{(0)} \mathds{1}( m =0   ) +\frac{\omega_{0}^2\omega_{m} ^2}{C_{00 0 0}}\left(   \frac{C_{00 0 0}}{\omega_{0}^2} 
     -2    \frac{  C_{0 0 mm}  }{\omega_{m} ^2}      
     \right)  h^{(m)} \mathds{1}( m \geq 1   ) .
 \end{align*} 

\end{lemma} 
\begin{proof}
Fix $\delta$ satisfying Assumption \ref{AssumptionsondeltaExistence} and let $\xi $ be the rescaled initial data defined by \eqref{Definition1modeinitialdata}. Also, let $\epsilon >0$, $h \in H^2 ((0,\pi/2); d\mu)$ and pick any integer $m\geq 0$. Then,  we compute 
\begin{align*}
 (  L (\xi +\epsilon h)  )^{(m)} &= \omega_{m} ^2 (\xi^{(m)} +\epsilon h^{(m)}) =   (  L \xi  )^{(m)}  +\epsilon   \omega_{m} ^2 h^{(m)} , \\
	( \Phi ^{t} (\xi +\epsilon h  ) )^{(m)} &=(\xi^{(m)} +\epsilon h^{(m)})\cos(\omega_{m}  t)=( \Phi ^{t} (\xi  ) )^{(m)} + \epsilon h^{(m)} \cos(\omega_{m}  t), \\
 \left ( f^{(3)} \left(\Phi ^{t} (\xi +\epsilon h ) \right) \right )^{(m)} & = 
	 - \sum_{i,j,k}  C_{ijkm}    ( \Phi ^{t} (\xi +\epsilon h  ) )^i 
	 ( \Phi ^{t} (\xi  ) )^j ( \Phi ^{t} (\xi +\epsilon h  ) )^k \\ 
& = \left ( f^{(3)} \left(\Phi ^{t} (\xi   ) \right) \right )^{(m)}   -3\epsilon  \kappa_{0}  ^2   \sum_{ k}  C_{0 0 km}    h^{(k)}   \cos^2(\omega_{0}  t)    \cos ( \omega_{k}  t)   + \mathcal{O}(\epsilon^2) , \\
 \left( \Phi ^{t} \left[f^{(3)} \left(\Phi ^{t} (\xi+\epsilon h ) \right) 
	\right] \right)^{(m)} &=   
	 \left( \Phi ^{t} \left[f^{(3)} \left(\Phi ^{t} (\xi  ) \right) 
	\right] \right)^{(m)} \\
	& -3\epsilon  \kappa_{0}  ^2   \sum_{ i}  C_{i0 0 m}    h^{(i)}   \cos^2(\omega_{0}  t)    \cos ( \omega_{i}  t) \cos(\omega_{m}  t)   + \mathcal{O}(\epsilon^2)  , \\
  \left(\langle f^{(3)} \rangle(\xi+\epsilon h) \right)^{(m)} & =  \left(\langle f^{(3)} \rangle(\xi ) \right)^{(m)}  \\
  &- \frac{3\epsilon  \kappa_{0}  ^2 }{2\pi }  \sum_{ i}  C_{i0 0 m}    h^{(i)} \int_{0}^{2\pi }  \cos^2(\omega_{0}  t)    \cos ( \omega_{i}  t) \cos(\omega_{m}  t) dt  + \mathcal{O}(\epsilon^2),
 \end{align*}
where we also used the symmetries of the Fourier coefficients. Therefore, we infer
\begin{align*}
	& \left( d \langle f^{(3)} \rangle(\xi )  [h]\right)^{(m)} = - \frac{3   \kappa_{0}  ^2 }{2\pi }  \sum_{ i}  C_{i0 0 m}    h^{(i)} \int_{0}^{2\pi }  \cos^2(\omega_{0}  t)    \cos ( \omega_{i}  t) \cos(\omega_{m}  t) dt \\
& =- \frac{3   \kappa_{0}  ^2 }{16\pi }  \sum_{ i}  C_{i0 0 m}    h^{(i)} \sum_{\pm} \int_{0}^{2\pi } \cos(( \omega_{i}  \pm\omega_{0}  \pm \omega_{0}  \pm \omega_{m}  )t) dt \\
& = - \frac{3   \kappa_{0}  ^2 }{8 }   \sum_{ i}  C_{i0 0 m}    h^{(i)} \sum_{\pm} \mathds{1}(\omega_{i}  \pm \omega_{0} \pm \omega_{0}   \pm \omega_{m} =0) .  
\end{align*}
Firstly, we consider pairs $(i,0,0,m)$ such that $\omega_{i}  \pm\omega_{0}  \pm\omega_{0}  \pm\omega_{m}  =0$ with only one minus sign.  For all these, Lemma \ref{LemmaVanishingFourier} implies that $C_{i00 m}  =0$.  Secondly, we consider pairs $(i,0,0,m)$ such that   such that $\omega_{i} \pm\omega_{0} \pm\omega_{0} \pm\omega_{m} =0$ with only two minus signs. For both KG and WM, these are given by
\begin{align*}
\begin{dcases}
	\omega_{i}  +\omega_{0}  - \omega_{0}  - \omega_{m}   =0 \\
	\omega_{i}  -\omega_{0}  + \omega_{0}  - \omega_{m}   =0 \\
	\omega_{i}   -\omega_{0}  - \omega_{0}  + \omega_{m}   =0 \\
	\end{dcases} \Longleftrightarrow
	\begin{dcases}
	i=m \\
	i=m \\
	i= -m    \geq 0
	\end{dcases}
\end{align*}
Therefore, we obtain
\begin{align*}
	&\left( d \langle f^{(3)} \rangle(\xi )  [h]\right)^{(m)}=  - \frac{3   \kappa_{0}  ^2 }{8 }   \sum_{ i}  C_{i0 0 m}    h^{(i)} \sum_{\pm} \mathds{1}(\omega_{i}  \pm \omega_{0} \pm \omega_{0}   \pm \omega_{m} =0)\\ 
& =- \frac{3   \kappa_{0}  ^2 }{8 } \left[  
\sum_{ i}  C_{i0 0 m}    h^{(i)} \mathds{1}( i=m ) +
\sum_{ i}  C_{i0 0 m}    h^{(i)} \mathds{1}( i=m   ) +
\sum_{ i}  C_{i0 0 m}    h^{(i)} \mathds{1}( i=-m \geq0   )  \right]  \\
& =- \frac{3   \kappa_{0}  ^2 }{8 } \left[  
2\sum_{ i}  C_{i0 0 m}    h^{(i)} \mathds{1}( i=m ) + 
\sum_{ i}  C_{i0 0 m}    h^{(i)} \mathds{1}( i=m =0   )  \right]  \\
& =- \frac{3   \kappa_{0}  ^2 }{8 } \left[  
2  C_{m0 0 m}    h^{(m)}  + 
   C_{m0 0 m}    h^{(m)} \mathds{1}( m =0   )  \right]\\
& =   -2\omega_{0}^2  \frac{   C_{0 0 mm}}{C_{00 0 0}  }    
   h^{(m)}  -\omega_{0}^2     h^{(0)} \mathds{1}( m =0   )  ,
\end{align*}
where we also used the definition of the constant $\kappa_0$ from  \eqref{Definition1modeinitialdata}. Finally, we obtain
\begin{align*}
	(	d \mathcal{M}( \xi )[h] )^{(m)} &=  \omega_{m} ^2 h^{(m)} +\left( d \langle f ^{(3)} \rangle(\xi )  [h] \right)^{(m)}\\
	&=  \omega_{m} ^2 h^{(m)}  -2\omega_{0}^2  \frac{   C_{0 0 mm}}{C_{00 0 0}  }    
   h^{(m)}  -\omega_{0}^2     h^{(0)} \mathds{1}( m =0   )  \\
    &=  \left[\omega_{0} ^2  - 2\omega_{0}^2  - \omega_{0}^2     
    \right] h^{(0)} \mathds{1}( m =0   ) +\left[ \omega_{m} ^2   -2\omega_{0}^2  \frac{   C_{0 0 mm}}{C_{00 0 0}  }    
     \right]h^{(m)} \mathds{1}( m \geq 1   ) \\
     &=   - 2\omega_{0}^2  h^{(0)} \mathds{1}( m =0   ) +\frac{\omega_{0}^2\omega_{m} ^2}{C_{00 0 0}}\left(   \frac{C_{00 0 0}}{\omega_{0}^2} 
     -2    \frac{  C_{0 0 mm}  }{\omega_{m} ^2}      
     \right)  h^{(m)} \mathds{1}( m \geq 1   ) ,
 \end{align*} 
that completes the proof. 
\end{proof}

  Next, we establish the condition \eqref{Assumption2BambusiPaleari} in Theorem \ref{TheoremBambusiPaleari}.

 \begin{prop}[Condition \eqref{Assumption2BambusiPaleari} in Theorem \ref{TheoremBambusiPaleari}]\label{PropositionAssumption2inBambusiPaleari}
Fix $\delta$ satisfying Assumption \ref{AssumptionsondeltaExistence} and let $\xi $ be the rescaled initial data defined by \eqref{Definition1modeinitialdata}.  Then,  
\begin{align*}
		 \ker \left( d \mathcal{M}( \xi ) \right) =\{ 0 \}.
	\end{align*} 
\end{prop}
\begin{proof}
	Fix $\delta$ satisfying Assumption \ref{AssumptionsondeltaExistence} and let $\xi $ be the rescaled initial data defined by \eqref{Definition1modeinitialdata}. Also, let $h \in H^4 ((0,\pi/2);d\mu) $ such that $d \mathcal{M}( \xi )[h]=0$ and pick an  integer $m \geq 0 $. Then, according to Lemma \ref{LemmaDifferentialofMat1modes}, we have that  
\begin{align}\label{Condition111Existence}
	\left( \frac{C_{00 0 0}}{\omega_{0}^2} 
     -2    \frac{  C_{0 0 mm}  }{\omega_{m} ^2} \right)h^{(m)}=0,
\end{align}
for all integers $m \geq 1$, and 
\begin{align}\label{Condition222Existence}
	 \omega_{0} ^2  h^{(0)} & = 0, 
\end{align}
for $m =0$. On the one hand, using the fact that $\omega_{0} \neq   0$, \eqref{Condition222Existence} yields $h^{(0)}=0$. On the other hand, using Proposition \ref{PropositionUniformEstimates}, \eqref{Condition111Existence} yields  $h^{(m)}=0$, for all integers $m \geq 1$. Consequently, the condition $\ker \left( d \mathcal{M}( \xi ) \right) =0$ implies $h=0$ and completes the proof.
\end{proof} 

Finally, we conclude this section by establishing the regularity conditions required in Theorem \ref{TheoremBambusiPaleari} for both models. 
To begin with, we consider the KG model given by \eqref{MainPDEIntroModelKGA}.

\begin{lemma}[Regularity conditions in Theorem \ref{TheoremBambusiPaleari} for the KG model]\label{LemmaRegularityConditionsKG}
	Fix $\delta$ satisfying Assumption \ref{AssumptionsondeltaExistence} and let
	\begin{align*} 
  f^{(3)} (\phi)=-  \mathsf{W} \phi^3  , \quad  f^{(4)} (\phi)=0,
\end{align*}
for all $\phi \in H^s ( (0,\pi/2);d\mu )$, with $d\mu=\mathsf{w} dx$, where $\mathsf{w}$ is given by \eqref{DefinitionWeightw} and $\mathsf{W}$ is given by \eqref{MainPDEIntroModelKGB}. Then, for all $s \ge 2$, we have that 
\begin{itemize}[leftmargin=*]
	\item $f^{(3)}:H^s ( (0,\pi/2);d\mu ) \rightarrow H^s ( (0,\pi/2);d\mu )$ is a bounded homogeneous polynomial of degree three,
	\item $f^{(3)}$ leaves invariant the domain $  H^{s+2} ( (0,\pi/2);d\mu )$. 
\end{itemize}  
\end{lemma}
\begin{proof}
Recall that $H^s( \mathbb{S}^3)$ is an algebra, for $s \ge 2 > 3/2$ and that $
\| u \|_{L^\infty(\mathbb{S}^3)} \lesssim \| u \|_{H^2(\mathbb{S}^3)}$, for all $u \in H^2(\mathbb{S}^3)$, by the Sobolev inequality. First, we consider the conformal case $\delta=2$. In this case, the original KG model on AdS for a scalar field $\psi$ is conformally equivalent to a KG equation on  $\mathbb{S}^3$ for the scalar field $u=\cos^{-1}(x) \psi$, with the conformal map preserving the $H^s$ regularity. The KG equation \eqref{MainPDEIntroModelKGA}-\eqref{MainPDEIntroModelKGB} is obtained by rescaling the unknown $\phi$ related to $u$ by $\phi=\cos^{-1}(x)u$. Then, for all $s \geq 2$, we use the algebra property to estimate  
\begin{align*}
	\| f^{(3)}(\phi) \|_{H^s ( (0,\pi/2);d\mu )} &= 
	\| \mathsf{W} \phi^3 \|_{H^s ( (0,\pi/2);d\mu )}  \simeq  
	\|  \cos(x) \mathsf{W} \phi^3 \|_{H^s(\mathbb{S}^3)} \\
	&= 
	\|  u^3 \|_{H^s(\mathbb{S}^3)} \lesssim 
	\|  u \|^3_{H^s(\mathbb{S}^3)}.
\end{align*}
Similarly, we can transfer the well-known bounds for $H^s(\mathbb{S}^3)$ to their AdS counterparts which completes the proof for the case $\delta=2$. Finally, for all $\delta \geq 2$, the $H^s ( (0,\pi/2);d\mu )$ norms are stronger than in the conformal case $\delta=2$, so the estimate above still holds true. We shall omit the details here since we explain both the conformal transformation and the comparison of the norms more precisely below in the proof of Hardy-Sobolev inequality for the KG model, see Lemma \ref{LemmaHardySolobelInequality}.
\end{proof}
Next, we consider the WM model given by \eqref{MainPDEIntroModelWMA}. 
\begin{lemma}[Regularity conditions in Theorem \ref{TheoremBambusiPaleari} for the WM model]\label{LemmaRegularityConditionsWM}
	Fix $\delta$ satisfying Assumption \ref{AssumptionsondeltaExistence} and let
	\begin{align*} 
  f^{(3)} (\phi)=-  \mathfrak{W} \phi^3  ,\quad  f^{(4)} (\phi)= \mathfrak{E} (\cdot,\phi)
\end{align*}
for all $\phi \in H^s ( (0,\pi/2);d\mu )$, with $d\mu= \mathfrak{w} dx$, where $\mathfrak{w}$ is given by \eqref{DefinitionWeightw} and both $\mathfrak{W}$ and $\mathfrak{E}$ are given by \eqref{MainPDEIntroModelWMB}. Then, for all $s > \delta+1$
, we have that 
\begin{itemize}[leftmargin=*]
	\item $f^{(3)}:H^s ( (0,\pi/2);d\mu ) \rightarrow H^s ( (0,\pi/2);d\mu )$ is a bounded homogeneous polynomial of degree three,
	\item $f^{(3)}$ leaves invariant the domain $  H^{s+2} ( (0,\pi/2);d\mu )$, 
	  \item $f^{(\geq 4)}(\phi)$ is differentiable in $H^s ( (0,\pi/2);d\mu )$, its differentiable is a Lipschitz map and satisfies the estimate
\begin{align*}
	\| d f^{(\geq 4)}(\phi_1) - d f^{(\geq 4)}(\phi_2) \|_{H^s ( (0,\pi/2);d\mu )} \lesssim \epsilon^3 \| \phi_1 -\phi_2 \|_{H^s ( (0,\pi/2);d\mu )}, 
\end{align*} 
for all $\| \phi_1 \|_{H^s ( (0,\pi/2);d\mu )} \lesssim \epsilon$, $\| \phi_2 \|_{H^s ( (0,\pi/2);d\mu )} \lesssim \epsilon$. 
\end{itemize}  
\end{lemma}
\begin{proof}
Fix $\delta$ satisfying Assumption \ref{AssumptionsondeltaExistence} and let $s > \delta+1$. All claims as stated above follow directly from the algebra property
\begin{align}\label{AlgebraPropertyStatement}
	\| \phi \psi\| _{ H^s ( (0,\pi/2);d\mu )} \lesssim 
	\| \phi \| _{ H^s ( (0,\pi/2);d\mu )}
	\| \psi\| _{ H^s ( (0,\pi/2);d\mu )},
\end{align} 
for all $\phi,\psi \in H^s ( (0,\pi/2);d\mu )$ with $d\mu=\mathfrak{w}dx$ where $\mathfrak{w}$ is given by \eqref{DefinitionWeightw}. To prove the latter, recall that $\mathfrak{w}(x)=\sin^{2\delta +1}(x) \cos^{3}(x)$ as well as, for any $\phi$, 
\begin{align*}
	\| \phi \|_{H^s( (0, \pi/2);d\mu)  }= 
	\begin{dcases}
		\| L^k \phi \|_{L^2( (0, \pi/2);d\mu)  }, & \text{ for }s=2k,,\quad k \in \mathbb{N}, \\
		 \| L^k \phi \|_{H^1( (0, \pi/2);d\mu)  }, & \text{ for }s=2k+1 ,\quad k \in \mathbb{N}.
	\end{dcases}
\end{align*} 
Let $\chi_1$ and $\chi_2$ be two smooth cut-off functions such that 
\begin{align*}
	\chi_1(x)= \begin{dcases}
1,  &x\in \left[0, \frac{\pi }{4}  \right], \\
 	0, &x\in  \left[\frac{3\pi }{8}, \frac{\pi }{2}  \right], 
 \end{dcases} \quad 
 \chi_2(x)= \begin{dcases}
 	0, &x\in  \left[0,\frac{\pi }{8}   \right], \\
 	1,  &x\in \left[\frac{\pi }{4} ,\frac{\pi }{2}  \right].
 	\end{dcases}
\end{align*}
Since only local operators are involved in the definition of the $H^s( (0, \pi/2);d\mu) $ norm, it follows that 
\begin{align*}
	\| \phi \|_{H^s( (0, \pi/2);d\mu) } \lesssim \| \chi_1 \phi \|_{H^s( (0, \pi/2);d\mu) } + \|  \chi_2 \phi \|_{H^s( (0, \pi/2);d\mu) }, 
\end{align*}
where now $\chi_i \phi$ vanish in neighbourhoods of the two endpoints of $(0, \pi/2)$.  We only focus on a neighbourhood of zero, since the other case can be treated similarly. For this, consider the warped product manifold $(0, \pi/2) \times \mathbb{S}^{2\delta+1}$ endowed with the metric 
\begin{align*}
	g(x,\omega)= dx^2 + \sin^2(x) \cos^{\frac{6}{2\delta+1}} (x) \sigma_{\mathbb{S}^{2\delta+1}}(\omega),
\end{align*}
where $\sigma_{\mathbb{S}^{2\delta+1}}$ is the standard round metric on the sphere $\mathbb{S}^{2\delta+1}$ of dimension $2\delta+1$. Its volume form is given by $d\text{vol}=\sqrt{|g| } dx=\sqrt{\det (g) } dx= \sin^{2\delta+1} \cos^3(x)dx=\mathfrak{w}dx=d\mu$. Let $\Delta_{g}$ be the Laplace-Beltrami operator associated to the metric $g$. Formally, when restricted to radial functions $\phi =\phi (x)$, it follows that 
\begin{align*}
	\Delta_{g} \phi =  
	 \frac{1}{\sqrt{|g|}} \partial_{x} \left(\sqrt{|g|} g^{xx}\partial_x \phi   \right)=
	  \frac{1}{\mathfrak{w}} \partial_{x} \left(\mathfrak{w} \partial_x \phi   \right) =
	  -\mathfrak{L} \phi +(\delta+2)^2 \phi ,
\end{align*}
where the definition of $\mathfrak{L}$ is given in Section \ref{SectionLineareigenvalueproblems}. Moreover, we claim that the warped product manifold above is regular at $x=0$. Indeed, introduce Cartesian coordinates $(x^1,\dots,x^{2\delta+2}) \in \mathbb{R}^{2\delta+2} \simeq \mathbb{R} \times \mathbb{S}^{2\delta+1}$ such that
\begin{align*}
\begin{dcases}
	x^1 &= x \cos (\omega_1), \\ 
	x^2 &= x \sin (\omega_1) \cos (\omega_2), \\
	x^3 &= x \sin (\omega_1)\sin (\omega_2) \cos (\omega_3), \\
	 \vdots &\quad  \quad \quad  \vdots  \quad\quad \quad  \vdots \quad\quad \quad  \vdots \\
	 x^{2\delta+1} &= x \sin (\omega_1)\sin (\omega_2) \sin (\omega_3) \dots \sin (\omega_{2\delta+1})\cos (\omega_{2\delta+1}), \\
	 x^{2\delta+2} &= x \sin (\omega_1)\sin (\omega_2) \sin (\omega_3) \dots \sin (\omega_{2\delta+1})\sin (\omega_{2\delta+1}).
\end{dcases}
\end{align*} 
We know that the standard Euclidean metric $\delta_{\mathbb{E}}$ is given by $\delta_{\mathbb{E}}(x,\omega)= dx^2 + x^2 \sigma_{\mathbb{S}^{2\delta+1}}(\omega)$, thus we have
\begin{align*}
	g(x,\omega) &= \delta_{\mathbb{E}}(x,\omega) + \left( \sin^2(x) \cos^{\frac{6}{2\delta+1}}(x) -x^2\right)\sigma_{\mathbb{S}^{2\delta+1}}(\omega) 
	\\
	&= \delta_{\mathbb{E}}(x,\omega) + x^4 f(x
^2)\sigma_{\mathbb{S}^{2\delta+1}}(\omega) \\
&=\delta_{\mathbb{E}}(x,\omega) +   f(x
^2)\left(x^2 \delta_{\mathbb{E}}(x,\omega)- x^2 dx^2 \right),
\end{align*} 
for some smooth function $f$. Obviously, $x^2$ is a smooth function with respect to the Cartesian coordinates $x=(x^i)$. Furthermore, 
\begin{align*}
	dx= \frac{x_i}{x} dx^i, \quad  dx^2= \frac{x_i x_j}{x^2} dx^i dx^j,
\end{align*}
hence $x^2 dx^2= x_i x_j dx^i dx^j $ is also smooth at $x=0$ with respect to the Cartesian coordinates $x=(x^i)$. In conclusion, the warped product manifold $(0, \pi/2) \times \mathbb{S}^{2\delta+1}$ is smooth and the weighted Sobolev space $H^s ( (0,\pi/2);d\mu )$ we consider coincides with the Sobolev space $H_{\text{rad}}^s ( (0,\pi/2)\times \mathbb{S}^{2\delta+1}  )$ of radial functions on the manifold. Then, by standard Sobolev regularity, the algebra property \eqref{AlgebraPropertyStatement} holds for any $\phi, \psi \in H^s ( (0,\pi/2);d\mu )  $ with compact support in $[0, 3\pi/8]$ provided that $s > (2\delta+2)/2=\delta+1$, that completes the proof. 
\end{proof}

\section{Stability of the solutions to the linear wave equation} \label{SectionStabilitySolutionstotheLinearWaveEquation}

 In this section, we establish the non-linear stability over exponentially long times to solutions to the KG equation
 \begin{align}\label{NonlinearWaveEquationSectionNonLinearStability}
	\partial_{t}^2 \phi +L \phi = -\epsilon^2 W \phi^3  ,
\end{align}
for a scalar field $\phi: \mathbb{R} \times (0,\pi/2) \rightarrow \mathbb{R}$, where $L$ and $W$ are given by \eqref{MainPDEIntroModelKGB}. To do so, we will apply Theorem \ref{TheoremBambusiNekhoroshev} to the same rescaled first 1-mode given by \eqref{Definition1modeinitialdata} and verify the conditions \eqref{Assumption1BambusiNekhoroshev} and \eqref{Assumption2BambusiNekhoroshev} of Theorem \ref{TheoremBambusiNekhoroshev}. We follow the setting in \cite{BAMBUSI199873} and rewrite  \eqref{NonlinearWaveEquationSectionNonLinearStability} in the phase space 
\begin{align*}
	\mathbf{p}= (p_1,p_2) \in  \mathcal{P} =H^{1} \left( (0,\pi/2);d\mu\right) \times L^2 \left( (0,\pi/2) ;d\mu\right) 
\end{align*}
using symbols in bold to denote its elements. In particular, for any $\boldsymbol{\phi} =(\phi,\partial_{t}\phi)\in  \mathcal{P}$, the non-linear wave equation \eqref{NonlinearWaveEquationSectionNonLinearStability} can be written as 
\begin{align}\label{MainEquationForBambusiNekhoroshev}
	\partial_{t}\boldsymbol{\phi}=
	A\boldsymbol{\phi} +   \epsilon^2  f^{(3)}(\boldsymbol{\phi}),  
\end{align}
where $A$ is a linear operator defined by
\begin{align*}
	A:\mathcal{D} (A)\subset \mathcal{P} \rightarrow \mathcal{P},\quad  A=
	\begin{pmatrix}
		0 & 1 \\
		-L & 0
	\end{pmatrix}
\end{align*}
and 
\begin{align}\label{NonLinearityf3StabilityKG}
	f^{(3)} (\boldsymbol{\phi})=\begin{pmatrix}
		0 \\
		-W \phi^3 
	\end{pmatrix} .
\end{align}
Furthermore, for any $\boldsymbol{\zeta}= (\zeta_1,\zeta_2) \in  \mathcal{P}$, let 
\begin{align*} 
	\boldsymbol{\Phi}^{t}(\boldsymbol{\zeta}) := e^{At}\boldsymbol{\zeta}
\end{align*}
be the  solution to the initial value problem consisting of the linearized equation in \eqref{MainEquationForBambusiNekhoroshev} coupled to $\boldsymbol{\zeta} $ as initial data,
\begin{align} \label{LinearizedEquationForBambusiNekhoroshev}
\begin{cases}
	\partial_{t}\boldsymbol{\phi}_{\text{linear}}(t)=
	A\boldsymbol{\phi}_{\text{linear}}(t), \quad t \in \mathbb{R}, \\
 \boldsymbol{\phi}_{\text{linear}}(0)=\boldsymbol{\zeta} 
\end{cases}
\end{align}
and denote by   
\begin{align*}
	\Gamma_{\text{linear}}(\boldsymbol{\zeta}) = \left \{\boldsymbol{\Phi}^{t}(\boldsymbol{\zeta}):t \in \mathbb{R} \right\} 
\end{align*} 
its trajectory in the phase space $\mathcal{P}$. In addition, let $\xi_1$ be the same rescale of the first 1-mode as in \eqref{Definition1modeinitialdata}, and set
\begin{align}  \label{Definition1modesInitialdataBambusiNekhoroshev}
	\boldsymbol{\xi}=(\xi_1,\xi_2), \quad \xi_1 =  \kappa_{0}e_{0}, \quad \xi_2 =0, \quad \kappa_{0}       = \pm \omega_{0}    \sqrt{\frac{8}{3C_{00 0 0}   }}.
\end{align}
To begin with, we endow the phase space $\mathcal{P}$ with an inner product $\langle \cdot,\cdot \rangle :  \mathcal{P}  \times \mathcal{P} \rightarrow \mathbb{R}$,
\begin{align*}
	\langle \mathbf{p},\tilde{\mathbf{p}} \rangle &=
	\langle p_1,\tilde{p}_1 \rangle_{H^1(0,\pi/2)}+
	\langle p_2,\tilde{p}_2 \rangle_{L^2(0,\pi/2)} 
	= \int_{0}^{\pi/2} \left(
	\partial_{x} p_1  \partial_{x}\tilde{p}_1 +
	 p_1  \tilde{p}_1 \right)d\mu +
	\int_{0}^{\pi/2}
	  p_2   \tilde{p}_2  d\mu,
\end{align*}
for all $\mathbf{p}=(p_1,p_2)\in \mathcal{P}$ and $\tilde{\mathbf{p}}=(\tilde{p}_1,\tilde{p}_2)\in \mathcal{P}$,  and a symplectic form $\Omega:  \mathcal{P}  \times \mathcal{P} \rightarrow \mathbb{R}$,
\begin{align*}
	 \Omega( \mathbf{p},\tilde{\mathbf{p}} ) = \int_{0}^{\pi/2}\left( p_1  \tilde{p}_2 - p_{2} \tilde{p}_1 \right)  d\mu,
\end{align*} 
for all $\mathbf{p}=(p_1,p_2)\in \mathcal{P}$ and $\tilde{\mathbf{p}}=(\tilde{p}_1,\tilde{p}_2)\in \mathcal{P}$. Notice that, since $L$ is self-adjoint in  $L^2((0,\pi/2);d\mu)$ (Section \ref{SectionLineareigenvalueproblems}), the operator $A$ is skew-symmetric with respect to $\Omega$, meaning that
\begin{align*}
	\Omega(A\mathbf{p},\tilde{\mathbf{p}})=-\Omega(\mathbf{p},A\tilde{\mathbf{p}}),
\end{align*}
for all $\mathbf{p}=(p_1,p_2)\in \mathcal{P}$ and $\tilde{\mathbf{p}}=(\tilde{p}_1,\tilde{p}_2)\in \mathcal{P}$ with $p_1,\tilde{p}_1 \in \mathcal{D}(L)$.
Moreover, we note that
\begin{align*}
	\Omega(A\mathbf{p},\mathbf{p}) &=\Omega \left(
	(p_2,-Lp_1)	,(p_1,p_2) \right) 
	 = \int_{0}^{\pi/2}\left(  p_2 ^2 +p_{1} Lp_{1} \right)  d\mu 
	 =\int_{0}^{\pi/2}\left( 
	 p_2 ^2 + \left(L^{1/2}p_{1}\right)^2 \right)  d\mu.
\end{align*}
The Hamiltonian function 
\begin{align*}
	\mathcal{H}:\mathcal{P}  \rightarrow \mathbb{R},  \quad \mathcal{H} =\mathcal{H}(\boldsymbol{\phi}), \quad \boldsymbol{\phi}=(\phi_1,\phi_2)  
\end{align*}
is given by 
 \begin{align*}
 	\mathcal{H}(\boldsymbol{\phi})  
 	&= h_{\Omega}(\boldsymbol{\phi})  + \epsilon^2  f(\boldsymbol{\phi})     , 
 \end{align*}
 where
 \begin{align}\label{LeadingOrderNonlinearityHamiltonianSectionStability}
 h_{\Omega}(\boldsymbol{\phi}) = \frac{1}{2} \Omega(A\boldsymbol{\phi},\boldsymbol{\phi}), \quad  	f(\boldsymbol{\phi}) =  \frac{1}{4} \int_{0}^{\pi/2} W \phi_1^4 d\mu    .
 \end{align}
 We call $ h_{\Omega}(\boldsymbol{\phi})$ the harmonic energy of $ \boldsymbol{\phi}$. Then, the partial differential equation  \eqref{NonlinearWaveEquationSectionNonLinearStability} is Hamiltonian and formally one has
 \begin{align*}
 	\frac{d}{dt}\mathcal{H}(\phi,\partial_{t}\phi) =0.
 \end{align*} 
In fact, Theorem \ref{TheoremBambusiNekhoroshev} is based a perturbation analysis on the surface of constant harmonic energy
 \begin{align*}
	\Sigma= \left \{\boldsymbol{\zeta}\in \mathcal{P}:~  h_{\Omega}(\boldsymbol{\zeta})=\kappa_{0}^2 \omega_{0}^2 \right\}.
\end{align*}
As usual in perturbation theory, one also needs an analyticity assumption for the perturbation. In the case of the KG model, this is a consequence of the following Hardy-Sobolev inequality. The following lemma guarantees the analyticity conditions of Theorem \ref{TheoremBambusiNekhoroshev}. 
 \begin{lemma}[Hardy-Sobolev inequality]\label{LemmaHardySolobelInequality}
	Fix $\delta$ satisfying Assumption \ref{AssumptionsondeltaExistence} and define $f^{(3)}(\boldsymbol{\phi})$ and $f(\boldsymbol{\phi}) $ by \eqref{NonLinearityf3StabilityKG} and \eqref{LeadingOrderNonlinearityHamiltonianSectionStability} respectively. Then, there exists a positive constant $c$ so that 
	\begin{align*}
		f(\boldsymbol{\phi})  \leq c \left( h_{\Omega}(\boldsymbol{\phi}) \right)^2, \quad \| f^{(3)} (\boldsymbol{\phi}) \|_{\mathcal{P}} \leq c \| \boldsymbol{\phi} \|_{\mathcal{P}}^3,
	\end{align*}
	for all $\boldsymbol{\phi} \in \mathcal{P} $. 
\end{lemma} 
\begin{proof} 
Fix $\delta$ satisfying Assumption \ref{AssumptionsondeltaExistence} and define $f^{(3)}(\boldsymbol{\phi})$ and $f(\boldsymbol{\phi}) $ by \eqref{NonLinearityf3StabilityKG} and \eqref{LeadingOrderNonlinearityHamiltonianSectionStability} respectively for all $\boldsymbol{\phi} =(\phi,\partial_{t}\phi) \in \mathcal{P}$. By density (Section \ref{se:sa}), we can assume that $\boldsymbol{\phi} \in C^\infty( 0, \pi/2)$. Firstly, we focus on $f(\boldsymbol{\phi})  \leq c \left( h_{\Omega}(\boldsymbol{\phi}) \right)^2$ and prove that
	\begin{align} \label{InequalitySobolevHardyGoalKG}
		\left( \int_{0}^{\pi/2} \mathsf{W} \phi^4 d\mu  \right)^{1/4}  \lesssim \left( \int_{0}^{\pi/2}   \phi \mathsf{L}\phi  d\mu  \right)^{1/2},
\end{align}
where
\begin{align*}
	 \mathsf{W}(x) &= \cos^{2(\delta-1)}(x), \quad  \mathsf{w}(x)=\sin^2(x) \cos^{2(\delta-1)}(x), \\ d\mu(x)&=\mathsf{w}(x) dx, \quad \mathsf{L} \phi   =-\frac{1}{\mathsf{w}} \partial_{x} \left(\mathsf{w} \partial_{x}\phi \right)+\delta^2 \phi.
\end{align*}  
We set $\psi=\cos^{\delta-1}(x) \phi $ so that $\psi(\pi/2)=0$ and compute
\begin{align*} 
	  \mathsf{w} \partial_{x} \phi &=   \mathsf{w} \partial_{x} ( \cos^{-(\delta-1)}(x) \psi) \\
	 &=  \sin^2(x) \cos^{2\delta-2}(x)  [ (\delta-1)\cos^{-\delta}(x) \sin(x) \psi + \cos^{-\delta+1}(x) \partial_{x}\psi] \\
	 &=   (\delta-1)   \sin^3(x) \cos^{\delta-2}(x) \psi 
	 +  \sin^2(x) \cos^{\delta-1}(x) \partial_{x}\psi, \\
	 \partial_{x} \left(  \mathsf{w} \partial_{x} \phi \right)&= 
	 (\delta-1)  \partial_{x} \left( \sin^3(x) \cos^{\delta-2}(x) \psi \right)\\
	 &+  \cos^{\delta-1}(x) \partial_{x} \left( \sin^2(x)\partial_{x}\psi \right)
	 +  \sin^2(x)\partial_{x}\psi \partial_{x} \left( \cos^{\delta-1}(x) \right) \\
	 &= 
	3 (\delta-1)     \sin^2(x) \cos^{\delta-1}(x) \psi  -
	  (\delta-1) (\delta-2)     \sin^4(x) \cos^{\delta-3}(x) \psi   \\
	 & +
	   (\delta-1)     \sin^3(x) \cos^{\delta-2}(x) \partial_{x}\psi+  \cos^{\delta-1}(x) \partial_{x} \left( \sin^2(x)\partial_{x}\psi \right)\\
	 &- (\delta-1) \sin^3(x)  \cos^{\delta-2}(x) \partial_{x}\psi   \\
	 &= 
	3 (\delta-1)     \sin^2(x) \cos^{\delta-1}(x) \psi  -
	  (\delta-1) (\delta-2)     \sin^4(x) \cos^{\delta-3}(x) \psi   \\
	 & +  \cos^{\delta-1}(x) \partial_{x} \left( \sin^2(x)\partial_{x}\psi \right),   \\
	\frac{1}{ \mathsf{w}} \partial_{x} \left(  \mathsf{w} \partial_{x} \phi \right)&=  
 3 (\delta-1)   \cos^{-\delta+1}(x)     \psi  -
  (\delta-1) (\delta-2)  \cos^{-\delta-1}(x)
	     \sin^2(x)   \psi   \\
	 & +  \cos^{-\delta+1  }(x) \sin^{-2}(x) \partial_{x} \left( \sin^2(x)\partial_{x}\psi \right),   \\
	 \mathsf{L} \phi &= -3 (\delta-1)   \cos^{-\delta+1}(x)     \psi  +
  (\delta-1) (\delta-2)  \cos^{-\delta-1}(x)
	     \sin^2(x)   \psi \\
	     & + \delta^2 \cos^{-\delta+1}(x) \psi 
	      -  \cos^{-\delta+1  }(x) \sin^{-2}(x) \partial_{x} \left( \sin^2(x)\partial_{x}\psi \right), \\
	       \mathsf{w}  \mathsf{L} \phi &= -3 (\delta-1)  \sin^2(x) \cos^{\delta-1}(x)      \psi  +
  (\delta-1) (\delta-2) \sin^4(x) \cos^{\delta-3}(x)    \psi \\
	     & + \delta^2  \sin^2(x) \cos^{\delta-1}(x) \psi 
	      -  \cos^{\delta-1}(x)  \partial_{x} \left( \sin^2(x)\partial_{x}\psi \right)
\end{align*} 
and hence
\begin{align*}
	 \mathsf{w} \phi \mathsf{L} \phi  &= -3 (\delta-1)  \sin^2(x)     \psi  ^2    +
  (\delta-1) (\delta-2) \sin^4(x) \cos^{ -2}(x)    \psi^2  \\
	     & + \delta^2  \sin^2(x)  \psi   ^2
	      -     \psi  \partial_{x} \left( \sin^2(x)\partial_{x}\psi \right)   .
\end{align*}
Now, using the Dirichlet boundary condition $\psi(\pi/2)=0$, integration by parts yields 
\begin{align*}
	    \int_{0}^{\pi/2} \psi  \partial_{x} \left( \sin^2(x)\partial_{x}\psi \right)       dx=   - \int_{0}^{\pi/2}     \left( \partial_{x}\psi \right) ^2   \sin^2(x)   dx
\end{align*}
and hence we deduce
\begin{align*}
		\int_{0}^{\pi/2} \phi \mathsf{L} \phi d \mu &=
		\int_{0}^{\pi/2} \mathsf{w} \phi \mathsf{L} \phi  d x  \\
		&= -3 (\delta-1) \int_{0}^{\pi/2} \sin^2(x)     \psi  ^2  dx  +
  (\delta-1) (\delta-2)\int_{0}^{\pi/2} \sin^4(x) \cos^{ -2}(x)    \psi^2  dx \\
	     & + \delta^2  \int_{0}^{\pi/2} \sin^2(x) \psi   ^2dx
	      -    \int_{0}^{\pi/2} \psi  \partial_{x} \left( \sin^2(x)\partial_{x}\psi \right)       dx \\
	      &=\left[  \delta^2 -3 (\delta-1) \right] \int_{0}^{\pi/2}     \psi  ^2  \sin^2(x) dx  +
  (\delta-1) (\delta-2)\int_{0}^{\pi/2}   \tan^2(x)  \psi^2  \sin^2(x)  dx  \\
	  & + \int_{0}^{\pi/2}    \left( \partial_{x}\psi \right) ^2   \sin^2(x)    dx.
\end{align*}
Notice that 
\begin{align*}
	\delta^2 -3 (\delta-1) \geq 1, \quad (\delta-1) (\delta-2) \geq 0,
\end{align*}
for all integers $\delta \geq 2$, and consequently we arrive at
\begin{align*} 
		\int_{0}^{\pi/2} \phi  \mathsf{L} \phi d \mu  &  \geq   \int_{0}^{\pi/2}     \psi  ^2  \sin^2(x) dx   + \int_{0}^{\pi/2}    \left( \partial_{x}\psi \right) ^2   \sin^2(x)    dx  .
\end{align*}
We use once again the Dirichlet boundary condition $\psi(\pi/2)=0$ to extend $\psi:(0,\pi/2) \rightarrow \mathbb{R}$ to a smooth function $\psi:(0,\pi)\rightarrow \mathbb{R}$ imposing the reflection symmetry $\psi(x)=-\psi(\pi-x)$, for all $x \in (0,\pi)$. Furthermore, let $f: \mathbb{S}^3= (0,\pi) \times (0,\pi) \times (0,2\pi)\rightarrow \mathbb{R}$ defined by $f(x,\vartheta,\varphi)=\psi(x)$ and note that $d\text{vol}_{\mathbb{S}^3}(x,\vartheta,\varphi) = \sin^2 (x) \sin(\vartheta) dx d\vartheta d\varphi$ is the volume form on the entire $3$-sphere. On the one hand, we have that
\begin{align} \label{Hardy1}
	\left( 	\int_{0}^{\pi/2} \phi  \mathsf{L} \phi d \mu \right)^{1/2} &  \geq \left(  \int_{0}^{\pi/2}   \left[  \psi  ^2 +\left( \partial_{x}\psi \right) ^2 \right] \sin^2(x) dx \right)^{1/2}   \simeq \left(  \int_{0}^{\pi}   \left[  \psi  ^2 +\left( \partial_{x}\psi \right) ^2 \right] \sin^2(x) dx \right)^{1/2}  \nonumber \\
	 & \simeq \left(  \int_{\mathbb{S}^3}  \left[   f ^2 +   \left( \nabla_{\mathbb{S}^3}f\right) ^2 \right] d\text{vol}_{\mathbb{S}^3}  \right)^{1/2}   =  \| f\|_{H^1(\mathbb{S}^3)}.
\end{align}
On the other hand, we also have that
\begin{align} \label{Hardy2}
\left(	\int_{0}^{\pi/2}  \mathsf{W} \phi^4 d\mu \right)^{1/4} & =  
	\left( \int_{0}^{\pi/2}  \psi^4 \sin^2(x) dx \right)^{1/4}  \simeq
\left( 	\int_{0}^{\pi}  \psi^4 \sin^2(x) dx   \right)^{1/4} \nonumber \\
& \simeq \left( \int_{\mathbb{S}^3}  f^4   d\text{vol}_{\mathbb{S}^3} \right)^{1/4}    \simeq  \| f\|_{L^4(\mathbb{S}^3)} .
\end{align}
In view of \eqref{Hardy1} and \eqref{Hardy2}, in order to prove \eqref{InequalitySobolevHardyGoalKG}, it suffices to show that there exists a positive constant so that
\begin{align*}
	\| f\|_{L^4(\mathbb{S}^3)} \lesssim  \| f\|_{H^1(\mathbb{S}^3)}.
\end{align*}
The latter is a direct consequence of the standard Sobolev inequality on the 3-sphere (Theorem 5.1, page 121 in \cite{MR1688256}),
\begin{align*} 
	\| f\|_{L^6(\mathbb{S}^3)} \lesssim  \| f\|_{H^1(\mathbb{S}^3)},
\end{align*}
 together with the embedding $L^6(\mathbb{S}^3) \hookrightarrow L^4(\mathbb{S}^3) $ since $ \mathbb{S}^3  $ is compact. Finally, we focus on the second estimate and prove that $\| f^{(3)} (\boldsymbol{\phi}) \|_{\mathcal{P}} \lesssim \| \boldsymbol{\phi} \|_{\mathcal{P}}^3$. To this end, we proceed as above and obtain
\begin{align*}
	\| f^{(3)} (\boldsymbol{\phi}) \|^{2}_{\mathcal{P}} &= 
	\|   \mathsf{W} \phi^3  \|^{2}_{L^2 (  (0,\pi/2);d\mu  )}  
	=  \int_{0}^{\pi/2}  \mathsf{W}^2(x) \phi^6(x)   d\mu(x)   =  \int_{0}^{\pi/2} 	\psi^6 (x) \sin^2(x)   dx \\
	&
	 \simeq 	\int_{0}^{\pi}  \psi^6 \sin^2(x) dx    
 \simeq  \int_{\mathbb{S}^3}  f^6   d\text{vol}_{\mathbb{S}^3}     \simeq  \| f\|_{L^6(\mathbb{S}^3)}^6 \lesssim  \| f\|_{H^1(\mathbb{S}^3)}^6 \\
	& \lesssim \|\phi\|_{H^1 ( (0,\pi/2);d\mu )}^6 \leq \|  \boldsymbol{\phi}  \| _{\mathcal{P}}^6 ,
\end{align*}
that completes the proof.  
 \end{proof}
Next, we compute the harmonic energy in the Fourier space.
\begin{lemma}[Harmonic energy]\label{LemmaHarmonicFlowWrtTheBasisem}
	Let $\boldsymbol{\phi}=(\phi_1,\phi_2) \in \mathcal{P}$ be expanded in terms of the eigenfunctions $\{e_m:m\geq 0 \}$ to the linearized operator $L$,
	 \begin{align*}
	 	\phi_{ \alpha} (t)=\sum_{m=0}^{\infty} \phi_{ \alpha}^{(m)} (t)e_m, \quad 
	   \alpha \in \{1,2\}.
	 \end{align*}
	 Then, 
	 \begin{align*}
	 h_{\Omega}(\boldsymbol{\phi}) &= \sum_{m=0}^{\infty} 
	 ( \phi_2^{(m)} )^2  +\sum_{m=0}^{\infty} \omega_{m}^2  ( \phi_1^{(m)} )^2  .
\end{align*}
\end{lemma}
\begin{proof}
Let $\boldsymbol{\phi}=(\phi_1,\phi_2) \in \mathcal{P}$ and expand its coefficients   in terms of the eigenfunctions $\{e_m:m\geq 0 \}$ to the linearized operator $L$,
	 \begin{align*}
	 	\phi_{ \alpha} (t)=\sum_{m=0}^{\infty} \phi_{ \alpha}^{(m)} (t)e_m, \quad 
	   \alpha \in \{1,2\}.
	 \end{align*}
Then, we use the fact that the set of eigenfunctions $\{e_m:m\geq 0\}$ is orthonormal with respect to the inner product in $L^2((0,\pi/2);d\mu)$, to compute that
\begin{align*} 
	 h_{\Omega}(\boldsymbol{\phi}) &= \frac{1}{2} \Omega(A\boldsymbol{\phi},\boldsymbol{\phi}) 
	 =\int_{0}^{\pi/2}\left( 
	 \phi_2 ^2 + \left(L^{1/2}\phi_{1}\right)^2 \right)  d\mu  \\
	 & =\sum_{i,j=0}^{\infty}  \int_{0}^{\pi/2}\left( 
	 \phi_2^{(i)}\phi_2^{(j)} e_i (x)e_j(x) +  \phi_1^{(i)}\phi_1^{(j)} L^{1/2}e_i (x)L^{1/2}e_j(x) \right)  d\mu (x)  \\
	  & =\sum_{i,j=0}^{\infty}  \left( 
	 \phi_2^{(i)}\phi_2^{(j)}   + \omega_{i}\omega_{j} \phi_1^{(i)}\phi_1^{(j)}   \right)\int_{0}^{\pi/2}e_i (x) e_j(x)  d\mu (x)  \\
	 & =\sum_{i,j=0}^{\infty}  \left( 
	 \phi_2^{(i)}\phi_2^{(j)}   + \omega_{i}\omega_{j} \phi_1^{(i)}\phi_1^{(j)}   \right) \delta_{ij}  = \sum_{m=0}^{\infty} 
	 ( \phi_2^{(m)} )^2  +\sum_{m=0}^{\infty} \omega_{m}^2  ( \phi_1^{(m)})^2  ,
\end{align*}
that completes the proof.
\end{proof}
Next, we compute the linear flow in the Fourier space. To do so, we denote by $ e^{At}$ the unitary group generated by $A$. 

\begin{lemma}[Linear Flow]\label{LemmaLinearFlowClosedFormula}
	Let $\boldsymbol{\zeta}=(\zeta_1,\zeta_2) \in \mathcal{P}$ be expanded in terms of the eigenfunctions $	\{ e_m: m\geq 0 \}$ to the linearized operator $L$,
	\begin{align*}
		\zeta_\alpha =\sum_{m=0}^{\infty} \zeta_\alpha^{(m)} e_m,\quad  \alpha \in \{1,2\}.
	\end{align*}
Then,  
\begin{align*} 
	\boldsymbol{\Phi}^{t}(\boldsymbol{\zeta}) := e^{At}\boldsymbol{\zeta}= 
	 \begin{pmatrix}
		\sum_{m=0}^{\infty} \left(\zeta_{1}^{(m)}  \cos(\omega_{m}t) + \frac{\zeta_2^{(m)}}{\omega_{m}} \sin (\omega_{m}t) \right) e_m \\
		\sum_{m=0}^{\infty} \left(-\omega_{m}\zeta_{1}^{(m)}  \sin(\omega_{m}t) +  \zeta_2^{(m)}  \cos (\omega_{m}t) \right) e_m 
	\end{pmatrix}.
\end{align*}
\end{lemma}
\begin{proof}
Let $\boldsymbol{\zeta}=(\zeta_1,\zeta_2) \in \mathcal{P}$ and recall that $\boldsymbol{\phi}_{\text{linear}}(t)=\boldsymbol{\Phi}^{t}(\boldsymbol{\zeta})=e^{tA}\boldsymbol{\zeta}$  is the solution to the initial value problem \eqref{LinearizedEquationForBambusiNekhoroshev}. For $\boldsymbol{\phi}_{\text{linear}}(t)=(\phi_{\text{linear},1}(t),\phi_{\text{linear},2}(t)) \in \mathcal{P}$, we compute
	\begin{align*}
	\begin{pmatrix}
		\partial_{t}\phi_{\text{linear},1}(t) \\
		\partial_{t}\phi_{\text{linear},2}(t)
	\end{pmatrix} 
	= 
		\begin{pmatrix}
		0 & 1 \\
		-L & 0
	\end{pmatrix}
	\begin{pmatrix}
		\phi_{\text{linear},1}(t) \\
		\phi_{\text{linear},2}(t)
	\end{pmatrix}=
	\begin{pmatrix}
		\phi_{\text{linear},2}(t) \\
		-L \phi_{\text{linear},1}(t)
	\end{pmatrix}  
	\end{align*} 
	that leads to
	\begin{align}\label{AuxiliarlySecondOrderODE}
		\partial_{t}^2 \phi_{\text{linear},1}(t) +L  \phi_{\text{linear},1}(t) =0.
	\end{align}
	 Expanding $\phi_{\text{linear},1}(t)$, $\phi_{\text{linear},2}(t)$, $\zeta_1$ and $\zeta_2$ in terms of the eigenfunctions of $L$,
	 \begin{align*}
	 	\phi_{\text{linear},\alpha} (t)=\sum_{m=0}^{\infty} \phi_{\text{linear},\alpha}^{(m)} (t)e_m, \quad 
	 	\zeta_\alpha =\sum_{m=0}^{\infty} \zeta_\alpha^{(m)} e_m,\quad  \alpha \in \{1,2\},
	 \end{align*}
	 we infer
	 \begin{align*}
	\phi_{\text{linear},1}^{(m)}(0) &= \left( \phi_{\text{linear},1}(0) | e_m \right)  = \left( \zeta_{1} | e_m \right)= \zeta_1^{(m)}, \\ 
	\partial_{t}  \phi_{\text{linear,1}}^{(m)}(0)&= \left(\partial_{t}\phi_{\text{linear,1}}(0) | e_m \right)  = \left( (A\zeta)_{1} | e_m \right)=\left( \zeta_{2} | e_m \right)= \zeta_2^{(m)},
	 \end{align*}
	 for all integers $m\geq0$, and hence \eqref{AuxiliarlySecondOrderODE} yields
	 \begin{align*}
	 	 \frac{d^2}{dt^2} \ddot{\phi}_{\text{linear,1}}^{(m)}(t)+\omega_{m}^2 \phi_{\text{linear},1}^{(m)}(t)=0 
	 	  \Longrightarrow   \phi_{\text{linear},1}^{(m)}(t)= \zeta_{1}^{(m)}  \cos(\omega_{m}t) + \frac{\zeta_2^{(m)}}{\omega_{m}} \sin (\omega_{m}t),
	 \end{align*}
	 for all integers $m\geq0$. Consequently, we obtain 
\begin{align*} 
	\boldsymbol{\Phi}^{t}(\boldsymbol{\zeta}) = e^{At}\boldsymbol{\zeta}= 
	 \begin{pmatrix}
		\sum_{m=0}^{\infty} \left(\zeta_{1}^{(m)}  \cos(\omega_{m}t) + \frac{\zeta_2^{(m)}}{\omega_{m}} \sin (\omega_{m}t) \right) e_m \\
		\sum_{m=0}^{\infty} \left(-\omega_{m}\zeta_{1}^{(m)}  \sin(\omega_{m}t) +  \zeta_2^{(m)}  \cos (\omega_{m}t) \right) e_m 
	\end{pmatrix},
\end{align*}
that completes the proof.
\end{proof}
Notice that $\boldsymbol{\Phi}^{t} (\boldsymbol{\zeta})$ is $2\pi$-periodic in time. The leading order term $f(\boldsymbol{\phi})$ in the Hamiltonian is given by \eqref{LeadingOrderNonlinearityHamiltonianSectionStability} and we define its time-average over the linear flow  by
 \begin{align*}
 	\langle f  \rangle (\boldsymbol{\zeta})= \frac{1}{2\pi} \int_{0}^{2\pi} f (\boldsymbol{\Phi}^{t}(\boldsymbol{\zeta}) )	dt,
 \end{align*}
where $\boldsymbol{\Phi}^{t}(\boldsymbol{\zeta}) $ stands for the linear flow, given by Lemma \ref{LemmaLinearFlowClosedFormula}, associated with $\boldsymbol{\zeta}$. Next, we compute this time-averaging operator in the Fourier space.

\begin{lemma}[Time-averaging]\label{LemmaTimeAveraging}
	Let $\boldsymbol{\zeta}=(\zeta_1,\zeta_2) \in \mathcal{P}$ be expanded in terms of the eigenfunctions $\{e_m:m\geq 0 \}$ to the linearized operator $L$,
	 \begin{align*} 
	 	\zeta_{ \alpha} =\sum_{m=0}^{\infty} \zeta_{ \alpha}^{(m)} e_m,
	 	 \quad 
	   \alpha \in \{1,2\}.
	 \end{align*} 
	 Then,
	 \begin{align*}
	 	\langle f \rangle (\boldsymbol{\zeta}) 
	&=\frac{3}{32}\sum_{i,j,k,m=0}^{\infty} C_{ijkm}   
	P_{ijkm}(\boldsymbol{\zeta})  \mathds{1}(m=i+j-k \geq 0) ,
\end{align*} 
where
\begin{align*}
	P_{ijkm}(\boldsymbol{\zeta})  &= -\frac{\zeta_2^{  ( i )} \zeta_2^{  ( j )} \zeta_1^{  ( k )} \zeta_1^{  ( m )}}{\omega_{ i } \omega_{ j }}+\frac{\zeta_2^{  ( i )} \zeta_1^{  ( j )} \zeta_2^{  ( k )} \zeta_1^{  ( m )}}{\omega_{ i } \omega_{ k }}+\frac{\zeta_1^{  ( i )} \zeta_2^{  ( j )} \zeta_2^{  ( k )} \zeta_1^{  ( m )}}{\omega_{ j } \omega_{ k }} +\frac{\zeta_2^{  ( i )} \zeta_1^{  ( j )} \zeta_1^{  ( k )} \zeta_2^{  ( m )}}{\omega_{ i } \omega_{ m }} \\
	&+\frac{\zeta_1^{  ( i )} \zeta_2^{  ( j )} \zeta_1^{  ( k )} \zeta_2^{  ( m )}}{\omega_{ j } \omega_{ m }}-\frac{\zeta_1^{  ( i )} \zeta_1^{  ( j )} \zeta_2^{  ( k )} \zeta_2^{  ( m )}}{\omega_{ k } \omega_{ m }}+\zeta_1^{  ( i )} \zeta_1^{  ( j )} \zeta_1^{  ( k )} \zeta_1^{  ( m )}+\frac{\zeta_2^{  ( i )} \zeta_2^{  ( j )} \zeta_2^{  ( k )} \zeta_2^{  ( m )}}{\omega_{ i } \omega_{ j } \omega_{ k } \omega_{ m }}.
\end{align*}
\end{lemma}
\begin{proof} 
Let $\boldsymbol{\phi} =(\phi_{ 1},\phi_{ 2}) \in \mathcal{P}$ and expand its component in terms of the eigenfunctions to the lienarized operator $L$, namely
	 \begin{align*}
	 	\phi_{ \alpha} (t)=\sum_{m=0}^{\infty} \phi_\alpha^{(m)} (t)e_m, \quad  \alpha \in \{1,2\}.
	 \end{align*} 
Then, we also expand
\begin{align}\label{fphi}
	 	f(\boldsymbol{\phi}) &=  \frac{1}{4} \int_{0}^{\pi/2} W \phi_1^4 d\mu \nonumber \\
	 	&=  \frac{1}{4}\sum_{i,j,k,m=0}^{\infty} \phi_{1}^{(i)} (t)\phi_{1}^{(j)} (t)\phi_{1}^{(k)} (t)\phi_{1}^{(m)} (t) \int_{0}^{\pi/2} W(x)e_i(x)e_j(x)e_k(x)e_m(x)d\mu (x)\nonumber \\
	 	&=  \frac{1}{4}\sum_{i,j,k,m=0}^{\infty} C_{ijkm}   \phi_{1}^{(i)} (t)\phi_{1}^{(j)} (t)\phi_{1}^{(k)} (t)\phi_{1}^{(m)} (t) ,
\end{align}
where we also used the definition of the Fourier coefficient from \eqref{DefinitionFourierByExpantion}. Let $\boldsymbol{\zeta}=(\zeta_1,\zeta_2) \in \mathcal{P}$ and consider the linear flow $\boldsymbol{\phi}_{\text{linear}}(t)=\boldsymbol{\Phi}^{t}(\boldsymbol{\zeta})=(\phi_{\text{linear},1}(t),\phi_{\text{linear},2}(t)) \in \mathcal{P}$ associated with $\boldsymbol{\zeta}$ given by Lemma \ref{LemmaLinearFlowClosedFormula}. In particular, we have that
\begin{align}\label{philinear1}
 \phi_{\text{linear},1}^{(m)}(t) = \zeta_{1}^{(m)}  \cos(\omega_{m}t) + \frac{\zeta_2^{(m)}}{\omega_{m}} \sin (\omega_{m}t) ,
\end{align}
for all integers $m \geq 0$. To compute the time average
\begin{align*}
\langle f  \rangle (\boldsymbol{\zeta})= \frac{1}{2\pi} \int_{0}^{2\pi} f (\boldsymbol{\Phi}^{t}(\boldsymbol{\zeta}))	dt, 
\end{align*}
 we firstly plug \eqref{philinear1} into \eqref{fphi} and use trigonometric identities to expand all products $\cos(\omega_{l_1}t)\sin(\omega_{l_2}t)$, $\sin(\omega_{l_1}t)\sin(\omega_{l_2}t)$ and $\cos(\omega_{l_1}t)\cos(\omega_{l_2}t)$ with $l_1,l_2\in \{i,j,k,m\}$ in terms of $\cos((\omega_{i}\pm \omega_{j}\pm \omega_{k}\pm \omega_{m})t)$ and $\sin((\omega_{i}\pm \omega_{j}\pm \omega_{k}\pm \omega_{m})t)$. We obtain a long formula of the form
\begin{align}\label{ComplicatedResult}
	f(\boldsymbol{\Phi}^{t}(\boldsymbol{\zeta}))	=\sum_{T\in \{\sin,\cos\}} \sum_{(i,j,k,m)\in I_0 \cup I_1 \cup I_2}Z_{ijkm} C_{ijkm} T ( (\omega_{i}\pm \omega_{j}\pm \omega_{k}\pm \omega_{m})t ),  
\end{align}
where $Z_{ijkm}$ depends on $\zeta_1^{(l)}$ and $\zeta_2^{(l)}$ for all $l\in\{i,j,k,m \}$ and 
\begin{align*}
	I_0 &= \left\{(i,j,k,m):~ \text{no minus signs in } \omega_{i}\pm \omega_{j}\pm \omega_{k}\pm \omega_{m} \right\}, \\
	I_1 &= \left\{(i,j,k,m):~ \text{only one minus sign in } \omega_{i}\pm \omega_{j}\pm \omega_{k}\pm \omega_{m} \right\}, \\
	I_2 &= \left\{(i,j,k,m):~ \text{only two minus signs in } \omega_{i}\pm \omega_{j}\pm \omega_{k}\pm \omega_{m} \right\}.
\end{align*}
According to Lemma \ref{LemmaVanishingFourier}, we have that $ C_{`ijkm} =0$, for all $(i,j,k,m) \in I_1$, hence we only need to compute the sum with respect to $(i,j,k,m) \in I_0 \cup I_2$. Then, we apply the integral with respect to $t\in(0,2\pi)$ to each of the summands in \eqref{ComplicatedResult} and use the fact that
\begin{align*}
	\frac{1}{2\pi }\int_{0}^{2\pi } \cos(\omega t)dt = \mathds{1}(\omega=0), \quad 
	\frac{1}{2\pi }\int_{0}^{2\pi } \sin(\omega t)dt=0 
\end{align*}  
to compute each integral obtaining 
 \begin{align*}
 \langle f  \rangle (\boldsymbol{\zeta})=	\frac{1}{2\pi }\int_{0}^{2\pi }f(\boldsymbol{\Phi}^{t}(\boldsymbol{\zeta}))dt=\sum_{(i,j,k,m)\in  I_1 \cup I_2}Z_{ijkm} C_{ijkm} \mathds{1} ( \omega_{i}\pm \omega_{j}\pm \omega_{k}\pm \omega_{m} =0 ).	
 \end{align*}
In conclusion, we infer
\begin{align*} 
	\langle f  \rangle (\boldsymbol{\zeta})  
	&=\frac{1}{32}\sum_{i,j,k,m=0}^{\infty} C_{ijkm}  \Big[
	P_{ijkm}(\boldsymbol{\zeta})  \mathds{1}(m=i+j-k \geq 0)\nonumber \\
	&+Q_{ijkm}(\boldsymbol{\zeta})  \mathds{1}(m=i-j+k \geq 0)
	+R_{ijkm}(\boldsymbol{\zeta}) \mathds{1}(m=-i+j+k \geq 0)
	\Big],
\end{align*}
where
\begin{align*}
	P_{ijkm}(\boldsymbol{\zeta})  &= -\frac{\zeta_2^{  ( i )} \zeta_2^{  ( j )} \zeta_1^{  ( k )} \zeta_1^{  ( m )}}{\omega_{ i } \omega_{ j }}+\frac{\zeta_2^{  ( i )} \zeta_1^{  ( j )} \zeta_2^{  ( k )} \zeta_1^{  ( m )}}{\omega_{ i } \omega_{ k }}+\frac{\zeta_1^{  ( i )} \zeta_2^{  ( j )} \zeta_2^{  ( k )} \zeta_1^{  ( m )}}{\omega_{ j } \omega_{ k }} +\frac{\zeta_2^{  ( i )} \zeta_1^{  ( j )} \zeta_1^{  ( k )} \zeta_2^{  ( m )}}{\omega_{ i } \omega_{ m }} \\
	&+\frac{\zeta_1^{  ( i )} \zeta_2^{  ( j )} \zeta_1^{  ( k )} \zeta_2^{  ( m )}}{\omega_{ j } \omega_{ m }}-\frac{\zeta_1^{  ( i )} \zeta_1^{  ( j )} \zeta_2^{  ( k )} \zeta_2^{  ( m )}}{\omega_{ k } \omega_{ m }}+\zeta_1^{  ( i )} \zeta_1^{  ( j )} \zeta_1^{  ( k )} \zeta_1^{  ( m )}+\frac{\zeta_2^{  ( i )} \zeta_2^{  ( j )} \zeta_2^{  ( k )} \zeta_2^{  ( m )}}{\omega_{ i } \omega_{ j } \omega_{ k } \omega_{ m }}, \\
	Q_{ijkm}(\boldsymbol{\zeta})  &=  \frac{\zeta_2^{  ( i )} \zeta_2^{  ( j )} \zeta_1^{  ( k )} \zeta_1^{  ( m )}}{\omega_{ i } \omega_{ j }}-\frac{\zeta_2^{  ( i )} \zeta_1^{  ( j )} \zeta_2^{  ( k )} \zeta_1^{  ( m )}}{\omega_{ i } \omega_{ k }}+\frac{\zeta_1^{  ( i )} \zeta_2^{  ( j )} \zeta_2^{  ( k )} \zeta_1^{  ( m )}}{\omega_{ j } \omega_{ k }}+\frac{\zeta_2^{  ( i )} \zeta_1^{  ( j )} \zeta_1^{  ( k )} \zeta_2^{  ( m )}}{\omega_{ i } \omega_{ m }} \\
	&+\zeta_1^{  ( i )} \zeta_1^{  ( j )} \zeta_1^{  ( k )} \zeta_1^{  ( m )} -\frac{\zeta_1^{  ( i )} \zeta_2^{  ( j )} \zeta_1^{  ( k )} \zeta_2^{  ( m )}}{\omega_{ j } \omega_{ m }}+\frac{\zeta_1^{  ( i )} \zeta_1^{  ( j )} \zeta_2^{  ( k )} \zeta_2^{  ( m )}}{\omega_{ k } \omega_{ m }}+\frac{\zeta_2^{  ( i )} \zeta_2^{  ( j )} \zeta_2^{  ( k )} \zeta_2^{  ( m )}}{\omega_{ i } \omega_{ j } \omega_{ k } \omega_{ m }} , \\
	R_{ijkm}(\boldsymbol{\zeta})  &=  \frac{\zeta_2^{  ( i )} \zeta_2^{  ( j )} \zeta_1^{  ( k )} \zeta_1^{  ( m )}}{\omega_{ i } \omega_{ j }}+\frac{\zeta_2^{  ( i )} \zeta_1^{  ( j )} \zeta_2^{  ( k )} \zeta_1^{  ( m )}}{\omega_{ i } \omega_{ k }}-\frac{\zeta_1^{  ( i )} \zeta_2^{  ( j )} \zeta_2^{  ( k )} \zeta_1^{  ( m )}}{\omega_{ j } \omega_{ k }}-\frac{\zeta_2^{  ( i )} \zeta_1^{  ( j )} \zeta_1^{  ( k )} \zeta_2^{  ( m )}}{\omega_{ i } \omega_{ m }} \\
	&+\zeta_1^{  ( i )} \zeta_1^{  ( j )} \zeta_1^{  ( k )} \zeta_1^{  ( m )} +\frac{\zeta_1^{  ( i )} \zeta_1^{  ( j )} \zeta_2^{  ( k )} \zeta_2^{  ( m )}}{\omega_{ k } \omega_{ m }}+\frac{\zeta_1^{  ( i )} \zeta_2^{  ( j )} \zeta_1^{  ( k )} \zeta_2^{  ( m )}}{\omega_{ j } \omega_{ m }}+\frac{\zeta_2^{  ( i )} \zeta_2^{  ( j )} \zeta_2^{  ( k )} \zeta_2^{  ( m )}}{\omega_{ i } \omega_{ j } \omega_{ k } \omega_{ m }} .
\end{align*}
Finally, notice that from the formulas above one can immediately verify that
\begin{align*}
	Q_{ijkm}(\boldsymbol{\zeta})=P_{ikjm}(\boldsymbol{\zeta}), \quad 
	R_{ijkm}(\boldsymbol{\zeta})=P_{kjim}(\boldsymbol{\zeta}),
\end{align*}
for all integers $i,j,k,m \geq 0$. Moreover, we also have the symmetries of the Fourier coefficients $C_{ijkm}=C_{ikjm}=C_{kjim}$, for all integers $i,j,k,m \geq 0$, from \eqref{DefinitionOfFourierCoefficientsModelKG} and \eqref{DefinitionOfFourierCoefficientsModelWM}. Putting these together, we exchange $k$ with $j$ in the second sum,
\begin{align*}
	 \sum_{i,j,k,m=0}^{\infty} C_{ijkm}Q_{ijkm}(\boldsymbol{\zeta})  \mathds{1}(m=i-j+k \geq 0)  &=  
	 \sum_{i,j,k,m=0}^{\infty} C_{ikjm}Q_{ikjm}(\boldsymbol{\zeta})  \mathds{1}(m=i-k+j \geq 0) \\
	  &=
	 \sum_{i,j,k,m=0}^{\infty} C_{ijkm}P_{ijkm}(\boldsymbol{\zeta})  \mathds{1}(m=i-k+j \geq 0),
\end{align*}
and $k$ with $i$ in the third sum,
\begin{align*}
	 \sum_{i,j,k,m=0}^{\infty} C_{ijkm}R_{ijkm}(\boldsymbol{\zeta}) \mathds{1}(m=-i+j+k \geq 0) &=  
	 \sum_{i,j,k,m=0}^{\infty} C_{kjim}R_{kjim}(\boldsymbol{\zeta}) \mathds{1}(m=-k+j+i \geq 0) \\
	  &=
	 \sum_{i,j,k,m=0}^{\infty} C_{ijkm}P_{ijkm}(\boldsymbol{\zeta}) \mathds{1}(m=-k+j+i \geq 0),
\end{align*}
 to infer 
\begin{align*} 
	\langle f  \rangle (\boldsymbol{\zeta}) 
		&=\frac{3}{32}\sum_{i,j,k,m=0}^{\infty} C_{ijkm}   
	P_{ijkm}(\boldsymbol{\zeta})  \mathds{1}(m=i+j-k \geq 0) ,
\end{align*}
that completes the proof.
\end{proof}

Next, we verify Condition \eqref{Assumption1BambusiNekhoroshev} in Theorem \ref{TheoremBambusiNekhoroshev}. To do so, we use the method of Lagrange multipliers and show that the rescaled first 1-mode defined in \eqref{Definition1modesInitialdataBambusiNekhoroshev} is a critical point for the  function  $\langle f\rangle $ restricted on a hypersurface on constant harmonic energy.

\begin{lemma}[Condition \eqref{Assumption1BambusiNekhoroshev} in Theorem \ref{TheoremBambusiNekhoroshev}]\label{LemmaAssumption1BambusiNekhoroshev}
Let $\boldsymbol{\xi}=(\xi_1,\xi_2)$ be the rescaled first 1-mode  given by \eqref{Definition1modesInitialdataBambusiNekhoroshev} and let
\begin{align*}
	\Sigma= \left \{\boldsymbol{\zeta}\in \mathcal{P}:~  h_{\Omega}(\boldsymbol{\zeta})=\kappa_{0}^2 \omega_{0}^2 \right\}.
\end{align*}
be the hypersurface of constant harmonic energy $\kappa_{0}^2 \omega_{0}^2$. Then, $\boldsymbol{\xi} \in \Sigma$ is a critical point for the function $\langle f\rangle |_{\Sigma}$.
\end{lemma}
\begin{proof}
Pick any $\boldsymbol{\zeta}=(\zeta_1,\zeta_2) \in \mathcal{P}$, expand its coefficients in terms of the eigenfunctions $\{e_m:m\geq 0 \}$ to the linearized operator $L$,  
\begin{align*}
	\zeta_\alpha =\sum_{m=0}^{\infty} \zeta_\alpha^{(m)} e_m,\quad  \alpha \in \{1,2\},
\end{align*}
and consider the functions 
\begin{align*}
	h_{\Omega}(\boldsymbol{\zeta}) &=
	h_{\Omega}(\zeta_1,\zeta_2)=
	h_{\Omega} (\zeta_1^{(0)},\zeta_1^{(1)},\zeta_1^{(2)},\dots, \zeta_2^{(0)}, \zeta_2^{(1)}, \zeta_2^{(2)},\dots), \\
	\langle f  \rangle (\boldsymbol{\zeta})&=
	\langle f  \rangle (\zeta_1,\zeta_2)=
	\langle f  \rangle (\zeta_1^{(0)},\zeta_1^{(1)},\zeta_1^{(2)},\dots, \zeta_2^{(0)}, \zeta_2^{(1)}, \zeta_2^{(2)},\dots).
\end{align*}
These are given in the Fourier space by Lemmata \ref{LemmaHarmonicFlowWrtTheBasisem} and \ref{LemmaTimeAveraging} respectively. According to the method of Lagrange multipliers, the point $\boldsymbol{\zeta}$ is a critical point for $\langle f\rangle |_{\Sigma}$ if it satisfies the equations
\begin{align}\label{LagrangeMultipliers}
 \nabla_{\boldsymbol{\zeta}} \left[ \lambda  h_{\Omega}(\boldsymbol{\zeta}) +   \langle f  \rangle (\boldsymbol{\zeta}) \right] =0 \Longleftrightarrow
	\begin{dcases}
		\lambda \frac{\partial}{\partial \zeta_1^{(l)}} h_{\Omega}(\boldsymbol{\zeta})+
	\frac{\partial}{\partial \zeta_1^{(l)}}  \langle f  \rangle (\boldsymbol{\zeta})=0, \\
	\lambda \frac{\partial}{\partial \zeta_2^{(l)}}  h_{\Omega}(\boldsymbol{\zeta})+
	\frac{\partial}{\partial \zeta_2^{(l)}}  \langle f  \rangle (\boldsymbol{\zeta})=0,
	\end{dcases}
\end{align}
for all integers $l \geq 0$ and some $\lambda \in \mathbb{R}$.  On the one hand, we use Lemma \ref{LemmaHarmonicFlowWrtTheBasisem} to compute
\begin{align*}
	\frac{\partial}{\partial \zeta_1^{(l)}} h_{\Omega}(\boldsymbol{\zeta}) = 2\omega_{l}^2  \zeta_1^{(l)}, \quad 
	\frac{\partial}{\partial \zeta_2^{(l)}} h_{\Omega}(\boldsymbol{\zeta}) = 2   \zeta_2^{(l)},
\end{align*}
for all integers $l \geq 0$. On the other hand, we use Lemma \ref{LemmaTimeAveraging} together with the differentiation rule
\begin{align*}
	\frac{\partial}{\partial \zeta_{\alpha}^{(l)}} (\zeta_{\alpha}^{(i)} \zeta_{\alpha}^{(j)} ) =  
	\zeta_{\alpha}^{(i)}\frac{\partial}{\partial \zeta_{\alpha}^{(l)}} \zeta_{\alpha}^{(j)}+
	 \zeta_{\alpha}^{(j)}\frac{\partial}{\partial \zeta_{\alpha}^{(l)}}\zeta_{\alpha}^{(i)} =  
	\zeta_{\alpha}^{(i)} \mathds{1}(l=j) +
	 \zeta_{\alpha}^{(j)}\mathds{1}(l=i),
\end{align*}
for all integers $\alpha \in \{1,2\}$ and $i,j,l \geq 0$, to obtain that
\begin{align*}
	 \frac{\partial}{\partial \zeta_1^{(l)}} 	P_{ijkm}(\boldsymbol{\zeta}) 
\end{align*}
is given by
\begin{align*}
%
%
%
%
%
%
%
%
%
%
%
%
& -\frac{\zeta_2^{  ( i )} \zeta_2^{  ( j )}  }{\omega_{ i } \omega_{ j }}
    \zeta_1^{  ( m )}\mathds{1}(l=k)
 -\frac{\zeta_2^{  ( i )} \zeta_2^{  ( j )}  }{\omega_{ i } \omega_{ j }} \zeta_1^{  ( k )} \mathds{1}(l=m) 
 +\frac{\zeta_2^{  ( i )}  \zeta_2^{  ( k )}  }{\omega_{ i } \omega_{ k }}
    \zeta_1^{  ( m )}\mathds{1}(l=j)
  +\frac{\zeta_2^{  ( i )}  \zeta_2^{  ( k )}  }{\omega_{ i } \omega_{ k }}\zeta_1^{  ( j )} \mathds{1}(l=m)   \\
 & +\frac{ \zeta_2^{  ( j )} \zeta_2^{  ( k )} }{\omega_{ j } \omega_{ k }} 
  \zeta_1^{  ( m )}  \mathds{1}(l=i)
  +\frac{ \zeta_2^{  ( j )} \zeta_2^{  ( k )} }{\omega_{ j } \omega_{ k }}   \zeta_1^{  ( i )} \mathds{1}(l=m)   
  +\frac{\zeta_2^{  ( i )}   \zeta_2^{  ( m )}}{\omega_{ i } \omega_{ m }} 
  \zeta_1^{  ( k )} \mathds{1}(l=j) 
  +\frac{\zeta_2^{  ( i )}   \zeta_2^{  ( m )}}{\omega_{ i } \omega_{ m }}   \zeta_1^{  ( j )} \mathds{1}(l=k)   \\
& +\frac{ \zeta_2^{  ( j )}  \zeta_2^{  ( m )}}{\omega_{ j } \omega_{ m }}
   \zeta_1^{  ( k )} \mathds{1}(l=i) 
  +\frac{ \zeta_2^{  ( j )}  \zeta_2^{  ( m )}}{\omega_{ j } \omega_{ m }} \zeta_1^{  ( i )}  \mathds{1}(l=k)  
-\frac{ \zeta_2^{  ( k )} \zeta_2^{  ( m )}}{\omega_{ k } \omega_{ m }}
  \zeta_1^{  ( j )}  \mathds{1}(l=i)
 -\frac{ \zeta_2^{  ( k )} \zeta_2^{  ( m )}}{\omega_{ k } \omega_{ m }}  \zeta_1^{  ( i )} \mathds{1}(l=j)    \\
& +   
  \zeta_1^{  ( j )} \zeta_1^{  ( k )} \zeta_1^{  ( m )}\mathds{1}(l=i)
 +\zeta_1^{  ( i )}  \zeta_1^{  ( k )} \zeta_1^{  ( m )}\mathds{1}(l=j)
 +\zeta_1^{  ( i )} \zeta_1^{  ( j )} \zeta_1^{  ( m )}\mathds{1}(l=k)
 +\zeta_1^{  ( i )} \zeta_1^{  ( j )} \zeta_1^{  ( k )} \mathds{1}(l=m) , \\
\end{align*}
and
\begin{align*}
	\frac{\partial}{\partial \zeta_2^{(l)}} 	P_{ijkm}(\boldsymbol{\zeta})  
\end{align*}
is given by
\begin{align*}
%
%
& -\frac{ \zeta_1^{  ( k )} \zeta_1^{  ( m )}}{\omega_{ i } \omega_{ j }}
  \zeta_2^{  ( j )} \mathds{1}(l=i)
 -\frac{ \zeta_1^{  ( k )} \zeta_1^{  ( m )}}{\omega_{ i } \omega_{ j }} \zeta_2^{  ( i )}\mathds{1}(l=j) 
+\frac{ \zeta_1^{  ( j )}  \zeta_1^{  ( m )}}{\omega_{ i } \omega_{ k }} 
  \zeta_2^{  ( k )}  \mathds{1}(l=i)
+\frac{ \zeta_1^{  ( j )}  \zeta_1^{  ( m )}}{\omega_{ i } \omega_{ k }}  \zeta_2^{  ( i )}\mathds{1}(l=k)   \\
&+\frac{\zeta_1^{  ( i )}  \zeta_1^{  ( m )}}{\omega_{ j } \omega_{ k }} 
   \zeta_2^{  ( k )}  \mathds{1}(l=j) 
 +\frac{\zeta_1^{  ( i )}  \zeta_1^{  ( m )}}{\omega_{ j } \omega_{ k }}   \zeta_2^{  ( j )} \mathds{1}(l=k)  
 +\frac{ \zeta_1^{  ( j )} \zeta_1^{  ( k )} }{\omega_{ i } \omega_{ m }}
 \zeta_2^{  ( m )}  \mathds{1}(l=i)
 +\frac{ \zeta_1^{  ( j )} \zeta_1^{  ( k )} }{\omega_{ i } \omega_{ m }} \zeta_2^{  ( i )}\mathds{1}(l=m)    \\ 
&+\frac{\zeta_1^{  ( i )}  \zeta_1^{  ( k )} }{\omega_{ j } \omega_{ m }}
   \zeta_2^{  ( m )}  \mathds{1}(l=j)
  +\frac{\zeta_1^{  ( i )}  \zeta_1^{  ( k )} }{\omega_{ j } \omega_{ m }}  \zeta_2^{  ( j )}\mathds{1}(l=m)  
-\frac{\zeta_1^{  ( i )} \zeta_1^{  ( j )} }{\omega_{ k } \omega_{ m }}
   \zeta_2^{  ( m )}  \mathds{1}(l=k)
 -\frac{\zeta_1^{  ( i )} \zeta_1^{  ( j )} }{\omega_{ k } \omega_{ m }}  \zeta_2^{  ( k )} \mathds{1}(l=m)   \\
&+\frac{  \zeta_2^{  ( j )} \zeta_2^{  ( k )} \zeta_2^{  ( m )}}{\omega_{ i } \omega_{ j } \omega_{ k } \omega_{ m }}\mathds{1}(l=i)
+\frac{\zeta_2^{  ( i )}  \zeta_2^{  ( k )} \zeta_2^{  ( m )}}{\omega_{ i } \omega_{ j } \omega_{ k } \omega_{ m }}\mathds{1}(l=j) 
+\frac{ \zeta_2^{  ( i )} \zeta_2^{  ( j )}  \zeta_2^{  ( m )}}{\omega_{ i } \omega_{ j } \omega_{ k } \omega_{ m }} \mathds{1}(l=k)
+\frac{\zeta_2^{  ( i )} \zeta_2^{  ( j )} \zeta_2^{  ( k )}}{\omega_{ i } \omega_{ j } \omega_{ k } \omega_{ m }}  \mathds{1}(l=m),
\end{align*}
for all integers $i,j,k,m,l \geq 0$. Next, we restrict our attention to the rescaled first 1-mode $\boldsymbol{\xi}=(\xi_1,\xi_2)$ defined by \eqref{Definition1modesInitialdataBambusiNekhoroshev}. In this case, the latter yields
\begin{align*}
	 \frac{\partial}{\partial \zeta_1^{(l)}} 	P_{ijkm}(\boldsymbol{\xi})  &= \kappa_{0}^3 \Big[
 \mathds{1}(k=j=m=0)  \mathds{1}(l=i)
 + \mathds{1}(i=k=m=0)\mathds{1}(l=j) \\
 &+ \mathds{1}(i=j=m=0)\mathds{1}(l=k)
 + \mathds{1}(i=j=k=0) \mathds{1}(l=m) \Big] , \\
  \frac{\partial}{\partial \zeta_2^{(l)}} 	P_{ijkm}(\boldsymbol{\xi})  &=
 0,
\end{align*}
for all integers $i,j,k,m,l \geq 0$. Now, one can see that there exists a choice of $\lambda$ such that both equations in \eqref{LagrangeMultipliers} hold true for the rescaled first 1-mode. Indeed, the second equation in \eqref{LagrangeMultipliers} is trivially satisfied for all $\lambda$ 
 whereas the first equation in \eqref{LagrangeMultipliers} is satisfied provided that 
\begin{align}\label{ValueLagrangeMultiplierLambda}
	\lambda=  - \frac{1}{2}.
\end{align}
Indeed, for all integers $l \geq 0$, we have that 
\begin{align*}
   -2\lambda \omega_{0}^2 \kappa_{0} \mathds{1}(l=0) &=	-2\lambda \omega_{l}^2  \xi_1^{(l)} =\frac{\partial  \langle f  \rangle (\boldsymbol{\xi}) }{\partial \zeta_1^{(l)}}  \\
	&=\frac{3}{32}\sum_{i,j,k,m=0}^{\infty} C_{ijkm}   
	\frac{\partial  P_{ijkm}(\boldsymbol{\xi}) }{\partial \zeta_1^{(l)}} \mathds{1}(m=i+j-k \geq 0) \\
	&=\frac{3}{32}\kappa_{0} ^3 \sum_{i=0}^{\infty} C_{i0 0 0}   
	   \mathds{1}(l=i =0) +
	   \frac{3}{32}\kappa_{0} ^3\sum_{j=0}^{\infty} C_{0 j 0 0}   
	   \mathds{1}( l= j=0  )\\
	&+\frac{3}{32}\kappa_{0} ^3\sum_{k=0}^{\infty} C_{0 0 k 0}   
	   \mathds{1}(l=k=0 ) 
	 +\frac{3}{32}\kappa_{0} ^3\sum_{ m=0}^{\infty} C_{0 0 0 m}   
	  \mathds{1}(l=m=0  ) \\
	 &= \frac{3}{8} \kappa_{0} ^3\mathds{1}(l =0  )   C_{0 0 0 0} \Longleftrightarrow  \\
	 -2\lambda \omega_{0}^2  \mathds{1}(l=0) &= \frac{3}{8} \kappa_{0} ^2 \mathds{1}(l =0  )   C_{0 0 0 0} =\frac{3}{8}     \frac{8\omega_{0}^2}{3C_{00 0 0}   }\mathds{1}(l=0) C_{0 0 0 0} \Longleftrightarrow \lambda=-\frac{1}{2},
\end{align*}
where we also used the definition of the rescaled constant $\kappa_{0}$ from \eqref{Definition1modeinitialdata}. Finally, we use Lemma \ref{LemmaHarmonicFlowWrtTheBasisem} to compute the harmonic energy of the rescaled 1--modes and show that $\boldsymbol{\xi} \in \Sigma$. We obtain
\begin{align*} 
	 h_{\Omega}(\boldsymbol{\xi})
	 & = \sum_{m=0}^{\infty} 
	 ( \xi_2^{(m)} )^2  +\sum_{m=0}^{\infty} \omega_{m}^2  ( \xi_1^{(m)})^2 
	  = \sum_{m=0}^{\infty} \omega_{m}^2  \left( \kappa_{0}\mathds{1}(m=0) \right)^2 = \kappa_{0}^2 \omega_{0}^2,
\end{align*}
that completes the proof. 
\end{proof}

Next, we show that the rescaled first 1-mode generates a curve of critical points.

\begin{lemma}[Critical curve]\label{LemmaCriticalCurve}
Let $\boldsymbol{\xi}=(\xi_1,\xi_2)$ be the rescaled first 1-mode  given by \eqref{Definition1modesInitialdataBambusiNekhoroshev}, 
\begin{align*}
	\Sigma= \left \{\boldsymbol{\zeta}\in \mathcal{P}:~  h_{\Omega}(\boldsymbol{\zeta})=\kappa_{0}^2 \omega_{0}^2 \right\}
\end{align*}
 the hypersurface of constant harmonic energy and 
\begin{align*}
	\Gamma_{\text{linear}}(\boldsymbol{\xi}) = \left \{\boldsymbol{\Phi}^{t}(\boldsymbol{\xi}):t \in \mathbb{R} \right\}
	=\left \{  
	\begin{pmatrix}
		 \kappa_{0}  \cos(\omega_{0}t)   e_{0} \\
		- \omega_{0}\kappa_{0}  \sin(\omega_{0}t) e_{0} 
	\end{pmatrix}
	:t \in \mathbb{R} \right\}
\end{align*} 
 the trajectory of the solution to the linearized equation with $\boldsymbol{\xi} $ as initial data. Then,  $ \Gamma_{\text{linear}}(\boldsymbol{\xi}) \subset \Sigma$ and every point $\boldsymbol{\eta}\in \Gamma_{\text{linear}}(\boldsymbol{\xi}) $ is a critical point for the function $\langle f\rangle |_{\Sigma}$.
\end{lemma}
\begin{proof}
It follows from the fact that both $h_{\Omega}$ and $\langle f\rangle |_{\Sigma}$ are invariant under the linear flow.  
\end{proof}

Next, we study the tangent space of the hypersurface of constant  harmonic energy.

\begin{lemma}[Tangent space of $\Sigma$]\label{LemmaTangentSpaceofSigma}
Let $\boldsymbol{\xi}=(\xi_1,\xi_2)$ be the rescaled first 1-mode  given by \eqref{Definition1modesInitialdataBambusiNekhoroshev}, 
\begin{align*}
	\Sigma= \left \{\boldsymbol{\zeta}\in \mathcal{P}:~  h_{\Omega}(\boldsymbol{\zeta})=\kappa_{0}^2 \omega_{0}^2 \right\}
\end{align*}
 the hypersurface of constant  harmonic energy and 
\begin{align*}
	\Gamma_{\text{linear}}(\boldsymbol{\xi}) = \left \{\boldsymbol{\Phi}^{t}(\boldsymbol{\xi}):t \in \mathbb{R} \right\}
	=\left \{  
	\begin{pmatrix}
		 \kappa_{0}  \cos(\omega_{0}t)   e_{0} \\
		- \omega_{0}\kappa_{0}  \sin(\omega_{0}t) e_{0} 
	\end{pmatrix}
	:t \in  \mathbb{R}  \right\} \subset \Sigma
\end{align*} 
 the trajectory of the solution to the linearized equation with $\boldsymbol{\xi} $ as initial data. Then, for any point $\boldsymbol{\eta}\in \Gamma_{\text{linear}}(\boldsymbol{\xi})$, the tangent space of $  \Sigma$ is given by
\begin{align*}
	 T_{\boldsymbol{\eta}}\Sigma   =
	\left \{\boldsymbol{X} = 
	\begin{pmatrix}
		\sum_{\substack{ m=0 \ }}^{\infty} X_{1}^{(m)} e_{m} \\
		\sum_{\substack{ m=0  }}^{\infty} X_{2}^{(m)} e_{m}
	\end{pmatrix} :~ \sin (\omega_{0}t)
 	  X_{2}^{(0)}  =\omega_{0} \cos (\omega_{0}t)X_{1}^{(0)} \right \}.
\end{align*}
Moreover, it splits into
\begin{align*} 
	T_{\boldsymbol{\eta}}\Sigma =T_{\boldsymbol{\eta}}\Gamma_{\text{linear}}(\boldsymbol{\xi}) \oplus \left(T_{\boldsymbol{\eta}}\Gamma_{\text{linear}}  (\boldsymbol{\xi})\right)^{\perp},
\end{align*}
where 
\begin{align*} 
	T_{\boldsymbol{\eta}}\Gamma_{\text{linear}}(\boldsymbol{\xi}) & = \left \{\boldsymbol{X} \in T_{\boldsymbol{\eta}} \Sigma:~\boldsymbol{X}=
	\mu_{0}
	\begin{pmatrix}
		  \sin(\omega_{0}t)   e_{0} \\
		  \omega_{0} \cos(\omega_{0}t) e_{0} 
	\end{pmatrix},~\mu_{0} \in \mathbb{R}
	  \right\}, \\
	\left(T_{\boldsymbol{\eta}}\Gamma_{\text{linear}} (\boldsymbol{\xi}) \right)^{\perp}  &  =\left \{\boldsymbol{X} \in T_{\boldsymbol{\eta}} \Sigma:~ X_{1}^{(0)}=X_{2}^{(0)}=0 \right \}  .
\end{align*}
\end{lemma}
\begin{proof}
Let $t \in \mathbb{R} $, set  $\boldsymbol{\eta}=\boldsymbol{\Phi}^{t}(\boldsymbol{\xi}) $ and consider the function
\begin{align*}
	h_{\Omega}(\boldsymbol{\eta})=h_{\Omega}(\eta_1,\eta_2)  
	 =\int_{0}^{\pi/2}\left( 
	 \eta_2 ^2 + \left(L^{1/2}\eta_{1}\right)^2 \right)  d\mu
\end{align*}
studied in Lemma \ref{LemmaHarmonicFlowWrtTheBasisem}. Lemma \ref{LemmaCriticalCurve} yields that $\boldsymbol{\eta}=\boldsymbol{\Phi}^{t}(\boldsymbol{\xi})\in \Gamma_{\text{linear}}(\boldsymbol{\xi}) \subset \Sigma$. A normal vector to $\Sigma$ at the point $\boldsymbol{\eta}$ is given by the gradient $\nabla_{\boldsymbol{\eta}} h_{\Omega}(\boldsymbol{\eta}) $ and hence
\begin{align*}
	 T_{\boldsymbol{\eta}}\Sigma   &= 
	\left \{\boldsymbol{X} \in T_{\boldsymbol{\eta}} \mathcal{P}:~ \langle \nabla_{\boldsymbol{\eta}} h_{\Omega}(\boldsymbol{\eta}),\boldsymbol{X} \rangle =0 \right \} 
	=
	\left \{\boldsymbol{X} \in T_{\boldsymbol{\eta}} \mathcal{P}:~ d h_{\Omega}(\boldsymbol{\eta})(\boldsymbol{X})   =0 \right \}.
\end{align*}
 Let $\boldsymbol{X}=(X_1,X_2) \in T_{\boldsymbol{\eta}} \mathcal{P}$ be expanded in terms of the eigenfunctions $\{e_m: m \geq 0\}$ to the linearization operator $L$, 
 \begin{align*}
 	X_{\alpha}=\sum_{m=0}^{\infty} X_{\alpha}^{(m)} e_{m}, \quad \alpha \in \{1,2\}.
 \end{align*}
 We compute
 \begin{align*}
 	d h_{\Omega}(\boldsymbol{\eta})(\boldsymbol{X}) &=
 	\int_{0}^{\pi/2}\left( 
	 2\eta_2 X_2  +2  L^{1/2}\eta_{1}  L^{1/2}X_{1}  \right)  d\mu \\
	& =2
 	\int_{0}^{\pi/2}\left( 
	 \eta_2 X_2  +  X_{1} L \eta_{1}   \right)  d\mu \\
	&= -2\omega_{0}\kappa_{0}\sin (\omega_{0}t)
 	\int_{0}^{\pi/2}  
	 e_{0} X_2 d\mu +2\kappa_{0}\cos (\omega_{0}t)  \int_{0}^{\pi/2}  X_{1} L  e_{0}     d\mu \\
	 &= -2\omega_{0}\kappa_{0}\sin (\omega_{0}t)
 	\sum_{m=0}^{\infty} X_{2}^{(m)}\int_{0}^{\pi/2}  
	 e_{0}  e_{m} d\mu +2\kappa_{0}\cos (\omega_{0}t)\sum_{m=0}^{\infty} X_{1}^{(m)}  \int_{0}^{\pi/2} e_{m} L  e_{0}     d\mu \\
	 &= -2\omega_{0}\kappa_{0}\sin (\omega_{0}t)
 	\sum_{m=0}^{\infty} X_{2}^{(m)} \mathds{1}(m=0) +2\omega_{0}^2 \kappa_{0}\cos (\omega_{0}t)\sum_{m=0}^{\infty} X_{1}^{(m)}  \mathds{1}(m=0) \\
 	&= -2\omega_{0}\kappa_{0}\sin (\omega_{0}t)
 	  X_{2}^{(0)}  +2\omega_{0}^2\kappa_{0}\cos (\omega_{0}t)X_{1}^{(0)} \\
 	  &=2\omega_{0}\kappa_{0} \left[ -\sin (\omega_{0}t)
 	  X_{2}^{(0)}  +\omega_{0} \cos (\omega_{0}t)X_{1}^{(0)} \right]
 \end{align*} 
and hence
\begin{align*}
	 T_{\boldsymbol{\eta}}\Sigma   =
	\left \{\boldsymbol{X} = 
	\begin{pmatrix}
		\sum_{\substack{ m=0 \ }}^{\infty} X_{1}^{(m)} e_{m} \\
		\sum_{\substack{ m=0  }}^{\infty} X_{2}^{(m)} e_{m}
	\end{pmatrix} :~ \sin (\omega_{0}t)
 	  X_{2}^{(0)}  =\omega_{0} \cos (\omega_{0}t)X_{1}^{(0)} \right \}.
\end{align*}
Moreover, the vectors tangent to the curve $\Gamma_{\text{linear}}(\boldsymbol{\xi}) = \left \{\boldsymbol{\Phi}^{t}(\boldsymbol{\xi}):t \in \mathbb{R} \right\}$ are proportional to
\begin{align*}
	\partial_{t}\boldsymbol{\Phi}^{t}(\boldsymbol{\xi})=\partial_{t}\begin{pmatrix}
		 \kappa_{0}  \cos(\omega_{0}t)   e_{0} \\
		- \omega_{0}\kappa_{0}  \sin(\omega_{0}t) e_{0} 
	\end{pmatrix}=
	\begin{pmatrix}
		- \kappa_{0} \omega_{0} \sin(\omega_{0}t)   e_{0} \\
		- \kappa_{0} \omega_{0}^2 \cos(\omega_{0}t) e_{0} 
	\end{pmatrix} \in T_{\boldsymbol{\eta}}\Gamma_{\text{linear}} (\boldsymbol{\xi}) \subset   T_{\boldsymbol{\eta}}\Sigma
\end{align*} 
and hence
\begin{align*}
	T_{\boldsymbol{\eta}}\Gamma_{\text{linear}}(\boldsymbol{\xi}) = \left \{\boldsymbol{X} \in T_{\boldsymbol{\eta}} \Sigma:~\boldsymbol{X}=
	\mu_{0}
	\begin{pmatrix}
		  \sin(\omega_{0}t)   e_{0} \\
		  \omega_{0} \cos(\omega_{0}t) e_{0} 
	\end{pmatrix},~\mu_{0} \in \mathbb{R}
	  \right\}.
\end{align*}
 Finally, we compute
 \begin{align*}
 	\left(T_{\boldsymbol{\eta}}\Gamma_{\text{linear}} (\boldsymbol{\xi}) \right)^{\perp}  
 	&= \left\{\boldsymbol{X} \in T_{\boldsymbol{\eta}} \Sigma:~ \langle \boldsymbol{X},\boldsymbol{Y}\rangle =0,~\forall \boldsymbol{Y} \in T_{\boldsymbol{\eta}}\Gamma_{\text{linear}}(\boldsymbol{\xi}) \subset   T_{\boldsymbol{\eta}}\Sigma
 	\right \} \\
 	&= \left\{\boldsymbol{X} \in T_{\boldsymbol{\eta}} \Sigma:~ \langle \boldsymbol{X},\partial_{t}\boldsymbol{\Phi}^{t}(\boldsymbol{\xi})\rangle =0 \right \} \\
 	&= \left\{\boldsymbol{X} \in T_{\boldsymbol{\eta}} \Sigma:~\left  \langle \boldsymbol{X},
 	\begin{pmatrix}
		  \sin(\omega_{0}t)   e_{0} \\
		  \omega_{0} \cos(\omega_{0}t) e_{0} 
	\end{pmatrix} \right \rangle =0 \right \}.
 \end{align*}
For every $\boldsymbol{X}=(X_1,X_2) \in T_{\boldsymbol{\eta}} \Sigma$, we use the orthogonality condition
\begin{align*}
	\left  \langle e_m,   e_{0}\right \rangle_{L^2 ((0,\pi/2);d\mu )} &=(e _{m}  |e _{0}  )=  \mathds{1}(m=0),
\end{align*}
for all integers $n,m \geq 0$, derived in \eqref{OrthogonalityoftheEigenfunctions1}, as well as the fact that 
\begin{align*}
	\left  \langle e_m,   e_{0}\right \rangle_{H^1((0,\pi/2);d\mu )} &=(e^{\prime}_{m}  |e^{\prime}_{0}  )+(e _{m}  |e _{0}  )=  \mathds{1}(m=0),
\end{align*}
since $e_0$ is constant in the KG model, to compute  
\begin{align*}
	& \left \langle \boldsymbol{X},
 	\begin{pmatrix}
		  \sin(\omega_{0}t)   e_{0} \\
		  \omega_{0} \cos(\omega_{0}t) e_{0} 
	\end{pmatrix} \right \rangle =\sin(\omega_{0}t) 
	 \left  \langle X_1,   e_{0}\right \rangle_{H^1((0,\pi/2);d\mu )}+
	  \omega_{0} \cos(\omega_{0}t) 
	  \left  \langle X_2,e_{0} \right \rangle_{L^2((0,\pi/2);d\mu )} \\
	 & =\sin(\omega_{0}t) 
	\sum_{m=0}^{\infty}X_{1}^{(m)} \left  \langle e_m,   e_{0}\right \rangle_{H^1((0,\pi/2);d\mu )}+
	  \omega_{0} \cos(\omega_{0}t) 
	  \sum_{m=0}^{\infty}X_{2}^{(m)} 
	  \left  \langle e_m ,e_{0} \right \rangle_{L^2((0,\pi/2);d\mu )} \\
	  &= \sin(\omega_{0}t) 
	 X_{1}^{(0)} +
	  \omega_{0} \cos(\omega_{0}t) 
	 X_{2}^{(0)}  .
\end{align*}
Hence, the tangent vectors  $\boldsymbol{X}  \in   T_{\boldsymbol{\eta}} \Sigma \cap \left(T_{\boldsymbol{\eta}}\Gamma_{\text{linear}}(\boldsymbol{\xi})  \right)^{\perp} $ are given by the linear system
\begin{align*}
	\begin{pmatrix}
		 \omega_{0} \cos(\omega_{0}t)  &
	- \sin(\omega_{0}t) \\
	   \sin  (\omega_{0}t) &
	   \omega_{0} \cos(\omega_{0}t)  
	\end{pmatrix}
	\begin{pmatrix}
		X_{1}^{(0)} \\ X_{2}^{(0)}
	\end{pmatrix}=
	\begin{pmatrix}
		0 \\ 0
	\end{pmatrix}.
\end{align*} 
The determinant is given by $\omega_{0}^2 \cos^2 (\omega_{0}t)+\sin^2 (\omega_{0}t) $ and it is non-zero for all $t \in \mathbb{R}$ that leads us to the unique solution $X_{1}^{(0)}=X_{2}^{(0)}=0$. We conclude that
 \begin{align*}
 	\left(T_{\boldsymbol{\zeta}}\Gamma_{\text{linear}} (\boldsymbol{\xi}) \right)^{\perp} =\left \{\boldsymbol{X} = 
	\begin{pmatrix}
		\sum_{\substack{ m=0 \ }}^{\infty} X_{1}^{(m)} e_{m} \\
		\sum_{\substack{ m=0  }}^{\infty} X_{2}^{(m)} e_{m}
	\end{pmatrix} :~ X_{1}^{(0)}=X_{2}^{(0)}=0 \right \},
 \end{align*}
that completes the proof. 
\end{proof}

 Next, we compute the second differential of $ \langle f\rangle $ restricted on a hypersurface of constant hormonic energy and evaluated at the rescaled first 1--mode.
 
 \begin{lemma}[Second differential at the critical points]\label{LemmaSecondDifferentialatCriticalPointComputation}
Let $\boldsymbol{\xi}=(\xi_1,\xi_2)$ be the rescaled first 1-mode  given by \eqref{Definition1modesInitialdataBambusiNekhoroshev} and let
\begin{align*}
	\Sigma= \left \{\boldsymbol{\zeta}\in \mathcal{P}:~  h_{\Omega}(\boldsymbol{\zeta})=\kappa_{0}^2 \omega_{0}^2 \right\}
\end{align*}
be the hypersurface of constant  harmonic energy. Then, for any tangent vector 
\begin{align*}
	\boldsymbol{X}=(V,W)=\left(\sum_{m=0}^{\infty}V^{(m)}e_{m},\sum_{m=0}^{\infty}W^{(m)}e_{m} \right)  \in T_{\boldsymbol{\xi}}\Sigma,
\end{align*}
we have that
\begin{align*}
	  d^2  ( \langle f\rangle |_{\Sigma} )(\boldsymbol{\xi})(\boldsymbol{X},\boldsymbol{X}) & =  - \frac{\omega_{0}^2}{C_{0 0 0 0}} \Bigg[  \sum_{l=0}^{\infty} \Bigg(  \frac{C_{0 0 0 0}}{\omega_{0}^2} - 2 \frac{ C_{l l0   0       } }{ \omega_{l}^2} \Bigg) ( \omega_{l}V^{(l)})^2 
   +\sum_{l=0}^{\infty} \Bigg( \frac{C_{0 0 0 0}}{\omega_{0}^2} -  2\frac{  C_{ l l 0   0  } }{\omega_{l}^2     }    \Bigg)(W^{(l)})^2   \Bigg] \\
    &+\omega_0^2 (V^{(0)})^2
	-     \left(   W^{(0)}  \right)^2.
\end{align*} 
\end{lemma}
\begin{proof}
	Pick any $\boldsymbol{\zeta}=(\zeta_1,\zeta_2) \in \mathcal{P}$, expand its coefficients in terms of the eigenfunctions $\{e_m:m\geq 0 \}$ to the linearized operator $L$,  
\begin{align*}
	\zeta_\alpha =\sum_{m=0}^{\infty} \zeta_\alpha^{(m)} e_m,\quad  \alpha \in \{1,2\},
\end{align*}
and consider the functions 
\begin{align*}
	h_{\Omega}(\boldsymbol{\zeta}) &=
	h_{\Omega}(\zeta_1,\zeta_2)=
	h_{\Omega} (\zeta_1^{(0)},\zeta_1^{(1)},\zeta_1^{(2)},\dots, \zeta_2^{(0)}, \zeta_2^{(1)}, \zeta_2^{(2)},\dots), \\
	\langle f  \rangle (\boldsymbol{\zeta})&=
	\langle f  \rangle (\zeta_1,\zeta_2)=
	\langle f  \rangle (\zeta_1^{(0)},\zeta_1^{(1)},\zeta_1^{(2)},\dots, \zeta_2^{(0)}, \zeta_2^{(1)}, \zeta_2^{(2)},\dots),
\end{align*}
given by Lemma \ref{LemmaHarmonicFlowWrtTheBasisem} and \ref{LemmaTimeAveraging} respectively. Firstly, we use the computations for
\begin{align*}
	\frac{\partial}{\partial \zeta_1^{(l)}} 	P_{ijkm}(\boldsymbol{\zeta}) , \quad \frac{\partial}{\partial \zeta_2^{(l)}} 	P_{ijkm}(\boldsymbol{\zeta}) 
\end{align*}
for all integers $l \geq 0$, derived in Lemma \ref{LemmaAssumption1BambusiNekhoroshev}, to compute
\begin{align*} 
	  \frac{\partial^2}{\partial \zeta_1^{(l)} \zeta_1^{(n)}  } 	P_{ijkm}(\boldsymbol{\zeta}) , \quad  
	  \frac{\partial^2}{\partial \zeta_1^{(l)} \zeta_2^{(n)}   } 	P_{ijkm}(\boldsymbol{\zeta}) , \quad  
	  \frac{\partial^2}{\partial \zeta_2^{(l)} \zeta_2^{(n)}  } 	P_{ijkm}(\boldsymbol{\zeta}), 
\end{align*}
for all integers $l,n \geq 0$, and for generic $\boldsymbol{\zeta}=(\zeta_1,\zeta_2)$. We obtain that 
\begin{align*}
	\frac{\partial^2}{\partial \zeta_1^{(l)}\zeta_1^{(n)} } 	P_{ijkm}(\boldsymbol{\zeta})  
\end{align*}
is given by
\begin{align*}
 &-\frac{\zeta_2^{  ( i )} \zeta_2^{  ( j )}  }{\omega_{ i } \omega_{ j }}
    \mathds{1}(n=m) \mathds{1}(l=k) 
 -\frac{\zeta_2^{  ( i )} \zeta_2^{  ( j )}  }{\omega_{ i } \omega_{ j }}  \mathds{1}(n=k) \mathds{1}(l=m) 
 +\frac{\zeta_2^{  ( i )}  \zeta_2^{  ( k )}  }{\omega_{ i } \omega_{ k }}
    \mathds{1}(n=m)\mathds{1}(l=j) \\
 & +\frac{\zeta_2^{  ( i )}  \zeta_2^{  ( k )}  }{\omega_{ i } \omega_{ k }} \mathds{1}(n=j) \mathds{1}(l=m) 
  +\frac{ \zeta_2^{  ( j )} \zeta_2^{  ( k )} }{\omega_{ j } \omega_{ k }} 
   \mathds{1}(n=m) \mathds{1}(l=i)
  +\frac{ \zeta_2^{  ( j )} \zeta_2^{  ( k )} }{\omega_{ j } \omega_{ k }}    \mathds{1}(n=i) \mathds{1}(l=m)   \\
&  +\frac{\zeta_2^{  ( i )}   \zeta_2^{  ( m )}}{\omega_{ i } \omega_{ m }} 
  \mathds{1}(n=k) \mathds{1}(l=j) 
  +\frac{\zeta_2^{  ( i )}   \zeta_2^{  ( m )}}{\omega_{ i } \omega_{ m }}    \mathds{1}(n= j) \mathds{1}(l=k) 
  +\frac{ \zeta_2^{  ( j )}  \zeta_2^{  ( m )}}{\omega_{ j } \omega_{ m }}
   \mathds{1}(n= k) \mathds{1}(l=i)  \\
&  +\frac{ \zeta_2^{  ( j )}  \zeta_2^{  ( m )}}{\omega_{ j } \omega_{ m }} \mathds{1}(n= i) \mathds{1}(l=k)  
-\frac{ \zeta_2^{  ( k )} \zeta_2^{  ( m )}}{\omega_{ k } \omega_{ m }}
 \mathds{1}(n= j)  \mathds{1}(l=i)
 -\frac{ \zeta_2^{  ( k )} \zeta_2^{  ( m )}}{\omega_{ k } \omega_{ m }}  \mathds{1}(n= i) \mathds{1}(l=j)    \\
& +  \mathds{1}(n= j) \zeta_1^{  ( k )} \zeta_1^{  ( m )}\mathds{1}(l=i)
 +  \zeta_1^{  ( j )} \mathds{1}(n= k) \zeta_1^{  ( m )}\mathds{1}(l=i)
  +  \zeta_1^{  ( j )} \zeta_1^{  ( k )} \mathds{1}(n= m)\mathds{1}(l=i) \\
 &+\mathds{1}(n= i)  \zeta_1^{  ( k )} \zeta_1^{  ( m )}\mathds{1}(l=j)
 +\zeta_1^{  ( i )} \mathds{1}(n= k) \zeta_1^{  ( m )}\mathds{1}(l=j)
 +\zeta_1^{  ( i )}  \zeta_1^{  ( k )} \mathds{1}(n= m)\mathds{1}(l=j) \\
& +\mathds{1}(n= i) \zeta_1^{  ( j )} \zeta_1^{  ( m )}\mathds{1}(l=k)
  +\zeta_1^{  ( i )} \mathds{1}(n= j) \zeta_1^{  ( m )}\mathds{1}(l=k)
   +\zeta_1^{  ( i )} \zeta_1^{  ( j )} \mathds{1}(n= m)\mathds{1}(l=k) \\
& +\mathds{1}(n= i) \zeta_1^{  ( j )} \zeta_1^{  ( k )} \mathds{1}(l=m)
 +\zeta_1^{  ( i )} \mathds{1}(n= j) \zeta_1^{  ( k )} \mathds{1}(l=m)
 +\zeta_1^{  ( i )} \zeta_1^{  ( j )} \mathds{1}(n= k)\mathds{1}(l=m) , 
\end{align*}
and 
\begin{align*}
	\frac{\partial^2}{\partial \zeta_1^{(l)} \partial \zeta_2^{(n)}  } 	P_{ijkm}(\boldsymbol{\zeta})    
\end{align*}
is given by
\begin{align*}
& -\frac{ \zeta_2^{  ( j )}  }{\omega_{ i } \omega_{ j }}
    \zeta_1^{  ( m )}\mathds{1}(n= i)\mathds{1}(l=k)
  -\frac{\zeta_2^{  ( i )}  }{\omega_{ i } \omega_{ j }}
    \zeta_1^{  ( m )}\mathds{1}(n= j)\mathds{1}(l=k) 
    -\frac{ \zeta_2^{  ( j )}  }{\omega_{ i } \omega_{ j }} \zeta_1^{  ( k )} \mathds{1}(n= i) \mathds{1}(l=m) \\
& -\frac{\zeta_2^{  ( i )}   }{\omega_{ i } \omega_{ j }} \zeta_1^{  ( k )} \mathds{1}(n= j)\mathds{1}(l=m)  
 +\frac{ \zeta_2^{  ( k )}  }{\omega_{ i } \omega_{ k }}
    \zeta_1^{  ( m )} \mathds{1}(n= i) \mathds{1}(l=j)
  +\frac{\zeta_2^{  ( i )}    }{\omega_{ i } \omega_{ k }}
    \zeta_1^{  ( m )} \mathds{1}(n= k) \mathds{1}(l=j) \\
 & +\frac{ \zeta_2^{  ( k )}  }{\omega_{ i } \omega_{ k }}\zeta_1^{  ( j )} \mathds{1}(n= i)  \mathds{1}(l=m) 
  +\frac{\zeta_2^{  ( i )}    }{\omega_{ i } \omega_{ k }}\zeta_1^{  ( j )}\mathds{1}(n= k) \mathds{1}(l=m) 
 +\frac{   \zeta_2^{  ( k )} }{\omega_{ j } \omega_{ k }} 
  \zeta_1^{  ( m )}\mathds{1}(n= j)  \mathds{1}(l=i)\\
&  +\frac{ \zeta_2^{  ( j )}  }{\omega_{ j } \omega_{ k }} 
  \zeta_1^{  ( m )}\mathds{1}(n= k)  \mathds{1}(l=i)
  +\frac{   \zeta_2^{  ( k )} }{\omega_{ j } \omega_{ k }}   \zeta_1^{  ( i )} \mathds{1}(n= j) \mathds{1}(l=m) 
  +\frac{ \zeta_2^{  ( j )}   }{\omega_{ j } \omega_{ k }}   \zeta_1^{  ( i )} 
  \mathds{1}(n= k) \mathds{1}(l=m)   \\
&  +\frac{   \zeta_2^{  ( m )}}{\omega_{ i } \omega_{ m }} 
  \zeta_1^{  ( k )}\mathds{1}(n= i)  \mathds{1}(l=j)
  +\frac{\zeta_2^{  ( i )}  }{\omega_{ i } \omega_{ m }} 
  \zeta_1^{  ( k )} \mathds{1}(n= m) \mathds{1}(l=j) 
  +\frac{   \zeta_2^{  ( m )}}{\omega_{ i } \omega_{ m }}   \zeta_1^{  ( j )}  \mathds{1}(n= i)\mathds{1}(l=k)  \\
 & +\frac{\zeta_2^{  ( i )}   }{\omega_{ i } \omega_{ m }}   \zeta_1^{  ( j )} \mathds{1}(n= m)\mathds{1}(l=k) 
 +\frac{  \zeta_2^{  ( m )}}{\omega_{ j } \omega_{ m }}
   \zeta_1^{  ( k )} \mathds{1}(n= j) \mathds{1}(l=i) 
   +\frac{ \zeta_2^{  ( j )}}{\omega_{ j } \omega_{ m }}
   \zeta_1^{  ( k )}  \mathds{1}(n= m) \mathds{1}(l=i) \\
&  +\frac{    \zeta_2^{  ( m )}}{\omega_{ j } \omega_{ m }} \zeta_1^{  ( i )}  \mathds{1}(n= j) \mathds{1}(l=k)
  +\frac{ \zeta_2^{  ( j )}   }{\omega_{ j } \omega_{ m }} \zeta_1^{  ( i )}  \mathds{1}(n= m) \mathds{1}(l=k) 
  -\frac{  \zeta_2^{  ( m )}}{\omega_{ k } \omega_{ m }}
  \zeta_1^{  ( j )}  \mathds{1}(n= k) \mathds{1}(l=i)\\
&  -\frac{ \zeta_2^{  ( k )}  }{\omega_{ k } \omega_{ m }}
  \zeta_1^{  ( j )} \mathds{1}(n= m) \mathds{1}(l=i)
 -\frac{   \zeta_2^{  ( m )}}{\omega_{ k } \omega_{ m }}  \zeta_1^{  ( i )} \mathds{1}(n= k) \mathds{1}(l=j) 
 -\frac{ \zeta_2^{  ( k )} }{\omega_{ k } \omega_{ m }}  \zeta_1^{  ( i )} \mathds{1}(n= m) \mathds{1}(l=j)    ,\\ 
\end{align*}
and 
\begin{align*}
	 \frac{\partial^2}{\partial \zeta_2^{(n)} \partial \zeta_2^{(l)}} 	P_{ijkm}(\boldsymbol{\zeta})   
\end{align*}
is given by
\begin{align*}
%
& -\frac{ \zeta_1^{  ( k )} \zeta_1^{  ( m )}}{\omega_{ i } \omega_{ j }}
 \mathds{1}(n=  j)\mathds{1}(l=i)
 -\frac{ \zeta_1^{  ( k )} \zeta_1^{  ( m )}}{\omega_{ i } \omega_{ j }} \mathds{1}(n=  i) \mathds{1}(l=j) 
 +\frac{ \zeta_1^{  ( j )}  \zeta_1^{  ( m )}}{\omega_{ i } \omega_{ k }} 
 \mathds{1}(n=  k) \mathds{1}(l=i) \\
& +\frac{ \zeta_1^{  ( j )}  \zeta_1^{  ( m )}}{\omega_{ i } \omega_{ k }}  \mathds{1}(n=  i)\mathds{1}(l=k)
+\frac{\zeta_1^{  ( i )}  \zeta_1^{  ( m )}}{\omega_{ j } \omega_{ k }} 
  \mathds{1}(n= k ) \mathds{1}(l=j) 
 +\frac{\zeta_1^{  ( i )}  \zeta_1^{  ( m )}}{\omega_{ j } \omega_{ k }}   \mathds{1}(n=  j) \mathds{1}(l=k)  \\
& +\frac{ \zeta_1^{  ( j )} \zeta_1^{  ( k )} }{\omega_{ i } \omega_{ m }}
\mathds{1}(n=  m) \mathds{1}(l=i)
 +\frac{ \zeta_1^{  ( j )} \zeta_1^{  ( k )} }{\omega_{ i } \omega_{ m }} \mathds{1}(n=  i)\mathds{1}(l=m)  
 +\frac{\zeta_1^{  ( i )}  \zeta_1^{  ( k )} }{\omega_{ j } \omega_{ m }}
  \mathds{1}(n=  m) \mathds{1}(l=j) \\
&  +\frac{\zeta_1^{  ( i )}  \zeta_1^{  ( k )} }{\omega_{ j } \omega_{ m }}  \mathds{1}(n=  j)\mathds{1}(l=m) 
-\frac{\zeta_1^{  ( i )} \zeta_1^{  ( j )} }{\omega_{ k } \omega_{ m }}
  \mathds{1}(n=  m) \mathds{1}(l=k)
 -\frac{\zeta_1^{  ( i )} \zeta_1^{  ( j )} }{\omega_{ k } \omega_{ m }}  \mathds{1}(n=  k) \mathds{1}(l=m)   \\
&+\frac{ \zeta_2^{  ( k )} \zeta_2^{  ( m )}}{\omega_{ i } \omega_{ j } \omega_{ k } \omega_{ m }}\mathds{1}(l=i)\mathds{1}(n= j) 
+\frac{  \zeta_2^{  ( j )}  \zeta_2^{  ( m )}}{\omega_{ i } \omega_{ j } \omega_{ k } \omega_{ m }}\mathds{1}(l=i)\mathds{1}(n= k)
+\frac{  \zeta_2^{  ( j )} \zeta_2^{  ( k )} }{\omega_{ i } \omega_{ j } \omega_{ k } \omega_{ m }}\mathds{1}(l=i)\mathds{1}(n= m) \\
& +\frac{  \zeta_2^{  ( k )} \zeta_2^{  ( m )}}{\omega_{ i } \omega_{ j } \omega_{ k } \omega_{ m }}\mathds{1}(n= i) \mathds{1}(l=j)
+\frac{\zeta_2^{  ( i )} \zeta_2^{  ( m )}}{\omega_{ i } \omega_{ j } \omega_{ k } \omega_{ m }} \mathds{1}(l=j) \mathds{1}(n= k)
+\frac{\zeta_2^{  ( i )}  \zeta_2^{  ( k )}}{\omega_{ i } \omega_{ j } \omega_{ k } \omega_{ m }} \mathds{1}(n= m) \mathds{1}(l=j) \\
& +\frac{ \zeta_2^{  ( j )}  \zeta_2^{  ( m )}}{\omega_{ i } \omega_{ j } \omega_{ k } \omega_{ m }} \mathds{1}(n= i) \mathds{1}(l=k)
+\frac{ \zeta_2^{  ( i )} \zeta_2^{  ( m )}}{\omega_{ i } \omega_{ j } \omega_{ k } \omega_{ m }} \mathds{1}(l=k) \mathds{1}(n= j) 
+\frac{ \zeta_2^{  ( i )} \zeta_2^{  ( j )} }{\omega_{ i } \omega_{ j } \omega_{ k } \omega_{ m }} \mathds{1}(l=k) \mathds{1}(n= m) \\
&+\frac{ \zeta_2^{  ( j )} \zeta_2^{  ( k )}}{\omega_{ i } \omega_{ j } \omega_{ k } \omega_{ m }}  \mathds{1}(l=m)\mathds{1}(n= i)
+\frac{\zeta_2^{  ( i )}  \zeta_2^{  ( k )}}{\omega_{ i } \omega_{ j } \omega_{ k } \omega_{ m }}  \mathds{1}(l=m)\mathds{1}(n= j)
+\frac{\zeta_2^{  ( i )} \zeta_2^{  ( j )}}{\omega_{ i } \omega_{ j } \omega_{ k } \omega_{ m }}  \mathds{1}(l=m) \mathds{1}(n= k),
\end{align*}
for all integers $i,j,k,m,n,l \geq 0$. Secondly, we restrict our attention to the rescaled first 1-mode $\boldsymbol{\xi}=(\xi_1,\xi_2)$ defined by \eqref{Definition1modesInitialdataBambusiNekhoroshev}. In this case, the latter yields that
 \begin{align*}
	\frac{\partial^2}{\partial \zeta_1^{(l)}\zeta_1^{(n)} } 	P_{ijkm}(\boldsymbol{\xi})  
\end{align*}
is given by
\begin{align*}
& + \kappa_{0}^2 \mathds{1}(k=m=0)  \mathds{1}(n= j) \mathds{1}(l=i)
 +  \kappa_{0}^2 \mathds{1}(j=m=0)  \mathds{1}(n= k)  \mathds{1}(l=i) \\
&  + \kappa_{0}^2 \mathds{1}(j=k=0)  \mathds{1}(n= m)\mathds{1}(l=i)
+\kappa_{0}^2 \mathds{1}(k=m=0) \mathds{1}(n= i)  \mathds{1}(l=j) \\
& +\kappa_{0}^2 \mathds{1}(i=m=0)  \mathds{1}(n= k)  \mathds{1}(l=j)
 +\kappa_{0}^2 \mathds{1}(i=k=0) \mathds{1}(n= m)\mathds{1}(l=j) \\
& +\kappa_{0}^2 \mathds{1}(j=m=0) \mathds{1}(n= i) \mathds{1}(l=k)
  +\kappa_{0}^2 \mathds{1}(i=m=0)  \mathds{1}(n= j) \mathds{1}(l=k) \\
  & +\kappa_{0}^2 \mathds{1}(i=j=0)  \mathds{1}(n= m)\mathds{1}(l=k) 
   +\kappa_{0}^2 \mathds{1}(j=k=0) \mathds{1}(n= i)   \mathds{1}(l=m)\\
& + \kappa_{0}^2 \mathds{1}(i=k=0)  \mathds{1}(n= j)   \mathds{1}(l=m)
 +\kappa_{0}^2 \mathds{1}(i=j=0)  \mathds{1}(n= k)\mathds{1}(l=m) , 
\end{align*}
and 
\begin{align*}
	\frac{\partial^2}{\partial \zeta_1^{(l)} \partial \zeta_2^{(n)}  } 	P_{ijkm}(\boldsymbol{\xi})    =0
\end{align*}
and 
\begin{align*}
	 \frac{\partial^2}{\partial \zeta_2^{(n)} \partial \zeta_2^{(l)}} 	P_{ijkm}(\boldsymbol{\xi})   
\end{align*}
is given by
\begin{align*}
& -    \frac{  \kappa_{0}^2}{\omega_{ i } \omega_{ j }}
\mathds{1}( k= m= 0) \mathds{1}(n=  j)\mathds{1}(l=i)
 -\frac{ \kappa_{0}^2 }{\omega_{ i } \omega_{ j }} \mathds{1}( k= m= 0)\mathds{1}(n=  i) \mathds{1}(l=j) \\
& +\frac{ \kappa_{0}^2 }{\omega_{ i } \omega_{ k }} 
\mathds{1}( j= m= 0) \mathds{1}(n=  k) \mathds{1}(l=i)
+\frac{ \kappa_{0}^2 }{\omega_{ i } \omega_{ k }} \mathds{1}( j= m= 0) \mathds{1}(n=  i)\mathds{1}(l=k) \\
&+\frac{\kappa_{0}^2  }{\omega_{ j } \omega_{ k }} 
 \mathds{1}( i= m= 0) \mathds{1}(n= k ) \mathds{1}(l=j) 
 +\frac{\kappa_{0}^2 }{\omega_{ j } \omega_{ k }} \mathds{1}(i = m= 0)   \mathds{1}(n=  j) \mathds{1}(l=k)  \\
& +\frac{ \kappa_{0}^2  }{\omega_{ i } \omega_{ m }}
\mathds{1}( j= k= 0)\mathds{1}(n=  m) \mathds{1}(l=i)
 +\frac{ \kappa_{0}^2  }{\omega_{ i } \omega_{ m }}\mathds{1}( j= k= 0) \mathds{1}(n=  i)\mathds{1}(l=m)  \\ 
& +\frac{\kappa_{0}^2 }{\omega_{ j } \omega_{ m }}
  \mathds{1}( i= k= 0)\mathds{1}(n=  m) \mathds{1}(l=j)
    +\frac{ \kappa_{0}^2  }{\omega_{ j } \omega_{ m }}  \mathds{1}(n=  j)\mathds{1}( i= k= 0)\mathds{1}(l=m)  \\
& -\frac{\kappa_{0}^2  }{\omega_{ k } \omega_{ m }}
  \mathds{1}(i =j = 0)\mathds{1}(n=  m) \mathds{1}(l=k)
 -\frac{\kappa_{0}^2 }{\omega_{ k } \omega_{ m }} \mathds{1}(i =j = 0)  \mathds{1}(n=  k) \mathds{1}(l=m)  ,
\end{align*}
for all integers $i,j,k,m,n,l \geq 0$. Now, restrict our attention to the rescaled first 1-mode, combine the three previous results together with Lemma \ref{LemmaTimeAveraging} and the symmetries for the Fourier coefficients from \eqref{DefinitionOfFourierCoefficientsModelKG} and \eqref{DefinitionOfFourierCoefficientsModelWM} to obtain
\begin{align*} 
	\frac{\partial^2 \langle f\rangle (\boldsymbol{\xi})}{\partial \zeta_1^{(l)} \partial \zeta_1^{(n)}}  &= \frac{3}{32}\sum_{i,j,k,m=0}^{\infty} C_{ijkm}   
	\frac{\partial^2 P_{ijkm}(\boldsymbol{\xi})}{\partial \zeta_1^{(l)} \partial \zeta_1^{(n)}}  \mathds{1}(m=i+j-k \geq 0)  \\
	&=  \frac{3}{32} \kappa_{0}^2\Big[
	C_{ln 00 }    \mathds{1}(0=l+n \geq 0)
 +  C_{l0 n 0  }     \mathds{1}(n=l \geq 0)
  +  C_{l 0 0 n }  \mathds{1}(n=l \geq 0) \\
 &+ C_{ n l 00 }   \mathds{1}(0=n+l \geq 0) 
  + C_{0 l n 0 }    \mathds{1}(n=l \geq 0)  
 + C_{0 l 0 n  }   \mathds{1}(n=l \geq 0)   \\ 
&
  + C_{ n 0 l 0  }  \mathds{1}(l=n \geq 0)  
  + C_{ 0 n l 0  }  \mathds{1}(l=n \geq 0) 
   + C_{ 00 l n }    \mathds{1}(n=-l \geq 0)   \\ 
&
   + C_{ n 00 l  }   \mathds{1}(l=n \geq 0) 
    +  C_{0 n 0 l}     \mathds{1}(l=n \geq 0)  
 + C_{00 n l }    \mathds{1}(l=-n \geq 0) 
	\Big]    \\
	&=  \frac{3}{32} \kappa_{0}^2\Big[
	C_{ln 00 }    \mathds{1}(l=n=0 )
 +  C_{l0 n 0  }     \mathds{1}( n=l )  \\
 & +  C_{l 0 0 n }  \mathds{1}(n=l ) 
 + C_{ n l 00 }   \mathds{1}(l=n=0)
  + C_{0 l n 0 }    \mathds{1}(  n=l  )  
 + C_{0 l 0 n  }   \mathds{1}(n= l   )   \\ 
& + C_{ n 0 l 0  }  \mathds{1}(n=l  )  
  + C_{ 0 n l 0  }  \mathds{1}(n=l  )  
    + C_{ 00 l n }    \mathds{1}(n=-l   )  
   + C_{ n 00 l  }   \mathds{1}(l=n  ) \\    
 & +  C_{0 n 0 l}     \mathds{1}(l= n  )  
 + C_{00 n l }    \mathds{1}(l=n=0)	\Big]    \\
	&=  \frac{3}{32} \kappa_{0}^2\Big[
	4 C_{ln 00 }    \mathds{1}(l=n=0 )
 + 8 C_{l0 n 0  }     \mathds{1}( n=l )  	\Big] \\
 &=  \frac{3}{8} \kappa_{0}^2\Big[
	 C_{ln 00 }    \mathds{1}(l=n=0 )
 + 2 C_{l0 n 0  }     \mathds{1}( n=l )  	\Big]  ,  \\ 
\frac{\partial^2 \langle f\rangle (\boldsymbol{\xi})}{\partial \zeta_2^{(n)} \partial \zeta_1^{(l)}}  &=  \frac{3}{32}\sum_{i,j,k,m=0}^{\infty} C_{ijkm}   
	\frac{\partial^2  P_{ijkm}(\boldsymbol{\xi})}{\partial \zeta_2^{(n)} \partial \zeta_1^{(l)}}   \mathds{1}(m=i+j-k \geq 0)  
	 =  0,\\
\frac{\partial^2 \langle f\rangle (\boldsymbol{\xi})}{\partial \zeta_2^{(n)} \partial \zeta_2^{(l)}}   &=   \frac{3}{32}\sum_{i,j,k,m=0}^{\infty} C_{ijkm}   
	\frac{\partial^2 P_{ijkm}(\boldsymbol{\xi}) }{\partial \zeta_2^{(n)} \partial \zeta_2^{(l)}}   \mathds{1}(m=i+j-k \geq 0)  \\
	&=  \frac{3}{32}\kappa_{0}^2 \Big[ 
	 -    \frac{  C_{ ln 0 0 }  }{\omega_{ l } \omega_{ n }}  
	 \mathds{1}(0 =l+n-0  \geq 0) 
 -\frac{C_{n l 0 0   }  }{\omega_{ n } \omega_{ l }} 
 \mathds{1}(-l =n  \geq 0)  
  +\frac{ C_{ l 0  n 0  }  }{\omega_{ l } \omega_{ n }}  
\mathds{1}(n =l  \geq 0)  \\ 
&
+\frac{ C_{ n 0  l 0  } }{\omega_{ n} \omega_{ l }}  
\mathds{1}(l =n  \geq 0) 
+\frac{ C_{0  l n 0   } }{\omega_{ l } \omega_{ n}} 
\mathds{1}(n = l \geq 0)   
 +\frac{ C_{ 0  n l 0  } }{\omega_{ n } \omega_{ l }} 
 \mathds{1}(l = n \geq 0)   \\ 
& +\frac{  C_{l 0  0  n }  }{\omega_{ l } \omega_{ n }}  
\mathds{1}(n=l  \geq 0) 
 +\frac{  C_{ n 0 0 l }  }{\omega_{ n } \omega_{ l }}  
 \mathds{1}(l=n  \geq 0) 
  +\frac{ C_{ 0  l 0  n } }{\omega_{ l } \omega_{ n}}  
\mathds{1}(n=l  \geq 0) 
    +\frac{  C_{ 0 n 0  l }  }{\omega_{ n } \omega_{ l }} 
    \mathds{1}(l=n  \geq 0)    \\
& -\frac{ C_{ 0  0  l n  } }{\omega_{ l } \omega_{ n }} 
 \mathds{1}(n=-l \geq 0) 
 -\frac{ C_{ 0  0  n l } }{\omega_{ n } \omega_{ l }}  
 \mathds{1}(l= -n \geq 0)  \Big] \\
	&=  \frac{3}{32}\kappa_{0}^2 \Big[ 
	 -    \frac{  C_{ ln 0 0 }  }{\omega_{ l } \omega_{ n }}  
	 \mathds{1}( l=n=0 ) 
 -\frac{C_{n l 0 0   }  }{\omega_{ n } \omega_{ l }} 
 \mathds{1}(l=n=0)    
  +\frac{ C_{ l 0  n 0  }  }{\omega_{ l } \omega_{ n }}  
\mathds{1}(l=n) \\
&+\frac{ C_{ n 0  l 0  } }{\omega_{ n} \omega_{ l }}  
\mathds{1}(n=l)  +\frac{ C_{0  l n 0   } }{\omega_{ l } \omega_{ n}} 
\mathds{1}(l=n)   
 +\frac{ C_{ 0  n l 0  } }{\omega_{ n } \omega_{ l }} 
 \mathds{1}(n=l) 
  +\frac{  C_{l 0  0  n }  }{\omega_{ l } \omega_{ n }}  
\mathds{1}(n=l ) \\
& +\frac{  C_{ n 0 0 l }  }{\omega_{ n } \omega_{ l }}  
 \mathds{1}(l=n )  
  +\frac{ C_{ 0  l 0  n } }{\omega_{ l } \omega_{ n}}  
\mathds{1}(n= l ) 
    +\frac{  C_{ 0 n 0  l }  }{\omega_{ n } \omega_{ l }} 
    \mathds{1}(l= n )    \\
& -\frac{ C_{ 0  0  l n  } }{\omega_{ l } \omega_{ n }} 
 \mathds{1}(l=n=0) 
 -\frac{ C_{ 0  0  n l } }{\omega_{ n } \omega_{ l }}  
 \mathds{1}(l=n=0)  \Big] \\
	&=  \frac{3}{32}\kappa_{0}^2 \Big[ 
	 -  4  \frac{  C_{ ln 0 0 }  }{\omega_{ l } \omega_{ n }}  
	 \mathds{1}( l=n=0 ) 
   +8\frac{ C_{ l 0  n 0  }  }{\omega_{ l } \omega_{ n }}  
\mathds{1}(l=n)  \Big] \\
	&=  \frac{3}{8 } \frac{\kappa_{0}^2}{\omega_{ l } \omega_{ n }} \Big[ 
	 -     C_{ ln 0 0 }   
	 \mathds{1}( l=n=0 ) 
   +2 C_{ l 0  n 0  }    \mathds{1}(l=n)  \Big] ,
\end{align*}
for all integers $n,l \geq 0$. Moreover, for any tangent vector 
\begin{align*}
	\boldsymbol{X}=(V,W)=\left(\sum_{m=0}^{\infty}V^{(m)}e_{m},\sum_{m=0}^{\infty}W^{(m)}e_{m} \right)  \in T_{\boldsymbol{\xi}}\Sigma,
\end{align*}
given by Lemma \ref{LemmaTangentSpaceofSigma}, the latter yields
\begin{align*}
	 d^2  &( \langle f\rangle   )(\boldsymbol{\xi})(\boldsymbol{X},\boldsymbol{X}) = \sum_{ \alpha,\beta=0}^{\infty}  \frac{\partial^2  \langle f\rangle (\boldsymbol{\xi})}{ \partial \boldsymbol{\zeta}_{\alpha} \partial \boldsymbol{\zeta}_{\beta} } \boldsymbol{X}^{\alpha} \boldsymbol{X}^{\beta} \\
	&=\sum_{n,l=0}^{\infty} \Bigg[ 
	\frac{\partial^2 \langle f\rangle (\boldsymbol{\xi})}{\partial \zeta_1^{(n)} \partial \zeta_1^{(l)}}  V^{(n)}V^{(l)}+
2\frac{\partial^2 \langle f\rangle (\boldsymbol{\xi})}{\partial \zeta_2^{(n)} \partial \zeta_1^{(l)}}  W^{(n)}V^{(l)}+
\frac{\partial^2 \langle f\rangle (\boldsymbol{\xi})}{\partial \zeta_2^{(n)} \partial \zeta_2^{(l)}}  W^{(n)}W^{(l)}
	\Bigg] \\
	&=\frac{3}{8} \kappa_{0}^2 \sum_{n,l=0}^{\infty} \Bigg[  \big(
	 C_{ln 00 }    \mathds{1}(l=n=0 )
 + 2 C_{l0 n 0  }     \mathds{1}( n=l )  	\big) V^{(n)}V^{(l)} \\
 &+    \big( 
	 -     C_{ ln 0 0 }   
	 \mathds{1}( l=n=0 ) 
   +2 C_{ l 0  n 0  }    \mathds{1}(l=n)  \big)\frac{W^{(n)}W^{(l)}}{\omega_{ n } \omega_{ l }}  \Bigg] \\
   &=\frac{3}{8} \kappa_{0}^2 \Bigg[    
	 C_{0000}   (V^{(0)})^2   
 + 2 \sum_{l=0}^{\infty} C_{l0 l 0  } (V^{(l)})^2 
  	 -       C_{ 0000 }  \left(\frac{W^{(0)} }{\omega_{ 0 } }   \right)^2
   +2   \sum_{l=0}^{\infty} C_{ l 0  l 0  }   \left( \frac{W^{(l)}}{\omega_{ l }  } \right)^2 \Bigg] \\
   &=  \frac{\omega_{0}^2}{C_{00 0 0}   } \Bigg[   
	 C_{0000 }  \left(V^{(0)}\right)^2 
 + 2 \sum_{l=0}^{\infty} C_{l0 l 0  } (V^{(l)})^2 
   -         C_{ 0000} \left( \frac{W^{(0)} }{\omega_{0 } }  \right)^2 
   +2   \sum_{l=0}^{\infty} C_{ l 0  l 0  }   \left( \frac{W^{(l)}}{\omega_{ l }  } \right)^2 \Bigg] \\
   &= \omega_{0}^2 \left(V^{(0)}\right)^2 
 + 2 \omega_{0}^2\sum_{l=0}^{\infty} \frac{C_{l0 l 0  }}{C_{0000}} (V^{(l)})^2 
   -         \left( W^{(0)}  \right)^2 
   +2   \omega_{0}^2 \sum_{l=0}^{\infty}  \frac{C_{l0 l 0  }}{C_{0000}}    \left( \frac{W^{(l)}}{\omega_{ l }  } \right)^2  ,
\end{align*}
where we also used the definition of the rescaled constant $\kappa_{0}$ given by \eqref{Definition1modesInitialdataBambusiNekhoroshev}. Thirdly, we use Lemma \ref{LemmaHarmonicFlowWrtTheBasisem} to compute
\begin{align*} 
	  \frac{\partial^2}{\partial \zeta_1^{(l)} \zeta_1^{(n)}  } h_{\Omega} (\boldsymbol{\zeta}) , \quad  
	  \frac{\partial^2}{\partial \zeta_1^{(l)} \zeta_2^{(n)}   } h_{\Omega}(\boldsymbol{\zeta}) , \quad  
	  \frac{\partial^2}{\partial \zeta_2^{(l)} \zeta_2^{(n)}  } 	h_{\Omega}(\boldsymbol{\zeta}) , 
\end{align*}
for all integers $l,n \geq 0$, and for any $\boldsymbol{\zeta}=(\zeta_1,\zeta_2)$. We obtain that 
\begin{align*}
	\frac{\partial^2}{\partial \zeta_1^{(l)} \zeta_1^{(n)}  } h_{\Omega} (\boldsymbol{\zeta}) = 2\omega_{l}^2 \mathds{1}(l=n)  ,  \quad 
	  \frac{\partial^2}{\partial \zeta_1^{(l)} \zeta_2^{(n)}   } h_{\Omega}(\boldsymbol{\zeta})  = 0,  \quad 
	  \frac{\partial^2}{\partial \zeta_2^{(l)} \zeta_2^{(n)}  } 	h_{\Omega}(\boldsymbol{\zeta}) = 2  \mathds{1}(l=n) ,
\end{align*}
for all integers $l,n \geq 0$. Moreover, restricting our attention to the rescaled first 1-mode $\boldsymbol{\xi}=(\xi_1,\xi_2)$ given by \eqref{Definition1modesInitialdataBambusiNekhoroshev}, for any tangent vector 
\begin{align*}
	\boldsymbol{X}=(V,W)=\left(\sum_{m=0}^{\infty}V^{(m)}e_{m},\sum_{m=0}^{\infty}W^{(m)}e_{m} \right)  \in T_{\boldsymbol{\xi}}\Sigma,
\end{align*}
given by Lemma \ref{LemmaTangentSpaceofSigma}, the latter yields
\begin{align*}
	 d^2  h_{\Omega}(\boldsymbol{\xi})(\boldsymbol{X},\boldsymbol{X}) &=
	\sum_{ \alpha,\beta=0}^{\infty}  \frac{\partial^2  h_{\Omega}(\boldsymbol{\xi})}{ \partial \boldsymbol{\zeta}_{\alpha} \partial \boldsymbol{\zeta}_{\beta} } \boldsymbol{X}^{\alpha} \boldsymbol{X}^{\beta} \\
	&=\sum_{n,l=0}^{\infty} \Bigg[ 
	\frac{\partial^2 h_{\Omega}(\boldsymbol{\xi})}{\partial \zeta_1^{(n)} \partial \zeta_1^{(l)}}  V^{(n)}V^{(l)}+
2\frac{\partial^2 h_{\Omega}(\boldsymbol{\xi})}{\partial \zeta_2^{(n)} \partial \zeta_1^{(l)}}  W^{(n)}V^{(l)}+
\frac{\partial^2 h_{\Omega}(\boldsymbol{\xi})}{\partial \zeta_2^{(n)} \partial \zeta_2^{(l)}}  W^{(n)}W^{(l)}
	\Bigg] \\
	&=\sum_{n,l=0}^{\infty} \Bigg[ 
	2\omega_{l}^2 \mathds{1}(l=n)  V^{(n)}V^{(l)}
	+   2  \mathds{1}(l=n) W^{(n)}W^{(l)}
	\Bigg] \\
	&=2\sum_{l=0}^{\infty}  
	\omega_{l}^2  (  V^{(l)}  )^2
	+   2 \sum_{l=0}^{\infty}   (  W^{(l)} )^2   .
\end{align*}
Finally, according to Remark page 78-79 in \cite{BAMBUSI199873}, we have that
\begin{align*}
	 d^2  ( \langle f\rangle |_{\Sigma} )(\boldsymbol{\xi})(\boldsymbol{X},\boldsymbol{X})  = \lambda  d^2  h_{\Omega}(\boldsymbol{\xi})(\boldsymbol{X},\boldsymbol{X}) + d^2  ( \langle f\rangle   )(\boldsymbol{\xi})(\boldsymbol{X},\boldsymbol{X}) ,
\end{align*} 
with $\lambda=-1/2$ due to \eqref{ValueLagrangeMultiplierLambda}, and hence we conclude that
\begin{align*}
	 d^2 &  ( \langle f\rangle |_{\Sigma} )(\boldsymbol{\xi})(\boldsymbol{X},\boldsymbol{X})  = -\frac{1}{2}  d^2  h_{\Omega}(\boldsymbol{\xi})(\boldsymbol{X},\boldsymbol{X}) + d^2  ( \langle f\rangle   )(\boldsymbol{\xi})(\boldsymbol{X},\boldsymbol{X}) \\
	& = -\sum_{l=0}^{\infty}  
	\omega_{l}^2  (  V^{(l)}  )^2
	-   \sum_{l=0}^{\infty}   (  W^{(l)} )^2  +     
	 \omega_{0}^2      (V^{(0)} )^2
 +  2\omega_{0}^2\sum_{l=0}^{\infty} \frac{C_{l0 l 0  } }{C_{00 0 0}   } (V^{(l)})^2 \\
 & 	 -      \left(  W^{(0)}     \right)^2
   +  2\omega_{0}^2 \sum_{l=0}^{\infty} \frac{ C_{ l 0  l 0  } }{C_{00 0 0}   }   \left( \frac{W^{(l)}}{\omega_{ l }  } \right)^2  \\ 
	&  =  - \frac{\omega_{0}^2}{C_{0 0 0 0}} \Bigg[  \sum_{l=0}^{\infty} \Bigg(  \frac{C_{0 0 0 0}}{\omega_{0}^2} - 2 \frac{ C_{l l0   0       } }{ \omega_{l}^2} \Bigg) ( \omega_{l}V^{(l)})^2 
   +\sum_{l=0}^{\infty} \Bigg( \frac{C_{0 0 0 0}}{\omega_{0}^2} -  2\frac{  C_{ l l 0   0  } }{\omega_{l}^2     }    \Bigg)(W^{(l)})^2   \Bigg]\\
   &  +  \omega_0^2 (V^{(0)})^2
	-     \left(   W^{(0)}  \right)^2    ,
\end{align*}
that completes the proof.
\end{proof}

Finally, we verify Condition \eqref{Assumption2BambusiNekhoroshev} in Theorem \ref{TheoremBambusiNekhoroshev} which states that the rescaled first 1-mode must also be an extremum point for the function $\langle f\rangle |_{\Sigma}$, namely there exists a positive constant $c_{\perp}$ such that
	\begin{align*}
		\pm  d^2  ( \langle f\rangle |_{\Sigma} )(\boldsymbol{\xi})(\boldsymbol{X},\boldsymbol{X}) \geq c_{\perp} \| \boldsymbol{X}\|^2, \quad \forall \boldsymbol{X} \in \left(T_{\boldsymbol{\xi}}\Gamma_{\text{linear}}(\boldsymbol{\xi})  \right)^{\perp},
	\end{align*}
	with the plus and minus sign if $\boldsymbol{\xi} \in \Sigma$ is a minimum or maximum for the function $\langle f\rangle |_{\Sigma}$ respectively. Recall that the orthogonal complement $\left(T_{\boldsymbol{\xi}}\Gamma_{\text{linear}} (\boldsymbol{\xi}) \right)^{\perp}$ is given by Lemma \ref{LemmaTangentSpaceofSigma}. As we will see next, in our case, the rescaled first 1-mode maximizes $\langle f\rangle |_{\Sigma}$ for both KG and WM if and only if $\delta$ satisfies Assumption \ref{AssumptionsondeltanonlinearStability}.

\begin{lemma}[Condition \eqref{Assumption2BambusiNekhoroshev} in Theorem \ref{TheoremBambusiNekhoroshev}]\label{LemmaAssumption2BambusiNekhoroshev}
	Let  $\boldsymbol{\xi}=(\xi_1,\xi_2)$ be the rescaled first 1-mode  given by \eqref{Definition1modesInitialdataBambusiNekhoroshev} and let 
\begin{align*}
	\Sigma= \left \{\boldsymbol{\zeta}\in \mathcal{P}:~  h_{\Omega}(\boldsymbol{\zeta})=\kappa_{0}^2 \omega_{0}^2 \right\}
\end{align*}
be the hypersurface of constant  harmonic energy. Then, $\delta$ satisfies Assumption \ref{AssumptionsondeltanonlinearStability} if and only if there exists a positive constant $c_{\perp}=c_{\perp}(\delta)$, depending only on $\delta$, such that
\begin{align*}
		  -d^2  ( \langle f\rangle |_{\Sigma} )(\boldsymbol{\xi})(\boldsymbol{X},\boldsymbol{X}) \geq c_{\perp} \| \boldsymbol{X}\|^2,  \quad \forall \boldsymbol{X} \in \left(T_{\boldsymbol{\xi}}\Gamma_{\text{linear}}  (\boldsymbol{\xi}) \right)^{\perp}. 
	\end{align*} 
\end{lemma}
\begin{proof}
Let  $\boldsymbol{\xi}=(\xi_1,\xi_2)$ be the rescaled first 1-mode  given by \eqref{Definition1modesInitialdataBambusiNekhoroshev}, pick any $\boldsymbol{X}=(V,W)\in  \left(T_{\boldsymbol{\xi}}\Gamma_{\text{linear}}  (\boldsymbol{\xi})\right)^{\perp} $ and expand its coefficients terms of the eigenfunctions $\{e_m:m\geq 0 \}$ to the linearized operator $L$,
	\begin{align*}
		V=\sum_{m=0}^{\infty} V^{(m)} e_{m}, \quad W=\sum_{m=0}^{\infty} W^{(m)} e_{m}.
	\end{align*}
	  Then, according to Lemma \ref{LemmaTangentSpaceofSigma}, we have that
\begin{align*}
	V^{(0)}=W^{(0)}=0 .
\end{align*} 
For convenience, we rescale $V$ and set
\begin{align*}
	Z^n = \omega_{n}V^n,
\end{align*}
for all integers $n \geq 0$. On the one hand, according to Lemma \ref{LemmaSecondDifferentialatCriticalPointComputation}, we have that 
\begin{align*}
	- \frac{C_{0 0 0 0}}{\omega_{0}^2}  d^2  ( \langle f\rangle |_{\Sigma} )(\boldsymbol{\xi})(\boldsymbol{X},\boldsymbol{X}) & =     \sum_{l=1}^{\infty} \Bigg(   \frac{C_{0 0 0 0}}{\omega_{0}^2} -2 \frac{ C_{l l0   0       }  }{ \omega_{l}^2} \Bigg) (Z^{(l)})^2 
   +\sum_{l=1}^{\infty} \Bigg(   \frac{C_{0 0 0 0}}{\omega_{0}^2} - 2 \frac{  C_{ l l 0   0  } }{\omega_{l}^2     }  \Bigg)(W^{(l)})^2 .
\end{align*} 
On the other hand, we have that
\begin{align*}
	\| \boldsymbol{X}\|_{\mathcal{P}}^2 &=
	\sum_{\substack{ l=1   }}^{\infty}\left(  (Z^{(l)})^2 +  (W^{(l)})^2 \right).
\end{align*}
Let $\{U^{(m)} : m \geq 1\}$ be either $\{Z^{(m)} : m \geq 1\}$ or $\{W^{(m)} : m \geq 1\}$. Then,  we deduce that 
\begin{align*}
	\exists c_{\perp}>0:\Hquad  -d^2  ( \langle f\rangle |_{\Sigma} )(\boldsymbol{\xi})(\boldsymbol{X},\boldsymbol{X}) &\geq c_{\perp} \| \boldsymbol{X}\|^2, \quad \forall	\boldsymbol{X} \in \left(T_{\boldsymbol{\xi}}\Gamma_{\text{linear}} (\boldsymbol{\xi}) \right)^{\perp}  \Longleftrightarrow \\
	\exists \tilde{c}_{\perp}>0:\Hquad  \sum_{\substack{ l=1   }}^{\infty} \Bigg(   \frac{C_{0 0 0 0}}{\omega_{0}^2} - 2 \frac{ C_{l l0   0       } }{ \omega_{l}^2}\Bigg) ( U^{(l)})^2 
		  & \geq \tilde{c}_{\perp} \sum_{\substack{ l=1   }}^{\infty}  ( U^{(l)})^2, \quad \forall	 \{U^{(m)}: m \geq 1 \} \subset \mathbb{R} \Longleftrightarrow \\
		  \exists \tilde{c}_{\perp}>0:\Hquad     \frac{C_{0 0 0 0}}{\omega_{0}^2} - 2 \frac{ C_{l l0   0       } }{ \omega_{l}^2} 
		  & \geq \tilde{c}_{\perp} ,
	\end{align*}
which, according to Proposition \ref{PropositionUniformEstimates}, is equivalent to the Assumption \ref{AssumptionsondeltanonlinearStability} for $\delta$, that completes the proof.
\end{proof}

\appendix

\section{Zeilberger's Algorithm} \label{AppendixZeilbergerAlgorithm}
In this section, we discuss Zeilberger's algorithm. In Section \ref{SectionAppendixalgorithm}, we follow Sections 5 and 6 in \cite{MR1379802} to describe how it works in general. In Section \ref{SectionAppendixalgorithmfortheLemma}, we fix $\delta$ and run it through by hand producing step-by-step the results of Lemma \ref{LemmaRecurrenceRelationKG}. 
 
\subsection{How the algorithm works}\label{SectionAppendixalgorithm}
Let $F=F(n,k)$ be a given hypergeometric function with respect to both arguments $n$ and $k$. This means that both $F(n+1,k)/F(n,k)$ and $F(n,k+1)/F(n,k)$ are rational functions of $n$ and $k$ and hence there exist polynomials $r_1,r_2,s_1,s_2$ so that
\begin{align}\label{Fhypergeometricfunction}
	\frac{F(n+1,k)}{F(n,k)}=\frac{s_1(n+1,k)}{s_2(n+1,k)} ,\quad \frac{F(n,k+1)}{F(n,k)}=\frac{r_1(n,k)}{r_2(n,k)}.
\end{align}
We are interested in finding a recurrence relation for the sum
\begin{align}\label{DefinitionSumf}
	f(n)=\sum_{k=0}^{K} F(n,k).
\end{align}
For simplicity, we assume that $K$ is a finite integer. However, as it will become clear later, this assumption can be relaxed. We claim that it suffices to find a recurrence relation for the summand $F(n,k)$. In other words, given $F$, we would like to find a function $G=G(n,k)$ in the form
\begin{align}\label{AnsatzforG}
	G(n,k)=R(n,k)F(n,k),
\end{align}
for some function $R=R(n,k)$, as well as coefficients $\{\alpha_{j}(n):j =0,1,\dots,J\}$, for some (possibly large) integer $J$, all in closed forms, so that  
\begin{align}\label{RecurrenceRelationF}
	\sum_{j=0}^{J} \alpha_{j}(n)F(n+j,k)=G(n,k+1)-G(n,k).
\end{align}
Then, if \eqref{RecurrenceRelationF} is possible, a recurrence relation for $f(n)$, given by \eqref{DefinitionSumf}, follows immediately by summing both sides of \eqref{RecurrenceRelationF} with respect to $k$. Indeed, since the coefficients $\{\alpha_{j}(n):j =0,1,\dots,J\}$ are all independent of $k$, summing \eqref{RecurrenceRelationF} with respect to $k$ yields
\begin{align}\label{RecurrenceRelationSmallf}
\sum_{j=0}^{J} \alpha_{j}(n)F(n+j,k) & =G(n,k+1)-G(n,k) \Longrightarrow \nonumber \\
 \sum_{k=0}^{K} \left(	\sum_{j=0}^{J} \alpha_{j}(n)F(n+j,k) \right) &=\sum_{k=0}^{K} \left(G(n,k+1)-G(n,k) \right) \Longrightarrow \nonumber \\
  \sum_{j=0}^{J}\alpha_{j}(n)\left(	\sum_{k=0}^{K} F(n+j,k) \right) &= G(n,K+1)-G(n,0)  \Longrightarrow \nonumber \\
   \sum_{j=0}^{J}\alpha_{j}(n)f(n+j) &=G(n,K+1)-G(n,0),
\end{align}
which is the desired recurrence relation for $f(n)$. Consequently, we turn our attention to how one can find the function $G=G(n,k)$ and coefficients $\{\alpha_{j}(n):j =0,1,\dots,J\}$ so that \eqref{RecurrenceRelationF} holds\footnote{We note that, for all proper hypergeometric functions $F$, a recurrence relation of the form \eqref{RecurrenceRelationF} always exists (Theorem 6.2.1, page 105 in \cite{MR1379802}).}. To this end, we fix $J$ and denote the left-hand-side of \eqref{RecurrenceRelationF} by $t_k$,
\begin{align}\label{DefinitiontkStart}
	t_k=\sum_{j=0}^{J} \alpha_{j}(n)F(n+j,k).
\end{align}
Using \eqref{Fhypergeometricfunction}, we infer
\begin{align*} 
	\frac{t_{k+1}}{t_{k}} & = \frac{\sum_{j=0}^{J} \alpha_{j}(n)F(n+j,k+1)}{\sum_{j=0}^{J} \alpha_{j}(n)F(n+j,k)}
	 = \frac{\sum_{j=0}^{J} \alpha_{j}(n)\frac{F(n+j,k+1)}{F(n,k+1)}}{\sum_{j=0}^{J} \alpha_{j}(n)
	 \frac{F(n+j,k)}{F(n,k)} 
	 } 
	 \frac{F(n,k+1)}{F(n,k)}  \\
	 & = \frac{\sum_{j=0}^{J} \alpha_{j}(n)\prod_{i=0}^{j-1} \frac{F(n+j-i,k+1)}{F(n+j-i-1,k+1)}}{\sum_{j=0}^{J} \alpha_{j}(n)\prod_{i=0}^{j-1} \frac{F(n+j-i,k)}{F(n+j-i-1,k)}
	 } \frac{F(n,k+1)}{F(n,k)} 
	  = \frac{\sum_{j=0}^{J} \alpha_{j}(n)\prod_{i=0}^{j-1} \frac{s_1(n+j-i,k+1)}{s_2(n+j-i,k+1)}}{\sum_{j=0}^{J} \alpha_{j}(n)\prod_{i=0}^{j-1} \frac{s_1(n+j-i,k)}{s_2(n+j-i,k)}
	 } \frac{r_1(n,k)}{r_2(n,k)}  \\
	 & = \frac{\sum_{j=0}^{J} \alpha_{j}(n) \left[
	 \prod_{i=0}^{j-1}  s_1(n+j-i,k+1) 
	  \prod_{r=j+1}^{J}  s_2(n+r,k+1)  
	\right]}{\sum_{j=0}^{J} \alpha_{j}(n) \left[
	 \prod_{i=0}^{j-1}  s_1(n+j-i,k) 
	  \prod_{r=j+1}^{J}  s_2(n+r,k)  
	\right]
	 } \frac{r_1(n,k)}{r_2(n,k)} \frac{\prod_{r=1}^{J}s_2(n+r,k) }{\prod_{r=1}^{J}s_2(n+r,k+1)} .
\end{align*}
Hence, we have that
\begin{align*}
	\frac{t_{k+1}}{t_{k}} & = \frac{p_0(k+1)}{p_0(k)} \frac{r(k)}{s(k)},
\end{align*}
where
\begin{align}
	p_0( k) &= \sum_{j=0}^{J} \left[ \prod_{i=0}^{j-1}  s_1(n+j-i,k) \prod_{r=j+1}^{J}  s_2(n+r,k) \right],\label{Defintionp0}  \\
	r( k) &=r_1(n,k)  \prod_{r=1}^{J}  s_2(n+r,k)  ,\label{Defintionr} \\
	s( k) &=r_2(n,k)  \prod_{r=1}^{J}  s_2(n+r,k+1)  .\label{Defintions}
\end{align}
Next, we use Theorem 5.3.1, page 82 in \cite{MR1379802}
 to find polynomials $p_1,p_2,p_3$ with $\text{gcd}(p_2(k),p_3(k+j))=1$, for all $j \in \{0,1,2,\dots\}$ so that
\begin{align} \label{Dividerands}
	\frac{r(k)}{s(k)} = \frac{p_1(k+1)}{p_1(k)} \frac{p_2(k)}{p_3(k)}.
\end{align}
Then, we set $p(k)=p_0(k)p_1(k)$ and have that
\begin{align}\label{FinalEquationtk}
	\frac{t_{k+1}}{t_{k}} & = \frac{p_0(k+1)}{p_0(k)} \frac{r(k)}{s(k)} 
	= \frac{p_0(k+1)}{p_0(k)} \frac{p_1(k+1)}{p_1(k)} \frac{p_2(k)}{p_3(k)} 
	= \frac{p(k+1)}{p(k)} \frac{p_2(k)}{p_3(k)}.
\end{align}
Now, one can show that \eqref{FinalEquationtk} has a indefinitely summable hypergeometric solution $t_k$ if and only if the recurrence relation
\begin{align}\label{Recurrencebk}
	p_2(k)b(k+1)-p_3(k-1)b(k)=p(k),
\end{align}
has a polynomial solution $b(k)$. For a detailed proof of this result, see Section 5.2 and Theorem 5.2.1 on page 76 in \cite{MR1379802}
. We look for a polynomial solution to \eqref{Recurrencebk} of degree $\Delta$ and set
\begin{align*}
	b(k)=\sum_{l=0}^{\Delta} \beta_{l}k^l,
\end{align*}
where the coefficients $\{\beta_l:l =0,1,\dots,\Delta\}$ are to be determined. By substituting this expression for $b(k)$ into \eqref{Recurrencebk}, one finds a system of simultaneous linear equations for $\Delta+J+2$ unknowns $\alpha_{0},\dots,\alpha_{J}$ and $\beta_{0},\dots,\beta_{\Delta}$. Finally, we solve this linear system, find the coefficients $\alpha_{0},\dots,\alpha_{J}$ and $\beta_{0},\dots,\beta_{\Delta}$ and set
\begin{align}\label{DefinitionofG}
	G(n,k)=\frac{p_3(k-1)}{p(k)} b(k) t_k.
\end{align}
 For more details, see Sections 5 and 6 in \cite{MR1379802}.

\subsection{The case of Lemma \ref{LemmaRecurrenceRelationKG}}\label{SectionAppendixalgorithmfortheLemma}

Next, we produce step-by-step the results of Lemma \ref{LemmaRecurrenceRelationKG} by running Zeilberger's algorithm through by hand. Due to the complexity of the formulas appearing, we fix the conformal mass and set $\delta=3$ and we are looking for a recurrence relation with $J=2$ steps for the sum
\begin{align*} 
	\mathsf{f}_{m}=	 \frac{ \mathsf{C}_{00 mm}}{\somega_{m}^2} =\sum_{k=0}^{2m+1} \mathsf{F}(m,k)
\end{align*}
that is in the form of \eqref{DefinitionSumf}. Here, $\mathsf{F}(m,k)$ is given by a closed formula as stated in Lemma \ref{LemmaComputationallyEfficientFormulaforC00mmKG}. Firstly, one can use the fact that $\Gamma(x+1)=x\Gamma(x)$, valid for all $x \in \mathbb{R}$, to compute
 \begin{align*}
 	\frac{\mathsf{F}(m,k)}{\mathsf{F}(m-1,k)} & = \frac{(2 m+1) (2 k-4 m-1) (2 k-4 m+1) (k+2 m+3) (k+2 m+4)}{(2 m+3) (k-2 m-1) (k-2 m) (2 k+4 m+7) (2 k+4 m+9)}, \\
 	\frac{\mathsf{F}(m,k+1)}{\mathsf{F(m,k)}} & =\frac{(k+2) (2 k-1) (2 k+1) (2 k+7) (4 k+11) (k-2 m-1) (k+2 m+5)}{(k+1) (k+4) (k+5) (2 k+5) (4 k+7) (2 k-4 m-1) (2 k+4 m+11)}.
 \end{align*}
Since are both rational functions we conclude that $\mathsf{F}(m,k)$ is a hypergeometric function both respect to $k$ and $m$ and hence Zeilberger's algorithm applies. According to the computations above and \eqref{Fhypergeometricfunction}, we define
\begin{align*}
	s_1(m,k) &= (2 m+1) (2 k-4 m-1) (2 k-4 m+1) (k+2 m+3) (k+2 m+4),\\
s_2 (m,k)& =(2 m+3) (k-2 m-1) (k-2 m) (2 k+4 m+7) (2 k+4 m+9), \\
r_1 (m,k)& = (k+2) (2 k-1) (2 k+1) (2 k+7) (4 k+11) (k-2 m-1) (k+2 m+5)  ,\\
r_2 (m,k)& = (k+1) (k+4) (k+5) (2 k+5) (4 k+7) (2 k-4 m-1) (2 k+4 m+11) .
\end{align*}
According to \eqref{DefinitiontkStart}, we set
\begin{align*}
	t_{k}=   \alpha_{0}(m)F(m,k)+\alpha_{1}(m)F(m+1,k)+\alpha_{2}(m)F(m+2,k).
\end{align*}
From now, we simplify the notation and write $\alpha_j$ instead of $\alpha_j(m)$, for all $j \in \{0,1,2\}$. Using the definitions \eqref{Defintionp0}, \eqref{Defintionr} and \eqref{Defintions}, we compute
\begin{align*}
	p_0(k)  &= \alpha_0 (2 m+5) (2 m+7) (-k+2 m+2) (-k+2 m+3) (-k+2 m+4) (-k+2 m+5)\cdot \\
	& \cdot (2 k+4 m+11) (2 k+4 m+13) (2 k+4 m+15) (2 k+4 m+17)
	 \\
	& 
	+(2 m+3) (k+2 m+5) (k+2 m+6) (2 k-4 m-5) (2 k-4 m-3) (\alpha_1 (2 m+7)\cdot \\
	& \cdot  (-k+2 m+4) (-k+2 m+5) (2 k+4 m+15) (2 k+4 m+17)   \\
	& +\alpha_2 (2 m+5) (k+2 m+7) (k+2 m+8) (-2 k+4 m+7) (-2 k+4 m+9)), \\
	r(k) &= (k+2) (2 k-1) (2 k+1) (2 k+7) (4 k+11) (2 m+5) (2 m+7) (k-2 m-5) (k-2 m-3)\cdot \\
	& \cdot  (k-2 m-1) (k+2 m+5) (2 k+4 m+11) (2 k+4 m+13) (2 k+4 m+15) (2 k+4 m+17) \cdot \\
	& \cdot (k-2 (m+1)) (k-2 (m+2)), \\
	s(k) &=(k+1) (k+4) (k+5) (2 k+5) (4 k+7) (2 m+5) (2 m+7) (2 k-4 m-1) (k-2 m-3)\cdot \\
	& \cdot  (k-2 m-1) (2 k+4 m+11) (2 k+4 m+13) (2 k+4 m+15) (2 k+4 m+17)\cdot \\
	& \cdot  (2 k+4 m+19) (k-2 (m+1)) (k-2 (m+2)).
\end{align*}
Notice that the only trace of the unknowns $\alpha_0$, $\alpha_1$ and $\alpha_2$ is in the polynomial $p_0(k)$ and not in $r(k)$ nor in $s(k)$. In addition, observe that $r(k)$ and $s(k)$ have many common factors. By dividing these two, the common terms are canceled and one obtains the canonical form 
\begin{align*}
	\frac{r(k)}{s(k)} &=\frac{(k+2) (2 k-1) (2 k+1) (2 k+7) (4 k+11) (k-2 m-5) (k+2 m+5)}{(k+1) (k+4) (k+5) (2 k+5) (4 k+7) (2 k-4 m-1) (2 k+4 m+19)}.
\end{align*}
This is of the form \eqref{Dividerands} with
\begin{align}
	p_1  (k) &= (1 + k) (5 + 2 k) (7 + 4 k)  ,\nonumber  \\
	p_2  (k) &=  (-1 + 2 k) (1 + 2 k) (-5 + k - 2 m) (5 + k + 2 m),\nonumber \\
	p_3  (k) &=  (k+2) (k+3) (2 k-4 m-1) (2 k+4 m+15) . \label{Definitionp3}
\end{align}
We set 
\begin{align*}
	p(k)=p_0(k)p_1(k),
\end{align*}
so that \eqref{FinalEquationtk} holds, and compute
\begin{align}
	p(k) &=(k+1) (2 k+5) (4 k+7) \big[\alpha_0 (2 m+5) (2 m+7) (-k+2 m+2) (-k+2 m+3)\cdot \nonumber \\
	& \cdot (-k+2 m+4) (-k+2 m+5) (2 k+4 m+11) (2 k+4 m+13) (2 k+4 m+15)\cdot \nonumber \\
	& \cdot  (2 k+4 m+17)+(2 m+3) (k+2 m+5) (k+2 m+6) (2 k-4 m-5) (2 k-4 m-3) \cdot \nonumber \\
	& \cdot \big[\alpha_1 (2 m+7) (-k+2 m+4) (-k+2 m+5) (2 k+4 m+15) (2 k+4 m+17)
	  \nonumber \\
	&  +
	\alpha_2 (2 m+5) (k+2 m+7) (k+2 m+8) (-2 k+4 m+7) (-2 k+4 m+9) \big]  \Big]. \label{Definitionp}
\end{align}
Notice that the only trace of the unknowns $\alpha_0$, $\alpha_1$ and $\alpha_2$ is in the polynomial $p(k)$ and not in $p_1(k)$, $p_2(k)$ nor in $p_3(k)$. Now, we focus on whether the recurrence relation \eqref{Recurrencebk} has a polynomial solution $b(k)$. On the one hand, notice that both $p_2(k)$ and $p_3(k)$ are polynomials of degree $4$ with respect to $k$. On the other hand, $p(k)$ is a polynomial of degree $11$ with respect to $k$. Hence, $b(k)$ should be a polynomial of degree at least $7$ with respect to $k$. However, we relax the algorithm\footnote{This is due to some cancellations in \eqref{Recurrencebk} for these particular choices of $p_2$, $p_3$ and $p$.} and search for a solution $b(k)$ that is a polynomial of degree $\Delta=8$ with respect to $k$ and define
\begin{align}\label{Definitionb}
	b(k)=\beta_8 k^8+\beta_7 k^7+ \beta_6 k^6+\beta_5 k^5+\beta_4 k^4+\beta_3 k^3+\beta_2 k^2+\beta_1 k + \beta_0,
\end{align}  
where the coefficients $\{\beta_j: j \in \{0,1,\dots,8\}\}$ are all to be determined (together with the $\{\alpha_0,\alpha_1,\alpha_2\}$). Next, we plug \eqref{Definitionb} into the recurrence relation \eqref{Recurrencebk} and match the coefficients of same powers of $k$ on both sides. After a long computation, the result is a linear system consisting of $12$ equations for the $12$ unknowns  $\{\alpha_0,\alpha_1,\alpha_2,\beta_0,\beta_1,\dots,\beta_8\}$. This is given by 
\begin{flushleft}
	\underline{\text{Equation 1}:}
\end{flushleft}
 \begin{align*}
 	& \left[-140 (1 + m) (2 + m) (3 + 2 m) (5 + 2 m)^2 (7 + 2 m) (11 + 4 m) (13 + 
   4 m) (15 + 4 m) (17 + 4 m)  \right] \alpha_0 + \\
 &	\left[ -140 (2 + m) (3 + m) (3 + 2 m) (5 + 2 m)^2 (7 + 2 m) (3 + 4 m) (5 + 
   4 m) (15 + 4 m) (17 + 4 m) \right] \alpha_1 + \\
 &	\left [ -140 (3 + m) (4 + m) (3 + 2 m) (5 + 2 m)^2 (7 + 2 m) (3 + 4 m) (5 + 
   4 m) (7 + 4 m) (9 + 4 m) \right] \alpha_2 + \\
 &	\left[49 (13 + 4 m (5 + m))  \right] \beta_0 + 	
 \left[ (5 + 2 m)^2 \right] \beta_1 +
 	\left[(5 + 2 m)^2 \right] \beta_2 +
 	\left[(5 + 2 m)^2 \right] \beta_3 +\\
 &
 	\left[(5 + 2 m)^2  \right] \beta_4 +
 	\left[(5 + 2 m)^2  \right] \beta_5 +
 	\left[ (5 + 2 m)^2 \right] \beta_6 +
 	\left[ (5 + 2 m)^2 \right] \beta_7 +
 	\left[ (5 + 2 m)^2 \right] \beta_8 =0,
 \end{align*}
 
\begin{flushleft}
	\underline{\text{Equation 2}:}
\end{flushleft}
 \begin{align*}
 	& \big[ -2 (5 + 2 m) (7 + 2 m) (97615125 + 
   2 m (21 + 4 m) (10060700 + 
      m (21 + 4 m) (744557 \\
      & + 
         4 m (21 + 4 m) (5927 + 69 m (21 + 4 m))))) \big] \alpha_0 + \\
 &	\big[ -2 (3 + 2 m) (5 + 2 m)^2 (7 + 2 m) (1722825 + 
   2 m (5 + m) (792789 + 32 m (5 + m) (6679 \\
   & + 552 m (5 + m)))) \big] \alpha_1 + \\
 &	\big [ -2 (3 + 2 m) (5 + 2 m) (28653975 + 
   2 m (19 + 4 m) (4255230 + 
      m (19 + 4 m) (430337 + \\
      &
         4 m (19 + 4 m) (4547 + 69 m (19 + 4 m))))) \big] \alpha_2 + \\
 &	\left[ 7 (3 + 16 m (5 + m)) \right] \beta_0 + 	
 \left[49 (13 + 4 m (5 + m))  \right] \beta_1 +
 	\left[2 (5 + 2 m)^2 \right] \beta_2 +
 	\left[ 3 (5 + 2 m)^2 \right] \beta_3 +\\
 &
 	\left[4 (5 + 2 m)^2  \right] \beta_4 +
 	\left[ 5 (5 + 2 m)^2\right] \beta_5 +
 	\left[ 6 (5 + 2 m)^2 \right] \beta_6 +
 	\left[ 7 (5 + 2 m)^2  \right] \beta_7 +
 	\left[ 8 (5 + 2 m)^2 \right] \beta_8 =0,
 \end{align*}

\begin{flushleft}
	\underline{\text{Equation 3}:}
\end{flushleft}
 \begin{align*}
 	& \big[ -7 (5 + 2 m) (7 + 2 m) (-4523755 + 
   8 m (21 + 4 m) (75805 + 
      m (21 + 4 m) (16433 + 4 m (21 + 4 m)\cdot \\
      & \cdot (200 + 3 m (21 + 4 m))))) \big] \alpha_0 + \\
 &	\big[ -7 (3 + 2 m) (7 + 2 m) (-6556875 + 
   8 m (5 + m) \cdot \\
      & \cdot (-96085 + 
      4 m (5 + m) (36407 + 32 m (5 + m) (325 + 24 m (5 + m))))) \big] \alpha_1 + \\
 &	\big [ -7 (3 + 2 m) (5 + 2 m) (-5009355 + 
   8 m (19 + 4 m) (-34525 + 
      m (19 + 4 m) (6233 + 4 m (19 + 4 m)\cdot \\
      & \cdot  (140 + 3 m (19 + 4 m))))) \big] \alpha_2 + \\
 &	\left[-294  \right] \beta_0 + 	
 \left[ -80 + 96 m (5 + m) \right] \beta_1 +
 	\left[536 + 180 m (5 + m) \right] \beta_2 +
 	\left[-26 - 4 m (5 + m) \right] \beta_3 +\\
 &
 	\left[49 + 8 m (5 + m)  \right] \beta_4 +
 	\left[ 149 + 24 m (5 + m) \right] \beta_5 +
 	\left[274 + 44 m (5 + m)  \right] \beta_6  \\
      &  +
 	\left[ 424 + 68 m (5 + m) \right] \beta_7 +
 	\left[599 + 96 m (5 + m)  \right] \beta_8 =0,
 \end{align*}

\begin{flushleft}
	\underline{\text{Equation 4}:}
\end{flushleft}
 \begin{align*}
 	& \big[-(5 + 2 m) (7 + 2 m) (-91051067 + 
   4 m (21 + 4 m) (-3125231 + 
      2 m (21 + 4 m) (-60393 \\
      & + 16 m (21 + 4 m) (-15 + m (21 + 4 m)))))  \big] \alpha_0 + \\
 &	\big[ -(3 + 2 m) (7 + 2 m) (-58593875 + 
   8 m (5 + m) (-5015259 + 
      64 m (5 + m) (-14847 \\
      & + 8 m (5 + m) (-55 + 8 m (5 + m))))) \big] \alpha_1 + \\
 &	\big [ -(3 + 2 m) (5 + 2 m) (-40305047 + 
   4 m (19 + 4 m) (-1969371 + 
      2 m (19 + 4 m) (-54393 \\
      & + 16 m (19 + 4 m) (-35 + m (19 + 4 m))))) \big] \alpha_2 + \\
 &	\left[-56  \right] \beta_0 + 	
 \left[ -294 \right] \beta_1 +
 	\left[ -181 + 80 m (5 + m) \right] \beta_2 +
 	\left[334 + 148 m (5 + m) \right] \beta_3 +\\
 &
 	\left[ -16 (19 + 3 m (5 + m)) \right] \beta_4 +
 	\left[ -5 (51 + 8 m (5 + m)) \right] \beta_5 +
 	\left[ -2 (53 + 8 m (5 + m)) \right] \beta_6  \\
 	& +
 	\left[28 (2 + m) (3 + m)  \right] \beta_7 +
 	\left[ 592 + 96 m (5 + m) \right] \beta_8 =0,
 \end{align*}

\begin{flushleft}
	\underline{\text{Equation 5}:}
\end{flushleft}
 \begin{align*}
 	& \big[ 7 (5 + 2 m) (7 + 2 m) (963731 + 
   16 m (21 + 4 m) (27348 + 5 m (21 + 4 m) (401 + 8 m (21 + 4 m)))) \big] \alpha_0 + \\
 &	\big[ 7 (3 + 2 m) (7 + 2 m) (-319565 + 
   16 m (5 + m) (65057 + 80 m (5 + m) (335 + 32 m (5 + m)))) \big] \alpha_1 + \\
 &	\big [ 7 (3 + 2 m) (5 + 2 m) (-502109 + 
   16 m (19 + 4 m) (10298 + 5 m (19 + 4 m) (281 + 8 m (19 + 4 m)))) \big] \alpha_2 + \\
 &	  	
 \left[-52  \right] \beta_1 +
 	\left[ -290 \right] \beta_2 +
 	\left[ -278 + 64 m (5 + m)\right] \beta_3 +\\
 &
 	\left[ 5 (7 + 20 m (5 + m)) \right] \beta_4 +
 	\left[ -881 - 140 m (5 + m) \right] \beta_5 +
 	\left[-4 (284 + 45 m (5 + m))  \right] \beta_6 \\
 	&+
 	\left[ -2 (621 + 98 m (5 + m)) \right] \beta_7 +
 	\left[ -6 (179 + 28 m (5 + m)) \right] \beta_8 =0,
 \end{align*}
 
 \begin{flushleft}
	\underline{\text{Equation 6}:}
\end{flushleft}
 \begin{align*}
 	& \big[ -(-5 - 2 m) (7 + 2 m) (-12838089 + 
   16 m (21 + 4 m) (-63803 + 2 m (21 + 4 m) (-227 \\
   &+ 16 m (21 + 4 m)))) \big] \alpha_0 + \\
 &	\big[ (3 + 2 m) (7 + 2 m) (-10661929 + 
   32 m (5 + m) (-125789 + 16 m (5 + m) (-359 + 64 m (5 + m)))) \big] \alpha_1 + \\
 &	\big [ (3 + 2 m) (5 + 2 m) (-7979449 + 
   16 m (19 + 4 m) (-56863 + 2 m (19 + 4 m) (-467 + 16 m (19 + 4 m)))) \big] \alpha_2 + \\
 & 
 	\left[-48 \right] \beta_2 +
 	\left[-282 \right] \beta_3 + 
 	\left[-367 + 48 m (5 + m)  \right] \beta_4 +
 	\left[ -353 + 36 m (5 + m) \right] \beta_5  \\
 	&+
 	\left[-2 (923 + 148 m (5 + m))  \right] \beta_6 
 	+
 	\left[-14 (213 + 34 m (5 + m)) \right] \beta_7 +
 	\left[ -96 (44 + 7 m (5 + m)) \right] \beta_8 =0,
 \end{align*}
 
 \begin{flushleft}
	\underline{\text{Equation 7}:}
\end{flushleft}
  \begin{align*}
 	& \big[ -98 (5 + 2 m) (7 + 2 m) (18097 + 
   16 m (21 + 4 m) (191 + 6 m (21 + 4 m))) \big] \alpha_0 + \\
 &	\big[ -98 (3 + 2 m) (7 + 2 m) (10641 + 32 m (5 + m) (325 + 48 m (5 + m))) \big] \alpha_1 + \\
 &	\big [-98 (3 + 2 m) (5 + 2 m) (5217 + 
   16 m (19 + 4 m) (131 + 6 m (19 + 4 m)))  \big] \alpha_2 + \\
 &	 
 	\left[-44 \right] \beta_3 + 
 	\left[ -270 \right] \beta_4 +
 	\left[ -444 + 32 m (5 + m) \right] \beta_5 +
 	\left[ -818 - 44 m (5 + m) \right] \beta_6 \\
 	&+
 	\left[ -28 (117 + 19 m (5 + m)) \right] \beta_7 +
 	\left[ -42 (149 + 24 m (5 + m)) \right] \beta_8 =0,
 \end{align*}

\begin{flushleft}
	\underline{\text{Equation 8}:}
\end{flushleft}
  \begin{align*}
 	& \big[ -16 (5 + 2 m) (7 + 2 m) (-39241 + 
   8 m (21 + 4 m) (-103 + 6 m (21 + 4 m))) \big] \alpha_0 + \\
 &	\big[ -16 (3 + 2 m) (7 + 2 m) (-38265 + 16 m (5 + m) (-263 + 48 m (5 + m))) \big] \alpha_1 + \\
 &	\big [ -16 (3 + 2 m) (5 + 2 m) (-33921 + 
   8 m (19 + 4 m) (-163 + 6 m (19 + 4 m))) \big] \alpha_2 + \\
 & 
 	\left[-40  \right] \beta_4 +
 	\left[ -254 \right] \beta_5 +
 	\left[-505 + 16 m (5 + m)  \right] \beta_6 +
 	\left[ -28 (48 + 5 m (5 + m)) \right] \beta_7 \\
 	& +
 	\left[ -48 (109 + 18 m (5 + m)) \right] \beta_8 =0,
 \end{align*}
 
 \begin{flushleft}
	\underline{\text{Equation 9}:}
\end{flushleft}
  \begin{align*}
 	& \big[ 672 (5 + 2 m) (7 + 2 m) (179 + 12 m (21 + 4 m)) \big] \alpha_0 + \\
 &	\big[672 (3 + 2 m) (7 + 2 m) (155 + 48 m (5 + m))  \big] \alpha_1 + \\
 &	\big [ 672 (3 + 2 m) (5 + 2 m) (119 + 12 m (19 + 4 m)) \big] \alpha_2 + \\
 & \left[ -36 \right] \beta_5 +
 	\left[ -234 \right] \beta_6 +
 	\left[ -546 \right] \beta_7 +
 	\left[ -21 (91 + 12 m (5 + m)) \right] \beta_8 =0,
 \end{align*}
 
 \begin{flushleft}
	\underline{\text{Equation 10}:}
\end{flushleft}
  \begin{align*}
 	& \big[16 (5 + 2 m) (7 + 2 m) (-421 + 32 m (21 + 4 m))  \big] \alpha_0 + \\
 &	\big[ 16 (3 + 2 m) (7 + 2 m) (-485 + 128 m (5 + m)) \big] \alpha_1 + \\
 &	\big [ 16 (3 + 2 m) (5 + 2 m) (-581 + 32 m (19 + 4 m)) \big] \alpha_2 + \\
 &	 
 	\left[ -32 \right] \beta_6 +
 	\left[ -210 \right] \beta_7 +
 	\left[ -563 - 16 m (5 + m) \right] \beta_8 =0,
 \end{align*}
 
 \begin{flushleft}
	\underline{\text{Equation 11}:}
\end{flushleft}
  \begin{align*}
 	& \big[-2464 (5 + 2 m) (7 + 2 m)  \big] \alpha_0 +
 		\big[-2464 (3 + 2 m) (7 + 2 m)  \big] \alpha_1 + \\
 &	\big [ -2464 (3 + 2 m) (5 + 2 m) \big] \alpha_2 + 	 
 	\left[-28  \right] \beta_7 +
 	\left[ -182 \right] \beta_8 =0,
 \end{align*}
 
 \begin{flushleft}
	\underline{\text{Equation 12}:}
\end{flushleft}
 \begin{align*}
 	& \big[-128 (5 + 2 m) (7 + 2 m)  \big] \alpha_0 + 
 		\big[-128 (3 + 2 m) (7 + 2 m)  \big] \alpha_1 + \\
 &	\big [ -128 (3 + 2 m) (5 + 2 m) \big] \alpha_2 +
 	\left[-24  \right] \beta_8 =0.
 \end{align*}
 We note that, this system has the matrix form
\begin{align*} 
	 \left[\begin{array}{@{}ccc|cccccccccc@{}}
    x_{00} & x_{01} & x_{02} & y_{00} & y_{01}  & y_{02} & y_{03}  & y_{04} & y_{05}  & y_{06}  & y_{07} & y_{08}  \\
   x_{10} & x_{11} & x_{12} & y_{10} & y_{11}  & y_{12} & y_{13}  & y_{14} & y_{15}  & y_{16}  & y_{17} & y_{18} \\ 
   x_{20} & x_{21} & x_{22} & y_{20} & y_{21}  & y_{22} & y_{23}  & y_{24} & y_{25}  & y_{26}  & y_{27} & y_{28} \\
    x_{30} & x_{31} & x_{32} & y_{30} & y_{31}  & y_{32} & y_{33}  & y_{34} & y_{35}  & y_{36}  & y_{37} & y_{38} \\
    x_{40} & x_{41} & x_{42} &0 & y_{41}  & y_{42} & y_{43}  & y_{44} & y_{45}  & y_{46}  & y_{47} & y_{48} \\
    x_{50} & x_{51} & x_{52} & 0 & 0  & y_{52} & y_{53}  & y_{54} & y_{55}  & y_{56}  & y_{57} & y_{58} \\
    x_{60} & x_{61} & x_{62} & 0 & 0 &0 & y_{63}  & y_{64} & y_{65}  & y_{66}  & y_{67} & y_{68} \\
    x_{70} & x_{71} & x_{72} & 0 & 0  & 0 & 0  & y_{74} & y_{75}  & y_{76}  & y_{77} & y_{78} \\
    x_{80} & x_{81} & x_{82} &0 & 0   & 0  & 0   & 0 & y_{85}  & y_{86}  & y_{87} & y_{88} \\
    x_{90} & x_{91} & x_{92} & 0 & 0  &0 & 0  & 0 & 0  & y_{96}  & y_{97} & y_{98} \\
    x_{10,0} & x_{10,1} & x_{10,2} & 0 & 0  &0 & 0  &0 & 0 & 0  & y_{10,7} & y_{10,8} \\
   x_{11,0} & x_{11,1} & x_{11,2} & 0 & 0  &0 & 0  & 0 & 0  &0  &0 & y_{11,8}  
  \end{array}\right]
	\begin{bmatrix}
		\alpha_0  \\
		\alpha_1 \\
		\alpha_2 \\
		\hline 
		\beta_0 \\
		\beta_1 \\
		\beta_2 \\
		\beta_3 \\
		\beta_4 \\
		\beta_5 \\
		\beta_6 \\
		\beta_7 \\
		\beta_8
	\end{bmatrix} =
	\begin{bmatrix}
		0  \\
		0  \\0  \\0  \\0  \\0  \\0  \\0  \\0  \\0  \\0  \\0  
	\end{bmatrix} ,
\end{align*}
 where the coefficients of the matrix are given as above. One can solve this linear system with respect to the unknowns  $\{\alpha_0,\alpha_1,\alpha_2,\beta_0,\dots,\beta_6\}$ so that $\alpha_0$ and $\beta_0$ are free variables. Then, we choose the consistent values
 \begin{align}\label{Choiceofbeta0}
 	\beta_0=0 
 \end{align}
 and
 \begin{align}\label{Definitionalpha0}
 	\alpha_0 = 2 (1 + m)^2 (3 + m) (1 + 2 m) (3 + 2 m)^2 (5 + 2 m) (6 + 2 m) (34 + 
   24 m + 4 m^2).
 \end{align}
 With these choices, one can solve the linear system above, obtaining the solutions
 \begin{align}
 	 \alpha_1 &=  -8 (5 + 2 m)^4 (819 + 4 m (5 + m) (106 + m (5 + m) (18 + m (5 + m)))),\label{Definitionalpha1} \\
 	  \alpha_2 &= 8 (2 + m)^2 (4 + m)^2 (5 + 2 m) (7 + 2 m)^2 (9 + 2 m) (7 + 
   2 m (4 + m)),   \label{Definitionalpha2}
   \end{align}
   as well as
   \begin{align*}
   \beta_1 &=-240 (2 + m) (3 + m) (3 + 2 m) (5 + 2 m)^2 (7 + 2 m) (424785 + 
   2 m (5 + m) (159981 \\
   & + 
      4 m (5 + m) (11215 + 4 m (5 + m) (347 + 16 m (5 + m))))), \\
   \beta_2 &=-32 (2 + m) (3 + m) (3 + 2 m) (5 + 2 m)^2 (7 + 2 m) (5795205 + 
   2 m (5 + m) (2262318 \\
   &+ 
      m (5 + m) (655655 + 8 m (5 + m) (10457 + 496 m (5 + m))))), \\
      \beta_3 & = -80 (2 + m) (3 + m) (3 + 2 m) (5 + 2 m)^2 (7 + 2 m) (937797 + 
   4 m (5 + m) (214330 \\
   &+ 
      m (5 + m) (70297 + 64 m (5 + m) (155 + 8 m (5 + m))))), \\
   \beta_4 &=-32 (2 + m) (3 + m) (3 + 2 m) (5 + 2 m)^2 (7 + 2 m) (-732039 + 
   m (5 + m) (-252965 \\
   &+ 
      8 m (5 + m) (1543 + 8 m (5 + m) (185 + 16 m (5 + m))))), \\
   \beta_5 &= 320 (2 + m) (3 + m) (3 + 2 m) (5 + 2 m)^2 (7 + 2 m) (49665 + 
   4 m (5 + m) (7427 \\
   &+ m (5 + m) (1469 + 96 m (5 + m)))), \\
   \beta_6 &= 128 (2 + m) (3 + m) (3 + 2 m) (5 + 2 m)^2 (7 + 2 m) (3255 + 
   2 m (5 + m) (2443 \\
   &+ 2 m (5 + m) (373 + 32 m (5 + m)))), \\
   \beta_7 & = -5120 (2 + m)^2 (3 + m)^2 (3 + 2 m) (5 + 2 m)^2 (7 + 2 m) (19 + 
   4 m (5 + m)) , \\
   \beta_8 &=  -512 (2 + m)^2 (3 + m)^2 (3 + 2 m) (5 + 2 m)^2 (7 + 2 m) (19 + 
   4 m (5 + m))  .
 \end{align*}
 On the one hand, using \eqref{DefinitiontkStart}, we compute the left-hand-side of the recurrence relation \eqref{RecurrenceRelationF}, that is 
 \begin{align}\label{Defintitionoftk}
 	t_{k} &= \frac{6720}{\pi ^2} \frac{(m+3) (2 m+5) (m+2) }{
 	\prod_{i=2}^{5} (2m-k+i) 
 	\prod_{j=5}^{8} (4m+2k+2j+1) }\frac{(k+1) (4 k+7)}{(2 k+1) (2 k+3)} \cdot \nonumber \\
 	& \cdot \frac{\Gamma \left(k-\frac{1}{2}\right) \Gamma \left(k+\frac{7}{2}\right)}{\Gamma (k+4) \Gamma (k+5)}\frac{\Gamma \left(-k+2 m+\frac{3}{2}\right) \Gamma (k+2 m+5)}{\Gamma (-k+2 m+2) \Gamma \left(k+2 m+\frac{11}{2}\right)}T_m(k) ,
 \end{align}
 where $T_m(k)$ is a polynomial with respect to $k$ of degree $8$,
 \begin{align*}
 	T_m(k) &= \sum_{j=0}^{8} \tilde{T}_m (j) k^j,
 \end{align*}
  and the coefficients are given by
  \begin{align*}
  	\tilde{T}_m (0)  &=  2 (2 + m) (3 + m) (5 + 2 m)^2 (24057 + 
   2 m (5 + m) (8121 + 8 m (5 + m) (223 + 16 m (5 + m))))  ,\\
  	\tilde{T}_m (1)  &= 21 (2535165 + 
   m (5 + m) (2762223 + 
      2 m (5 + m) (588501 + 
         4 m (5 + m) (30749 + 8 m (5 + m)\cdot \\
         & \cdot (395 + 16 m (5 + m))))))  ,\\
  	\tilde{T}_m (2)  &= -36879489 + 
 2 m (5 + m) (-12352033 + 
    m (5 + m) (-2533381 + 
       12 m (5 + m)\cdot \\
         & \cdot  (-3819 + 16 m (5 + m)  (199 + 16 m (5 + m)))))   ,\\
  	\tilde{T}_m (3)  &=  -42 (406041 + 
   2 m (5 + m) (195702 + 
      m (5 + m) (66649 + 4 m (5 + m) (2421 + 128 m (5 + m))))) ,\\
  	\tilde{T}_m (4)  &=  5822802 - 
 m (5 + m) (-2447647 + 
    4 m (5 + m) (-30745 + 48 m (5 + m) (299 + 32 m (5 + m))))  ,\\
  	\tilde{T}_m (5)  &=  504 (4315 + m (5 + m) (2771 + 2 m (5 + m) (291 + 20 m (5 + m)))) ,\\
  	\tilde{T}_m (6)  &=  8 (-13212 + m (5 + m) (1877 + 40 m (5 + m) (53 + 6 m (5 + m)))) ,\\
  	\tilde{T}_m (7)  &=-672 (2 + m) (3 + m) (19 + 4 m (5 + m))   ,\\
  	\tilde{T}_m (8)  &=  -48 (2 + m) (3 + m) (19 + 4 m (5 + m))  .
  \end{align*}
  On the other hand, using \eqref{DefinitionofG} together with the definitions of of $t_k$ from \eqref{Defintitionoftk}, $p_3(k)$ from \eqref{Definitionp3}, $p(k)$ from \eqref{Definitionp} and $b(k)$ from \eqref{Definitionb}, we find $G=G (m,k)$ in the form of \eqref{AnsatzforG}, that is
   \begin{align*} 
 	G (m,k)= R(m,k) \mathsf{F}(m,k),
 \end{align*}
 for some function $R(m,k)$ which we compute explicitly. This has the form
 \begin{align*}
 	R(m,k)=k \tilde{R}(m,k),
 \end{align*}
for some complicated function $ \tilde{R}(m,k)$ which we leave out. From the latter, we get immediately $R(m,0)=0$ and hence\footnote{The fact that $G(m,0)=0$, for all integers $m \geq 1$, also follows from the choice \eqref{Choiceofbeta0}. Indeed, according to \eqref{Choiceofbeta0}, we have that $\beta_0=0$ and hence, based on the definition \eqref{Definitionb}, we get  $b(0)=\beta_0=0$. Then, \eqref{DefinitionofG} yields $G(m,0)=0$. }
 \begin{align*}
 	G(m,0)=0.
 \end{align*}
 Since $K=2m+1$, we also compute $\mathsf{F}_{m} (2m+2)=0$ and hence we get
 \begin{align*}
 	G_{m} (2m+2)=0 .
 \end{align*} 
 Putting these together, we infer
 \begin{align*}
 	G (m,2m+2)-G (m,0)=0.
 \end{align*}
  Consequently, the right-hand-side of \eqref{RecurrenceRelationSmallf} vanishes and the desired recurrence relation for $\mathsf{f}(m) $ boils down to
  \begin{align*}
   \alpha_{0}(m)\mathsf{f}_{m} +\alpha_{1}(m)\mathsf{f}_{m+1}+\alpha_{2}(m)\mathsf{f}_{m+2}=0,
  \end{align*}
  where the coefficients $\{\alpha_{0},\alpha_{1},\alpha_{2}\}$ are given by \eqref{Definitionalpha0}, \eqref{Definitionalpha1} and \eqref{Definitionalpha2} respectively. Finally, we note that these coeffiicents coincide with the corresponding coefficients $\mathsf{P}_m$, $-\mathsf{Q}_m$ and $\mathsf{R}_m$ from Lemma \ref{LemmaRecurrenceRelationKG} for $\delta=3$.

\section{Auxiliary computations for the KG}\label{AuxiliarycomputationsKGAppendix}

  \begin{lemma}\label{Lemma1AuxiliarycomputationsKGAppendix}
	Fix a real number $\delta  \geq 3/2$ and let $\mathsf{R}_m $, $\mathsf{A}_m$ and $\mathsf{B}_m$ be the functions defined in Lemma \ref{LemmaMonotonicityC00mmKG}. Then, 
\begin{align*}
	\mathsf{R}_m > 0, \quad \mathsf{B}_m >0, \quad \mathsf{A}_m -\mathsf{B}_m <1,
\end{align*} 
for all integers $m \geq 1$
\end{lemma}
\begin{proof}
Fix a real number $\delta  \geq 3/2$. Firstly, notice that $\mathsf{R}_m$ is given in a factorized form. From this, it is clear that $\mathsf{R}_m >0$, for all integers $m \geq 1$ and real $\delta \geq 3/2$. Then,
\begin{align*}
	\mathsf{A}_m -\mathsf{B}_m -1 <0 \Longleftrightarrow \mathsf{R}_m (\mathsf{A}_m -\mathsf{B}_m -1) <0 \Longleftrightarrow  \mathsf{Q}_m -\mathsf{P}_m -\mathsf{R}_m <0.
\end{align*}
Now, a long computation shows that  
\begin{align*}
	\mathsf{Q}_m -\mathsf{P}_m -\mathsf{R}_m = \sum_{j=0}^{8} \mathsf{c}_{j}  m^j,
\end{align*}
where the coefficients are given by
\begin{align*}
	\mathsf{c}_0 &= -85 \delta ^7-1532 \delta ^6-6202 \delta ^5-10288 \delta ^4-7401 \delta ^3-2196 \delta ^2-64 \delta +120   ,\\
	\mathsf{c}_1 &=  -108 \delta ^7-3140 \delta ^6-21004 \delta ^5-51892 \delta ^4-54816 \delta ^3-25736 \delta ^2-4760 \delta +176  ,\\
	\mathsf{c}_2 &=  -32 \delta ^7-2092 \delta ^6-24908 \delta ^5-96308 \delta ^4-149484 \delta ^3-101032 \delta ^2-29816 \delta -2600  ,\\
	\mathsf{c}_3 &= -448 \delta ^6-12416 \delta ^5-83392 \delta ^4-198400 \delta ^3-191808 \delta ^2-78784 \delta -11904   ,\\
	\mathsf{c}_4 &= -2208 \delta ^5-33984 \delta ^4-137568 \delta ^3-198144 \delta ^2-111136 \delta -22688   ,\\
	\mathsf{c}_5 &= -5248 \delta ^4-47744 \delta ^3-113664 \delta ^2-89984 \delta -23296   ,\\
	\mathsf{c}_6 &= -6528 \delta ^3-33920 \delta ^2-41728 \delta -13440   ,\\
	\mathsf{c}_7 &=  -4096 \delta ^2-10240 \delta -4096  ,\\
	\mathsf{c}_8 &=  -1024 \delta -512  .
\end{align*}
One can easily see that $\mathsf{c}_j<0$, for all $j \in \{0,1,\dots,8\}$ and all real $\delta \geq 3/2$, that proves the claim.
\end{proof}

\begin{lemma}\label{Lemma2AuxiliarycomputationsKGAppendix}
Fix a real number $\delta  \geq 3/2$ and let $\mathsf{x}_1$ be the initial value defined in Lemma \ref{LemmaMonotonicityC00mmKG}. Then, $\mathsf{x}_1<0$.
\end{lemma}
 \begin{proof}
Fix a real number $\delta  \geq 3/2$ and recall that
\begin{align*}
	\mathsf{x}_1 &  = \frac{63 \delta ^5+196 \delta ^4+111 \delta ^3-78 \delta ^2-52 \delta -24}{80 \delta ^5+448 \delta ^4+492 \delta ^3-136 \delta ^2-236 \delta -48}.
\end{align*}
One can easily show that the denominator is strictly positive for all real $\delta \geq 3/2$. Then, a direct computation shows that $\mathsf{x}_1 <1$ is equivalent to
\begin{align*}
	-17 \delta ^5-252 \delta ^4-381 \delta ^3+58 \delta ^2+184 \delta +24 <0.
\end{align*}
Finally, a straight forward computation verifies the latter and completes the proof.
 \end{proof}

\section{Auxiliary computations for the WM}\label{AuxiliarycomputationsWMAppendix}

  \begin{lemma}\label{Lemma1AuxiliarycomputationsWMAppendix}
Fix a real numbabusier $\delta  >0$ and let $\mathfrak{R}_m $, $\mathfrak{A}$ and $\mathfrak{B}$ be the functions defined in Lemma \ref{LemmaMonotonicityC00mmWM}. Then, 
\begin{align*}
	\mathfrak{R}_m >0, \quad \mathfrak{B}_m >0, \quad \mathfrak{A}_m- \mathfrak{B}_m <1, 
\end{align*}
for all integers $m \geq 1$.
\end{lemma}
\begin{proof}
Fix a real number $\delta  >0$. Firstly, notice that $\mathfrak{R}_m$ is given in a factorized form. From this, it is clear that $\mathfrak{R}_m >0$, for all integers $m \geq 1$ and real $\delta  >0$. Then,
\begin{align*}
	\mathfrak{A}_m -\mathfrak{B}_m -1 <0 \Longleftrightarrow \mathfrak{R}_m (\mathfrak{A}_m -\mathfrak{B}_m -1) <0 \Longleftrightarrow  \mathfrak{Q}_m -\mathfrak{R}_m -\mathfrak{R}_m <0.
\end{align*}
Now, a long computation shows that  
\begin{align*}
	\mathfrak{Q}_m -\mathfrak{R}_m -\mathfrak{R}_m &=
	\frac{-\mathfrak{c}(m)}{(m+2) (\delta +m+2) (\delta +m+3) (\delta +2 m+3) (\delta +2 m+6) }\cdot \\
	&\cdot 
	\frac{1}{(2 \delta +2 m+7) (\delta +2 m (\delta +m+3)+3)} ,
\end{align*}
where $\mathfrak{c}(m)$ is a polynomial with respect to $m$ of degree $6$,
\begin{align*}
	\mathfrak{c}(m)=\sum_{j=0}^{6}\mathfrak{c}_j m^j,
\end{align*}
and its coefficients are given by
\begin{align*}
	\mathsf{c}_0 &= \delta ^6+29 \delta ^5+291 \delta ^4+1379 \delta ^3+3260 \delta ^2+3472 \delta +1072  ,\\
	\mathsf{c}_1 &=  2 \delta ^6+68 \delta ^5+784 \delta ^4+4200 \delta ^3+11054 \delta ^2+12868 \delta +4144   ,\\
	\mathsf{c}_2 &=  26 \delta ^5+544 \delta ^4+4044 \delta ^3+13256 \delta ^2+17998 \delta +6156  ,\\
	\mathsf{c}_3 &=  112 \delta ^4+1576 \delta ^3+7352 \delta ^2+12512 \delta +4608 ,\\
	\mathsf{c}_4 &= 216 \delta ^3+1924 \delta ^2+4620 \delta +1856   ,\\
	\mathsf{c}_5 &= 192 \delta ^2+864 \delta +384   ,\\
	\mathsf{c}_6 &= 64 \delta +32 .
\end{align*}
One can easily see that $\mathfrak{c}_j>0$, for all $j \in \{0,1,\dots,6\}$ and all real $\delta  >0$, that proves the claim.
\end{proof}

  \begin{lemma}\label{Lemma2AuxiliarycomputationsWMAppendix}
Fix a real number $\delta  >0$	and let $\mathfrak{x}_1 $  be the initial value defined in Lemma  \ref{LemmaMonotonicityC00mmWM}. Then, $ \mathfrak{x}_1<1$.
\end{lemma}
\begin{proof}
	Fix a real number $\delta  >0$ and recall that
\begin{align*}
	\mathfrak{x}_1 &  = \frac{3 \delta ^5+38 \delta ^4+189 \delta ^3+482 \delta ^2+708 \delta +560}{4 \delta ^5+62 \delta ^4+364 \delta ^3+1038 \delta ^2+1472 \delta +840}.
\end{align*}
One can easily show that the denominator is strictly positive for all real $\delta >0$. Then, a direct computation shows that $\mathfrak{x}_1 <1$ is equivalent to
\begin{align*} 
	-\delta ^5-24 \delta ^4-175 \delta ^3-556 \delta ^2-764 \delta -280<0.
\end{align*}
Finally, a straight forward computation verifies the latter and completes the proof.
\end{proof}

  \bibliographystyle{plain}
\bibliography{V5_NekhoroshevStability_KG_WM}

\end{document}